\documentclass{article}

\usepackage{natbib}
\usepackage{fullpage}
\usepackage[utf8]{inputenc} 
\usepackage{hyperref}       
\usepackage{url}            
\usepackage{booktabs}       
\usepackage{amsfonts}       
\usepackage{nicefrac}       
\usepackage{microtype}      

\newcommand{\mytextstyle}{}                   

\usepackage{algorithm}
\usepackage{algpseudocode}
\usepackage{mathtools}
\mathtoolsset{showonlyrefs}
\mathtoolsset{showmanualtags}
\usepackage{amsthm}
\usepackage{xcolor}
\usepackage{tabularx}
\usepackage{bbm}
\usepackage{todonotes}
\usepackage{amsmath}
\usepackage{makecell}
\usepackage{subcaption}

\DeclarePairedDelimiter\floor{\lfloor}{\rfloor}

\allowdisplaybreaks[1]

\newtheorem{theorem}{Theorem}
\newtheorem{lemma}{Lemma}
\newtheorem{corollary}{Corollary}
\newtheorem{remark}{Remark}

\newcommand{\range}{\mathrm{range}}
\newcommand{\Comp}{\mathrm{Comp}}
\newcommand{\Comm}{\mathrm{Comm}}
\newcommand{\Span}{\mathrm{Span}}
\newcommand{\Sp}{\mathrm{Sp}}

\newcommand{\bg}{\mathrm{D}}

\newcommand{\R}{\mathbb{R}}

\def\<#1,#2>{\langle #1,#2\rangle}

\newcommand{\Dotprod}[1]{\left\langle#1\right\rangle}

\newcommand{\norm}[1]{\|#1\|}
\newcommand{\sqn}[1]{\norm{#1}^2}
\newcommand{\Norm}[1]{\left\|#1\right\|}
\newcommand{\sqN}[1]{\Norm{#1}^2}
\newcommand{\vect}[1]{\begin{bmatrix*}[l]#1\end{bmatrix*}}

\newcommand{\cG}{\mathcal{G}}
\newcommand{\cV}{\mathcal{V}}
\newcommand{\cE}{\mathcal{E}}

\newcommand{\sX}{{\mathsf E}}
\newcommand{\cO}{\mathcal{O}}

\newcommand{\mW}{\mathbf{W}}

\newcommand{\mI}{\mathbf{I}}
\newcommand{\mP}{\mathbf{P}}
\newcommand{\mM}{\mathbf{M}}
\newcommand{\mWp}{\mathbf{W}^{\dagger}}
\newcommand{\sqnw}[1]{\sqn{#1}_{\mWp}}
\newcommand{\eqdef}{\coloneqq}
\DeclareMathOperator*{\argmin}{arg\,min}

\graphicspath{{./plot/}} 

\title{\bf Optimal and Practical Algorithms for Smooth and Strongly Convex Decentralized Optimization}

\author{Dmitry Kovalev \qquad Adil Salim \qquad  Peter Richt\'{a}rik   \\ \phantom{xx} \\
	King Abdullah University of Science and Technology\\ Thuwal, Saudi Arabia
}

\begin{document}
	
	\maketitle

	\begin{abstract}
We consider the task of decentralized minimization of the sum of smooth strongly convex functions  stored across the nodes of a network. For this problem, lower bounds on the number of gradient computations and the number of communication rounds required to achieve $\varepsilon$ accuracy have recently been proven. We propose two new algorithms for this decentralized optimization problem and equip them with complexity guarantees. We show that our first method is optimal both in terms of the number of communication rounds and in terms of the number of gradient computations. Unlike existing optimal algorithms, our algorithm does not rely on the expensive evaluation of dual gradients. Our second algorithm is optimal in terms of the number of communication rounds, without a logarithmic factor. Our approach relies on viewing the two proposed algorithms as accelerated variants of the Forward Backward algorithm to solve monotone inclusions associated with the decentralized optimization problem. We also verify the efficacy of our methods against state-of-the-art algorithms through numerical experiments.
	\end{abstract}

	\section{Introduction}

In this paper we are concerned with the design and analysis of new efficient algorithms for solving optimization problems in a {\em decentralized} storage and computation regime. In this regime, a network of agents/devices/workers, such as mobile devices, hospitals, wireless sensors, or smart home appliances, collaborates to solve a single optimization problem whose description is stored across the nodes of the network. Each node can perform computations using its local state and data, and is only allowed to communicate with its neighbors.

Problems of this form have been traditionally studied in the signal processing community~\citep{xu2020distributed}, but are attracting increasing interest from the machine learning and optimization community as well \citep{scaman2017optimal}. Indeed, the training of supervised machine learning models via empirical risk minimization from training data stored across a network is most naturally cast as a  decentralized optimization problem. Finally, while current federated learning  \citep{FEDLEARN, FL2017-AISTATS} systems rely on a star network topology, with a trusted server performing aggregation and coordination placed at the center of the network, advances in decentralized optimization could be useful in new generation federated learning formulations that would rely on fully decentralized computation  \citep{FL_survey_2019}. In summary, decentralized optimization is of direct relevance to machine learning, present and future.

\subsection{Formalism}

Formally, given an undirected connected  network $\cG = (\cV, \cE)$ with nodes/vertices $\cV = \{1,\ldots,n\}$ and edges $\cE \subset \cV \times \cV$, we consider optimization  problems of the form
	\begin{equation}\label{eq:1}
	\min\limits_{x \in  \R^d} \sum_{i \in \cV} f_i(x),
	\end{equation}
	where the data describing functions $f_i:\R^d\to \R$ is stored on node $i$ and not directly available to any other node. Decentralized algorithms for solving this problem need to respect the network structure of the problem, which is to say that computation can only be made on the nodes $i\in \cV$ from data and information available on the nodes, and communication is constrained to only happen along the edges $e\in \cE$. 	

\subsection{Computation and communication}	Several decentralized gradient-type algorithms have been proposed to solve \eqref{eq:1} in the  smooth and strongly convex regime. Two key efficiency measures used to compare such methods are: i) the {\em number of gradient evaluations} (where one gradient evaluation refers to computing $\nabla f_i(x_i)$ for all $i\in \cV$ for some input vectors $x_i$), and ii) the {\em number of communication rounds}, where one  round allows each node to send $\cO(1)$ vectors of size $d$ to their neighbors. If computation is costly, the first comparison metric is more important, and if communication is costly, the second is more important.

Note that problem \eqref{eq:1} poses certain intrinsic difficulties each method designed for it needs to address. Clearly, more information can be communicated in each communication round  if the network $\cG$ is ``more highly'' connected. By $\chi$ we  denote the {\em condition number} associated with (the connectivity of) the graph $\cG$; a formal definition is given later. Likewise, more computation will be needed if the functions $f_i$ are ``more complicated''. We will entirely focus on problems where all functions $f_i$ are  $L$-smooth and $\mu$-strongly convex, which naturally leads to the quantity $\kappa \eqdef L/\mu$ as a condition number associated with computation.

	Much of decentralized optimization research is focused on designing decentralized algorithms with computation and communication guarantees which have as good as possible dependence on  the intrinsic properties of the problem, i.e., on the condition numbers $\kappa$ and $\chi$.

\section{Related Work and Contributions}

In this section we first briefly review some of the key results on decentralized	optimization, and subsequently provide a brief summary of our key contributions.

\subsection{Related work}	
	Existing gradient-type decentralized methods for solving  problem \eqref{eq:1} can be informally classified into three classes: {\em non-accelerated} algorithms, {\em accelerated} algorithm and {\em optimal} algorithms.  

\paragraph{Non-accelerated methods.}  Loosely speaking,  a method is non-accelerated if it has at least a linear dependence on the condition numbers $\kappa$ and $\chi$, i.e., $\cO(\kappa)$ and $\cO(\chi)$. Please refer to \citep[Table~1]{xu2020distributed} for a summary of many such methods, see also~\citep{alghunaim2019decentralized,li2020revisiting}.  \citet{xu2020distributed} provide a tight unified analysis of many of these nonaccelerated algorithms, and relies on similar tools as those used in this paper, such as operator splitting and Chebyshev acceleration. 

\paragraph{Accelerated methods.} Accelerated methods have an improved (sublinear) dependence on the condition numbers, typically $\cO(\sqrt{\kappa})$ and $\cO(\sqrt{\chi})$. Accelerated algorithms include accelerated DNGD of \citet{qu2019accelerated} and accelerated EXTRA of \citet{li2020revisiting}; the latter using the  Catalyst~\citep{lin2017catalyst} framework to accelerate EXTRA~\citep{shi2015extra}. Additional accelerated methods include, the Accelerated Penalty Method of~\citet{li2018sharp,dvinskikh2019decentralized}, SSDA and MSDA of \citet{scaman2017optimal} and Accelerated Dual Ascent of~\citet{uribe2020dual}.
		\paragraph{Optimal algorithms.} \citet{scaman2017optimal} provide {\em lower bounds} for  the gradient computation and communication complexities of finding an $\varepsilon$-accurate solution; see Section~\ref{sec:lower-bound} below. 
		There have been several attempts to match these lower bounds, which include algorithms summarized in Table~\ref{tab:relat-works}.
	 Note, that  gradient computation complexity is left as N/A 	for SSDA and MSDA. This is because they rely on the computation of the gradient of the Fenchel conjugate of $f_i$, called {\em dual gradients} in the sequel, which can be intractable. Indeed, computing a dual gradient can be as hard as minimizing $f_i$.
		Finally, we remark that \citet{scaman2018optimal} provide lower bounds in the   nonsmooth regime as well, and an algorithm matching this lower bound is called MSPD. MSPD is primal dual~\cite{chambolle2011first}, similarly to the algorithms developed in this paper.
		
		\renewcommand{\arraystretch}{2}
		\begin{table}[H]
			\caption{ Comparison of existing state of the art decentralized algorithms with our results in terms of  gradient computation and communication complexity of finding  $x$ such that $\|x-x^*\|^2 \leq \varepsilon$, where $x^*$ is a solution to Problem~\eqref{eq:1}\\}
			\label{tab:relat-works}
			\centering
			\begin{tabular}{|c|c|c|}
				\hline
				\bf Algorithm & \makecell{\bf Gradient computation \\ \bf complexity} & \makecell{\bf Communication\\ \bf complexity}\\
				\hline
				\multicolumn{3}{|c|}{\bf Existing State of the art Decentralized Algorithms}\\
				\hline
				\makecell{Accelerated Dual Ascent\\ \citet{uribe2020dual}}  & $\cO\left(\kappa \sqrt{\chi} \log^2\frac{1}{\varepsilon}\right)$ & $\cO\left(\sqrt{\kappa\chi } \log\frac{1}{\varepsilon}\right)$\\
				\hline
				\makecell{Single/Multi Step Dual Ascent \\ \citet{scaman2017optimal}}&
				N/A&
				$\cO\left(\sqrt{\kappa\chi}\log\frac{1}{\varepsilon}\right)$\\
				\hline
				\makecell{Accelerated Penalty Method\\ \citet{li2018sharp,dvinskikh2019decentralized}}&
				$\cO\left(\sqrt{\kappa} \log\frac{1}{\varepsilon}\right)$&
				$\cO\left(\sqrt{\kappa \chi} \log^2\frac{1}{\varepsilon}\right)$\\
				\hline
				\makecell{Accelerated EXTRA\\\citet{li2020revisiting}}&
				$\cO\left(\sqrt{\kappa\chi}\log( \kappa\chi )\log\frac{1}{\varepsilon}\right)$&
				$\cO\left(\sqrt{\kappa\chi}\log( \kappa\chi )\log\frac{1}{\varepsilon}\right)$\\
				\hline
				\multicolumn{3}{|c|}{\bf Our Results}\\
				\hline 
				\makecell{Algorithm~\ref{alg:ALV}\\this paper, Theorem~\ref{th:ALV}}&
				$\cO\left(\left(\sqrt{\kappa\chi} + \chi\right)\log\frac{1}{\varepsilon}\right)$&
				$\cO\left(\left(\sqrt{\kappa\chi} + \chi\right)\log\frac{1}{\varepsilon}\right)$\\
				\hline
				\makecell{\textbf{Algorithm~\ref{alg:ALV-opt}}\\\textbf{this paper, Corollary~\ref{th:ALV-opt-cor}}}&
				$\cO\left(\sqrt{\kappa}\log\frac{1}{\varepsilon}\right)$&
				$\cO\left(\sqrt{\kappa\chi}\log\frac{1}{\varepsilon}\right)$\\
				\hline
				\makecell{Algorithm~\ref{alg:scary}\\this paper, Appendix}&
				$\cO\left(\sqrt{\kappa\chi}\log\frac{1}{\varepsilon}\right)$&
				$\cO\left(\sqrt{\kappa\chi}\log\frac{1}{\varepsilon}\right)$\\
				\hline
				\hline
				{\bf Lower bounds~\citet{scaman2017optimal}}&
				$\cO\left(\sqrt{\kappa}\log\frac{1}{\varepsilon}\right)$&
				$\cO\left(\sqrt{\kappa\chi}\log\frac{1}{\varepsilon}\right)$\\
				\hline
			\end{tabular}
		\end{table}

	\subsection{Summary of contributions}
	
The starting point of this paper is the realization that, {\em to the best of our knowledge, 	in the class of algorithms not relying on the computation of the dual gradients, there is no algorithm optimal in communication complexity, and as a result, no algorithm optimal in both gradient computation and communication complexity.} To remedy this situation, we do the following:

\begin{itemize}
\item We propose a new accelerated decentralized algorithm not relying on dual gradients: Accelerated Proximal Alternating Predictor-Corrector (APAPC) method (Algorithm~\ref{alg:ALV}).
We show that  in order to obtain $x$ for which $\sqn{x - x^*} \leq \varepsilon$,
where 	$x^*$ is the solution of \eqref{eq:1}, this method only needs  $$\cO((\sqrt{\kappa \chi} + \chi)\log(1/\varepsilon))$$ gradient computations and communication rounds (Theorem~\ref{th:ALV}). When combined with Chebyshev acceleration, similarly to the trick used in~\citep[Section 4.2]{scaman2017optimal}, we show that our method, which we then call
Optimal Proximal Alternating Predictor-Corrector (OPAPC) method
(Algorithm~\ref{alg:ALV-opt}),  leads to an {\em optimal} decentralized method both in terms of gradient computation and communication complexity (Corollary~\ref{th:ALV-opt-cor}). In particular, OPAPC finds an $\varepsilon$-solution in at most
$$\cO\left(\sqrt{\kappa}\log(1/\varepsilon)\right)$$
gradient computations and at most
$$\cO\left(\sqrt{\kappa \chi}\log(1/\varepsilon)\right)$$
communication rounds.
Algorithm~\ref{alg:ALV-opt} reaches  the lower bounds (Theorem~\ref{th:lower}), and hence it is indeed optimal.



\item  We also propose another accelerated algorithm (Algorithm~3) not relying on dual gradients, one that is optimal in communication complexity (this algorithm is presented in the appendix only). Compared to the above development, this algorithm has the added advantage that it requires the computation of a single gradient per communication step. This can have practical benefits  when communication is expensive.      

\end{itemize}


	\section{Background}

	\label{sec:background}

	
	\subsection{Basic formulation of the decentralized problem}
	Problem~\eqref{eq:1} can be reformulated as a lifted (from $\R^d$ to $\R^{dn}$) optimization problem via consensus constraints:
	\begin{equation}
	\label{eq:pb-equiv}
	\min\limits_{\substack{x_1,\ldots,x_n \in  \R^d\\x_1=\ldots=x_n}}  \sum_{i \in \cV} f_i(x_i).
	\end{equation}
	Consider the function $F : (\R^{d})^\cV \to \R$ defined by $F(x_1,\dots,x_n) = \sum_{i \in \cV} f_i(x_i)$, where $x_1,\dots,x_n \in \R^{d}$. Then, $F$ is $\mu$-strongly convex and $L$-smooth since the individual functions $f_i$ are. Consider also any linear operator (equivalently, any matrix) $\mW : (\R^{d})^\cV \to (\R^{d})^\cV$ such that $\mW(x_1,\dots,x_n) = 0$ if and only if $x_1=\ldots=x_n$. Denoting $x = (x_1,\ldots,x_n) \in (\R^{d})^\cV$, Problem~\eqref{eq:pb-equiv} is equivalent to 
	\begin{equation}
	\label{eq:2}
	\min\limits_{\substack{x \in \ker(\mW)}}  F(x).
	\end{equation}
	Many optimization algorithms converge exponentially fast (i.e., linearly) to a minimizer of Problem~\eqref{eq:2}, e.g. the projected gradient algorithm. However, only few of them are decentralized. A decentralized algorithm typically relies on multiplication by $\mW$, in cases where $\mW$ is a \textit{gossip matrix}. Consider a $n \times n$ matrix $\hat{\mW}$ satisfying the following properties: 1) $\hat{\mW}$ is symmetric and positive semi definite, 2) $\hat{\mW}_{i,j} \neq 0$ if and only if $i=j$ or $(i,j) \in \cE$, and 3) $\ker \hat{\mW} = \mathrm{span}(\mathbbm{1})$, where $\mathbbm{1} = (1,\ldots,1)^\top$. 	Such a matrix is called a {\em gossip matrix}. A typical example is the Laplacian of the graph $\cG$.	Denoting $\mI$ the $d \times d$ identity matrix and $\otimes$ the Kronecker product, consider $\mW : (\R^{d})^\cV \to (\R^{d})^\cV$ the $nd \times nd$ matrix defined by
$
	\mW \eqdef \hat{\mW}\otimes \mI.
$
	This matrix can be represented as a block matrix $\mW = (\mW_{i,j})_{(i,j)\in \cV^2}$, where each block $\mW_{i,j} = \hat{\mW}_{i,j} \mI$ is a $d \times d$ matrix proportional to $\mI$. In particular, if $d=1$, then $\mW = \hat{\mW}$. Moreover, $\mW$ satisfy similar properties to $\hat{\mW}$:

	
	\begin{enumerate}
		\item $\mW$ is symmetric and positive semi definite,
		\item $\mW_{i,j} \neq 0$ if and only if $i=j$ or $(i,j) \in \cE$,
	   	\item $\ker \mW$ is the consensus space, $\ker(\mW) = \{(x_1,\dots,x_n) \in (\R^{d})^\cV, x_1=\dots=x_n\}$, 
		\item $\lambda_{\max}(\mW) = \lambda_{\max}(\hat{\mW})$ and  $\lambda_{\min}^{+}(\mW) = \lambda_{\min}^{+}(\hat{\mW})$, where $\lambda_{\max}$ (resp. $\lambda_{\min}^{+}$) denotes the largest (resp. the smallest positive) eigenvalue.
	\end{enumerate}
	
		Throughout the paper, we denote $\mW^{\dagger} : \range(\mW) \to \range(\mW)$ the inverse of the map $\mW : \range(\mW) \to \range(\mW)$. The operator $\mW^{\dagger}$ is positive definite over $\range(\mW)$ and we denote $\sqn{y}_{\mW^{\dagger}} = \Dotprod{\mW^{\dagger}y,y}$ for every $y \in \range(\mW)$.	With a slight abuse of language, we shall say that $\mW$ is a gossip matrix. 
 Note that decentralized communication can be represented as a multiplication of $\mW$ by a vector $x \in (\R^d)^\cV$. Indeed, the $i^{\text{th}}$ component of $\mW x$ is a linear combination of $x_j$, where $j$ is a neighbor of $i$ (we shall write $j \sim i$). In other words, one matrix vector multiplication involving $\mW$ is equivalent to one communication round.

	\textit{In the rest of the paper, our goal is to solve the equivalent problem~\eqref{eq:2} with $\mW$ being a gossip matrix via an optimization algorithm which uses only evaluations of $\nabla F$ and multiplications by $\mW$.}


	\subsection{Lower bounds}
\label{sec:lower-bound}
Linearly converging decentralized algorithms using a gossip matrix $\mW$ often have a linear rate depending on the condition number of the $f_i$, $\kappa \eqdef \frac{L}{\mu}$ and the condition number (or spectral gap) of $\mW$, $\chi(\mW) \eqdef \frac{\lambda_{\max}(\mW)}{\lambda_{\min}^{+}(\mW)}$. Indeed, the spectral gap of the Laplacian matrix is known to be a measure of the connectivity of the graph.

In this paper, we define {\em the class of (first order) decentralized algorithms} as the subset of black box optimization procedure~\cite[Section 3.1]{scaman2017optimal} not using dual gradients, i.e. a decentralized algorithm is not allowed to compute $\nabla f_i^*$ (a formal definition is given in the Supplementary material). Complexity lower bounds for solving Problem~\eqref{eq:1} by a black-box optimization procedure are given by~\cite{scaman2017optimal}. These lower bounds relate the number of gradient computations (resp.\ number of communication rounds) to achieve $\varepsilon$ accuracy to the condition numbers $\kappa$ and $\chi(\mW)$. Since a decentralized algorithm is a black-box optimization procedure, these lower bounds apply to decentralized algorithms. Therefore, we obtain our first result as a direct application of~\cite[Corollary 2]{scaman2017optimal}.

\begin{theorem}[\cite{scaman2017optimal}]
\label{th:lower}
    Let $\chi \geq 1$. There exist a gossip matrix $\mW$ with condition number $\chi$, and a family of smooth strongly convex functions $(f_i)_{i \in \cV}$ with condition number $\kappa > 0$ such that the following holds: for any $\varepsilon >0$, any decentralized algorithm requires at least $\Omega\left(\sqrt{\kappa \chi}\log(1/\varepsilon)\right)$ communication rounds, and at least $\Omega\left(\sqrt{\kappa}\log(1/\varepsilon)\right)$ gradient computations to output $x = (x_1, \ldots, x_n)$ such that $\sqn{x - x^*} < \varepsilon$, where $x^* = \argmin F.$ 
\end{theorem}


Although the lower bounds of Theorem~\ref{th:lower} are obvious consequences of~\citep[Corollary~2]{scaman2017optimal}, their tightness is not. Indeed, the lower bounds of Theorem~\ref{th:lower} are tight on the class of black-box optimization procedures since they are achieved by MSDA~\cite{scaman2017optimal}. However, MSDA uses dual gradients and whether these lower bounds are tight on the class of decentralized algorithms is not known. In this paper, we propose decentralized algorithms achieving these lower bounds, showing in particular that they are tight.


\subsection{Operator splitting}
\label{sec:op}
Recall that in this paper, any optimization algorithm solving Problem~\eqref{eq:2} by using  evaluations of $\nabla F$ and multiplications by the gossip matrix $\mW$ only is a decentralized algorithm. Such algorithms can be obtained in several ways, e.g., by applying operator splitting methods to primal dual reformulations of Problem~\eqref{eq:2}, see~\cite{con19}. This is the approach we chose in this work.

	We now provide some minimal background knowledge on the Forward Backward algorithm involving monotone operators. We restrict ourselves to single valued, continuous monotone operators. For the general case of set valued monotone operators, the reader is referred to~\cite{bau-com-livre11}. 
	
	Let $\sX$ be an Euclidean space and denote $\Dotprod{\cdot,\cdot}_{\sX}, \|\cdot\|_{\sX}$ its inner product and the associated norm. Given $\nu \in \R$, a map $A : \sX \to \sX$ is {\em $\nu$-monotone} if for every $x,y \in \sX$, $$\Dotprod{A(x)-A(y),x-y}_{\sX} \geq \nu \|x-y\|_{\sX}^2.$$ If $\nu < 0$, $A$ is  {\em weakly monotone}, if $\nu > 0$, $A$ is {\em strongly monotone} and if $\nu = 0$ then $A$ is {\em monotone}. In this paper, a monotone operator is defined as a monotone continuous map. For every monotone operator and every $\gamma >0$, the map $I + \gamma A : \sX \to \sX$ is one-to-one and its inverse $J_{\gamma A} = (I + \gamma A)^{-1} : \sX \to \sX$, called {\em resolvent}, is well defined. Let $F$ be a {\em smooth} convex function, i.e., $F$ is differentiable and its gradient is Lipschitz continuous. Then $\nabla F$ is a monotone operator, and the resolvent $J_{\gamma \nabla F}$ is the proximity operator of $\gamma F$. However, there exist monotone operators which are not gradients of convex functions. For instance, a skew symmetric operator $S$ on $\sX$ defines the linear map $x \mapsto S x$ which is not a gradient. This map is a monotone operator since $\Dotprod{Sx,x}_{\sX} = 0$. The {\em set of zeros} of $A$, defined as $Z(A) \eqdef \{x \in \sX, A(x) = 0\}$, is often of interest in optimization. For instance, $Z(\nabla F) = \argmin F$.

	\paragraph{Forward Backward.} In order to find an element in $Z(A+B)$, where $B$ is another monotone operator, the {\em Forward Backward algorithm} iterates
	\begin{equation}
		\label{eq:FB}
		x^{k+1} = J_{B}(x^k - A(x^k)).
	\end{equation}      
	Note that if $A = \nabla F$ and $B = \nabla G$, where $G$ is another differentiable convex function, the Forward Backward algorithm boils down to the proximal gradient algorithm. In this particular case, Nesterov acceleration can be applied to~\eqref{eq:FB} and leads to faster convergence rates compared to the proximal gradient algorithm~\citep{nesterov1983method,beck2009fast}. 
	
	\paragraph{Generalized Forward Backward.} 
	For every positive definite operator $\mP$ on $\sX$, the algorithm 
	\begin{equation}
		\label{eq:FBP}
		x^{k+1} = J_{\mP^{-1} B}(x^k - \mP^{-1} A(x^k)),
	\end{equation}      
	called the {\em Generalized Forward Backward} method, can be seen as an instance of~\eqref{eq:FB} because $Z(\mP^{-1}A + \mP^{-1}B) = Z(A+B)$ and $\mP^{-1}A$, $\mP^{-1}B$ are monotone operators under the inner product induced by $\mP$ on $\sX$. For example, the gradient of $F$ under this inner product is $\mP^{-1} \nabla F$. A primal dual optimization algorithm is an algorithm solving a primal dual formulation of a minimization problem, see below. Many primal dual algorithms can be seen as instances of~\eqref{eq:FBP}, with general monotone operators $A,B$, for a well chosen parameter $\mP$, see~\citep{con19}.

\section{New Decentralized Algorithms}
\label{sec:algo}
\subsection{An accelerated primal dual algorithm}	\label{sec:ALV}

	Before presenting our  algorithm, we introduce an accelerated decentralized algorithm which we then use to motivate the development of our method.

	In this section, $\sX$ is the Euclidean space $\sX = (\R^{d})^{\cV} \times  \range(\mW)$ endowed with the norm $\sqn{(x,y)}_{\sX} \eqdef \sqn{x} + \sqn{y}_{\mW^{\dagger}}$.
	
	Using the first order optimality conditions, a point $x^*$ is a solution to Problem~\eqref{eq:2} if and only if $\nabla F(x^*) \in \range(\mW)$ and $x^* \in \ker(\mW)$. Solving Problem~\eqref{eq:2} is thus equivalent to finding $(x^*, y^*) \in \sX$ such that
	\begin{equation} 	\label{eq:primaldualoptimal}
	\begin{array}{r@{}l}
		0 &{}= \nabla F(x^*) + y^*,\\
		0 &{}= \mW x^*.
	\end{array}
	\end{equation}
	Indeed, the first line of \eqref{eq:primaldualoptimal} is equivalent to $\nabla F(x^*)  = - y^* \in \range \mW$, because $(x^*, y^*) \in \sX = (\R^{d})^{\cV} \times  \range(\mW)$.
		The second line of \eqref{eq:primaldualoptimal} is just a definition of $x^* \in \ker \mW$.
	Consider the maps $M,A,B : \sX \to \sX$ defined by
	\begin{equation}
	M(x,y) \eqdef \begin{bmatrix} \nabla F(x) + y \\ -\mW x  \end{bmatrix}, \quad A(x,y) \eqdef \begin{bmatrix*}[c]
		\nabla F(x)\\
		0
	\end{bmatrix*}, \quad
	B(x,y) \eqdef \begin{bmatrix*}[c]
	y\\
	-\mW x
	\end{bmatrix*}.
	\end{equation} 
	Then $M,A$ and $B$ are monotone operators. Indeed, $A$ is the gradient of the convex function $(x,y) \mapsto F(x)$, $B$ satisfies $$\Dotprod{B(x,y),(x,y)}_{\sX} = \Dotprod{x - \mW^{\dagger}\mW x, y} = 0$$ for every $(x,y) \in \sX$ (since $y \in \range(\mW)$), and $M = A+B$. Moreover, $M(x^*,y^*) = 0$, i.e., $(x^*,y^*)$ is a zero of $M$. 
	
	One idea to solve~\eqref{eq:primaldualoptimal} is therefore to apply Algorithm~\eqref{eq:FB} to the sum $A+B$. However, computing the resolvent $J_{B}$ in a decentralized way across the network $\cG$ is notably challenging. 	Another idea is to apply~\eqref{eq:FBP} using the symmetric positive definite operator $P : \sX \to \sX$ defined by
	\begin{equation}
		\label{eq:Ppapc}
		\mP = \begin{bmatrix*}[c]
		\frac{1}{\eta}\mI & 0\\
		0 & \frac{1}{\theta}\mI - \eta \mW
		\end{bmatrix*}.
	\end{equation}
	Indeed, for every $(x,y) \in \sX$, $(x',y') = J_{P^{-1}B}(x,y)$ implies $x' = x - \eta y'$ and $\frac{1}{\theta}(y' - y) - \eta \mW (y' - y) = \mW x' = \mW (x - \eta y')$. Therefore, $y' = y + \theta \mW(x - \eta y)$, and the computation of $J_{P^{-1}B}$ only requires one multiplication by $\mW$, i.e., one local communication round. The resulting algorithm is 
	\begin{equation} \label{eq:papc}
			\begin{array}{r@{}l}
		y^{k+1} &{} = y^k + \theta \mW(x^k - \eta \nabla F(x^k) - \eta y^k), \\
		x^{k+1} &{} = x^k - \eta \nabla F(x^k) - \eta y^{k+1}.
			\end{array}
	\end{equation}

\begin{remark} The Proximal Alternating Predictor–Corrector (PAPC) algorithm, a.k.a.\ Loris–Verhoven~\citep{loris2011generalization,drori2015simple,chen2013primal,con19} is a primal dual algorithm that can tackle Problem~\eqref{eq:2}. Up to a change of variable, Algorithm~\eqref{eq:papc} can be shown to be equivalent to PAPC applied to \eqref{eq:2}. Moreover, it was already noticed that the PAPC can be represented as a Forward Backward algorithm~\eqref{eq:FBP}~\citep{con19}.
	\end{remark}

Invoking a complexity result on the PAPC from~\cite{salim2020dualize}, the complexity of Algorithm~\eqref{eq:papc} is $\cO\left((\kappa+\chi(\mW))\log(1/\varepsilon)\right),$	 both in communication and gradient computations. This complexity is equivalent to that of the best performing non accelerated algorithm proposed recently, such as Exact diffusion, NIDS and EXTRA (see~\cite{li2020revisiting,xu2020distributed}). In spite of this, we are able to accelerate the convergence of Algorithm~\eqref{eq:papc}.
	
	In particular, we propose a new algorithm that can be seen as an accelerated version of Algorithm~\eqref{eq:papc}. The proposed algorithm (APAPC) is defined in Algorithm~\ref{alg:ALV}) , and its complexity is given in Theorem~\ref{th:ALV}. We prove that the complexity of APAPC is $\cO((\sqrt{\kappa \chi(\mW)} + \chi(\mW))\log(1/\varepsilon)),$ both in communication rounds and gradient computations. The proposed algorithm is accelerated because its dependence on the condition number $\kappa$ is $\cO(\sqrt{\kappa})$ instead of $\cO(\kappa)$.

	\begin{algorithm}[H]
		\caption{Accelerated PAPC (APAPC)}
		\label{alg:ALV}
		\begin{algorithmic}[1]
			\State {\bf Parameters:}  $x^0 \in \R^{nd},  y^0 \in  \range \mW$, $\eta,\theta ,\alpha>0, \tau\in (0,1)$
			\State Set $x_f^0 = x^0$
			\For{$k=0,1,2,\ldots$}{}
			\State $x_g^k = \tau x^k + (1-\tau)x_f^k$\label{alg:ALV:line:x:1}
			\State $x^{k+1/2} = (1+\eta\alpha)^{-1}(x^k - \eta (\nabla F(x_g^k) - \alpha  x_g^k + y^k))$\label{alg:ALV:line:x:2}
			\State $y^{k+1} = y^k  + \theta  \mW x^{k+1/2}$\label{alg:ALV:line:y}
			\State $x^{k+1} = (1+\eta\alpha)^{-1}(x^k - \eta (\nabla F(x_g^k) - \alpha  x_g^k + y^{k+1}))$\label{alg:ALV:line:x:3}
			\State $x_f^{k+1} = x_g^k + \tfrac{2\tau}{2-\tau}(x^{k+1} - x^k)$\label{alg:ALV:line:x:4}			
			\EndFor
		\end{algorithmic}
	\end{algorithm}

\newpage

	\begin{theorem}[Accelerated PAPC]
		\label{th:ALV}
		Set the parameters $\eta, \theta, \alpha, \tau$ to
		$
		\mytextstyle	\eta = \frac{1} {4\tau L}$, $\theta = \frac{1}{\eta\lambda_{\max}(\mW)}$, $\alpha = \mu$, and $\tau = \min \left\{1,\frac{1}{2}\sqrt{\frac{\chi(\mW)}{\kappa} }\right\}.
		$
		Then,
		\begin{equation}
		\mytextstyle	
		\frac{1}{\eta}\sqN{x^k - x^*} +
			\frac{2(1-\tau)}{\tau}\bg_F(x_f^k, x^*)
			\leq
			\left(1 + \frac{1}{4}\min\left\{\frac{1}{\sqrt{\kappa\chi(\mW)}},\frac{1}{\chi(\mW)}\right\}\right)^{-k}
			C,
		\end{equation}
		where $\bg_F$ is the Bregman divergence of $F$ and
		$C \eqdef \frac{1}{\eta}\sqN{x^0 - x^*} + \frac{1}{\theta}\sqn{y^0 - y^*}_{\mW^\dagger} +
		\frac{2(1-\tau)}{\tau}\bg_F(x_f^0, x^*).
	$
	Moreover, for every $\varepsilon >0$, APAPC finds $x^k$ for which $\sqn{x^k - x^*} \leq \varepsilon$ in at most $\cO\left(\left(\sqrt{\kappa \chi(\mW)} + \chi(\mW)\right)\log(1/\varepsilon)\right)$ computations (resp.\ communication rounds). 
	\end{theorem}
	The proposed algorithm~\ref{alg:ALV} provably accelerates Algorithm~\eqref{eq:papc}.
	The proof intuitively relies on viewing Algorithm~\ref{alg:ALV} as an accelerated version of~\eqref{eq:FBP}, although Nesterov's acceleration does not apply to general monotone operators \textit{a priori}.

\subsection{A decentralized algorithm optimal both in communication and computation complexity}
As mentioned before, while APAPC is accelerated, it is not optimal. We now derive a  variant which is optimal both in gradient computations and communication rounds. Following \citet[Section 4.2]{scaman2017optimal}, our main tool to derive the new decentralized optimal algorithm is the Chebyshev acceleration \citep{scaman2017optimal,arioli2014chebyshev}.

In particular, there exists a polynomial $P$ such that
\begin{itemize}
\item [(i)]  $P(\mW)$ is a Gossip matrix, 
\item [(ii)]  multiplication by $P(\mW)$ requires $\floor*{\sqrt{\chi(\mW)}}$ multiplications by $\mW$ (i.e., communication rounds) and is described by the subroutine \textsc{AcceleratedGossip} proposed in~\cite[Algorithm 2]{scaman2017optimal} and recalled in Algorithm~\ref{alg:ALV-opt} for the ease of reading,
\item [(iii)] $\chi(P(\mW)) \leq 4$.
\end{itemize}

		Therefore, one can replace $\mW$ by $P(\mW)$ in Problem~\eqref{eq:2} to obtain an equivalent problem. Applying
		APAPC to the equivalent problem leads to a linearly converging decentralized algorithm. This new algorithm, called Optimal PAPC (OPAPC), is formalized as Algorithm~\ref{alg:ALV-opt}. 
	\begin{algorithm}[H]
		\caption{Optimal PAPC (OPAPC)}
		\label{alg:ALV-opt}
		\begin{algorithmic}[1]
			\State {\bf Parameters:}  $x^0 \in \R^{nd},  y^0 \in  \range P(\mW)$, $T \in \mathbb N^{*}$, $c_1,c_2,c_3,\eta,\theta ,\alpha>0, \tau \in (0,1)$
			\State Set $x_f^0 = x^0$
			\For{$k=0,1,2,\ldots$}{}
			\State $x_g^k = \tau x^k + (1-\tau)x_f^k$\label{alg:ALV-opt:line:x:1}
			\State $x^{k+1/2} = (1+\eta\alpha)^{-1}(x^k - \eta (\nabla F(x_g^k) - \alpha  x_g^k + y^k))$\label{alg:ALV-opt:line:x:2}
			\State $y^{k+1} = y^k  + \theta  \textsc{AcceleratedGossip}(\mW,x^{k+1/2},T)$\label{alg:ALV-opt:line:y}
			\State $x^{k+1} = (1+\eta\alpha)^{-1}(x^k - \eta (\nabla F(x_g^k) - \alpha  x_g^k + y^{k+1}))$\label{alg:ALV-opt:line:x:3}
			\State $x_f^{k+1} = x_g^k + \tfrac{2\tau}{2-\tau}(x^{k+1} - x^k)$\label{alg:ALV-opt:line:x:4}
			\EndFor
		\Procedure{AcceleratedGossip}{$\mW,x,T$}
		\State Set $a_0 = 1$, $a_1 = c_2$, $x_0 = x$, $x_1 = c_2 (I-c_3 \mW)x$
		\For{$i=1,\ldots,T-1$}{}
			\State $a_{i+1} = 2c_2 a_i - a_{i-1}$
			\State $x_{i+1} = 2c_2 (I - c_3 \mW)x_i - x_{i-1}$
		\EndFor

		\Return $x - \frac{x_T}{a_T}$
		\EndProcedure
		\end{algorithmic}
	\end{algorithm}

Using the properties of $P(\mW)$ mentioned above, we obtain the following corollary of Theorem~\ref{th:ALV}.	
	\begin{corollary}[Optimal PAPC]
		\label{th:ALV-opt-cor}
		Set the parameters $T,c_1,c_2,c_3,\eta, \theta, \alpha, \tau$ to
\[ T = \floor*{\sqrt{\chi(\mW)}},\quad
			c_1 = \frac{\sqrt{\chi(\mW)} - 1}{\sqrt{\chi(\mW)} + 1},\quad 
			c_2 = \frac{\chi(\mW)+1}{\chi(\mW)-1},\quad c_3 = \frac{2\chi(\mW)}{(1+\chi(\mW))\lambda_{\max}(\mW)}, \]
\[\eta = \frac{1} {4\tau L},
			\quad \theta = \frac{1+c_1^{2T}}{\eta (1+c_1^T)^2},\quad \alpha = \mu,\quad
		 \tau = \min \left\{1, \frac{1+c_1^T}{2\sqrt{\kappa} (1-c_1^T)}\right\}.
\]
		Then, there exists $C \geq 0$ such that
		\begin{equation}
			\mytextstyle \frac{1}{\eta}\sqN{x^k - x^*} +
			\frac{2(1-\tau)}{\tau}\bg_F(x_f^k, x^*)
			\leq
			\left(1 + \frac{1}{16}\min\left\{\frac{2}{\sqrt{\kappa}},1\right\}\right)^{-k}
			C.
		\end{equation}
		Moreover, for every $\varepsilon >0$, OPAPC finds $x^k$ for which $\sqn{x^k - x^*} \leq \varepsilon$ in at most
	 $\cO\left(\sqrt{\kappa}\log(1/\varepsilon)\right)$  gradient computations and at most 
		$\cO\left(\sqrt{\kappa \chi(\mW)}\log(1/\varepsilon)\right)$ communication rounds.
	\end{corollary}

	The Algorithm~\ref{alg:ALV-opt} achieves both the lower bounds of Theorem~\ref{th:lower}. In particular, the lower bounds of Theorem~\ref{th:lower} are tight.

\section{Numerical Experiments}
\label{sec:exp}

In this section, we perform experiments with logistic regression for binary classification with $\ell^2$ regularizer, where our loss function has the form
$$\mytextstyle f_i(x) = \frac{1}{m}\sum \limits_{j=1}^m \log(1 + \exp(- b_{ij} a_{ij}^\top x)) + \frac{r}{2}\sqn{x},$$
where $a_{ij}\in \R^d$, $b_{ij}\in \{-1,+1\}$ are data points, $r$ is the regularization parameter, $m$ is the number of data points stored on each node. 

In our experiments we used $10,000$ data samples randomly distributed to the nodes of network of size $n=100$, $m=100$ samples per each node. We used 2 networks: $10\times10$ grid and Erd\"{o}s-R\'{e}nyi random graph of average degree 6. Same setup was tested by \citet{scaman2017optimal}.

\subsection{Experiments with LIBSVM data}
We use three LIBSVM\footnote{The LIBSVM dataset collection is available at \texttt{\url{https://www.csie.ntu.edu.tw/~cjlin/libsvmtools/datasets/}}} datasets: a6a, w6a, ijcnn1. The regularization parameter was chosen so that $\kappa \approx 10^3$. 
Additional experiments with synthetic data are given in the Supplementary material.

\begin{figure}[ht]
	\begin{subfigure}{.5\textwidth}
		\includegraphics[width=.49\textwidth]{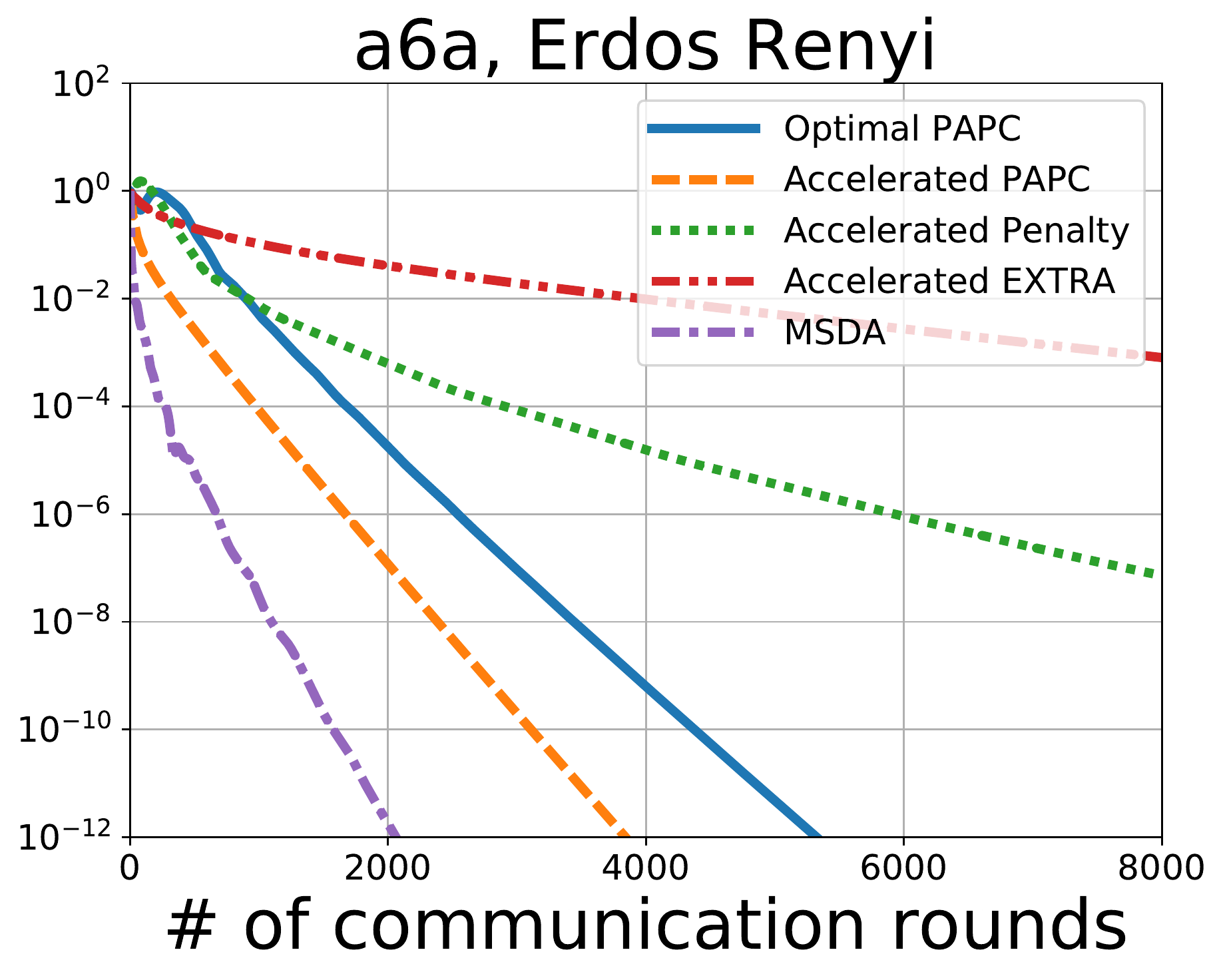}
		\includegraphics[width=.49\textwidth]{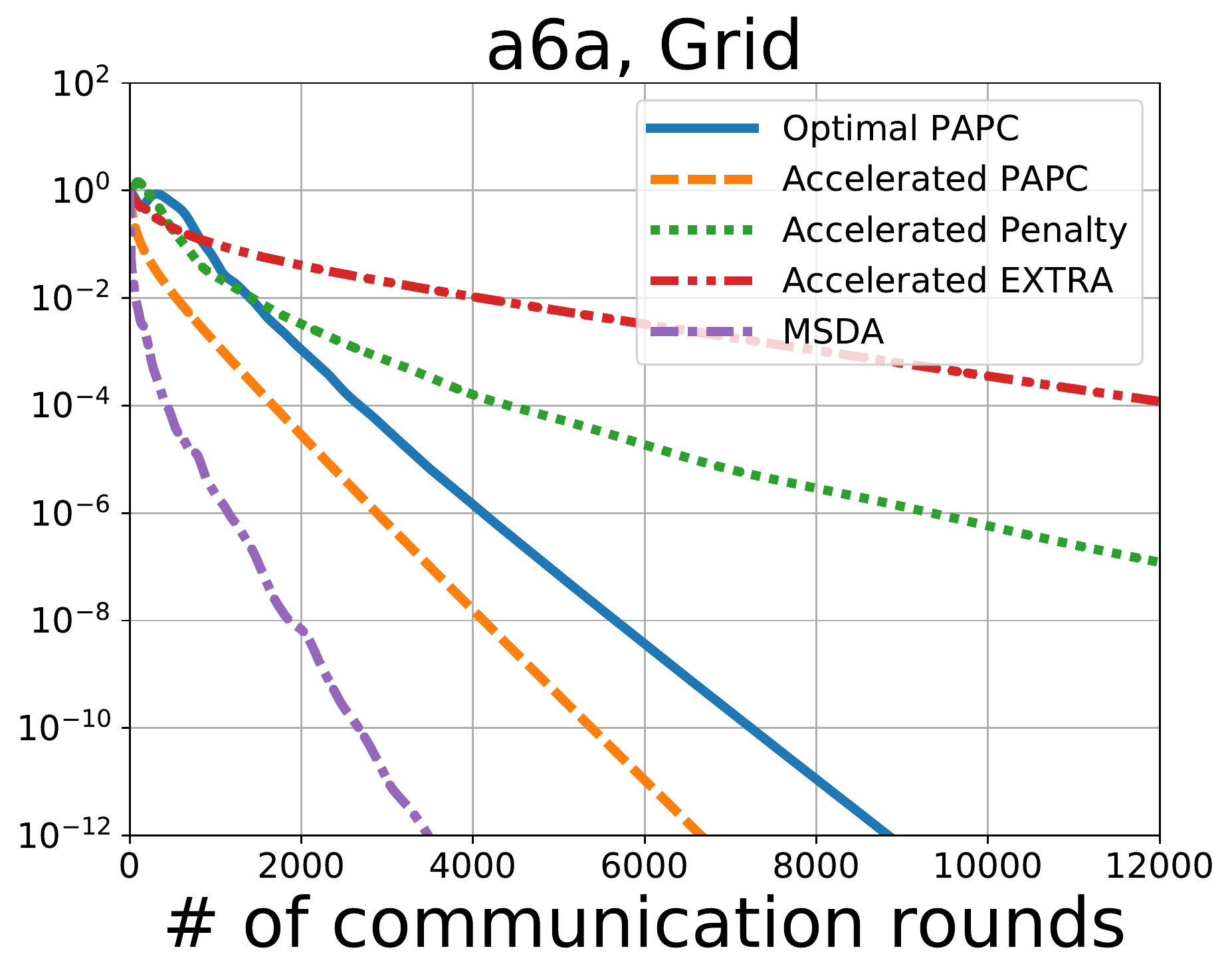}
		\includegraphics[width=.49\textwidth]{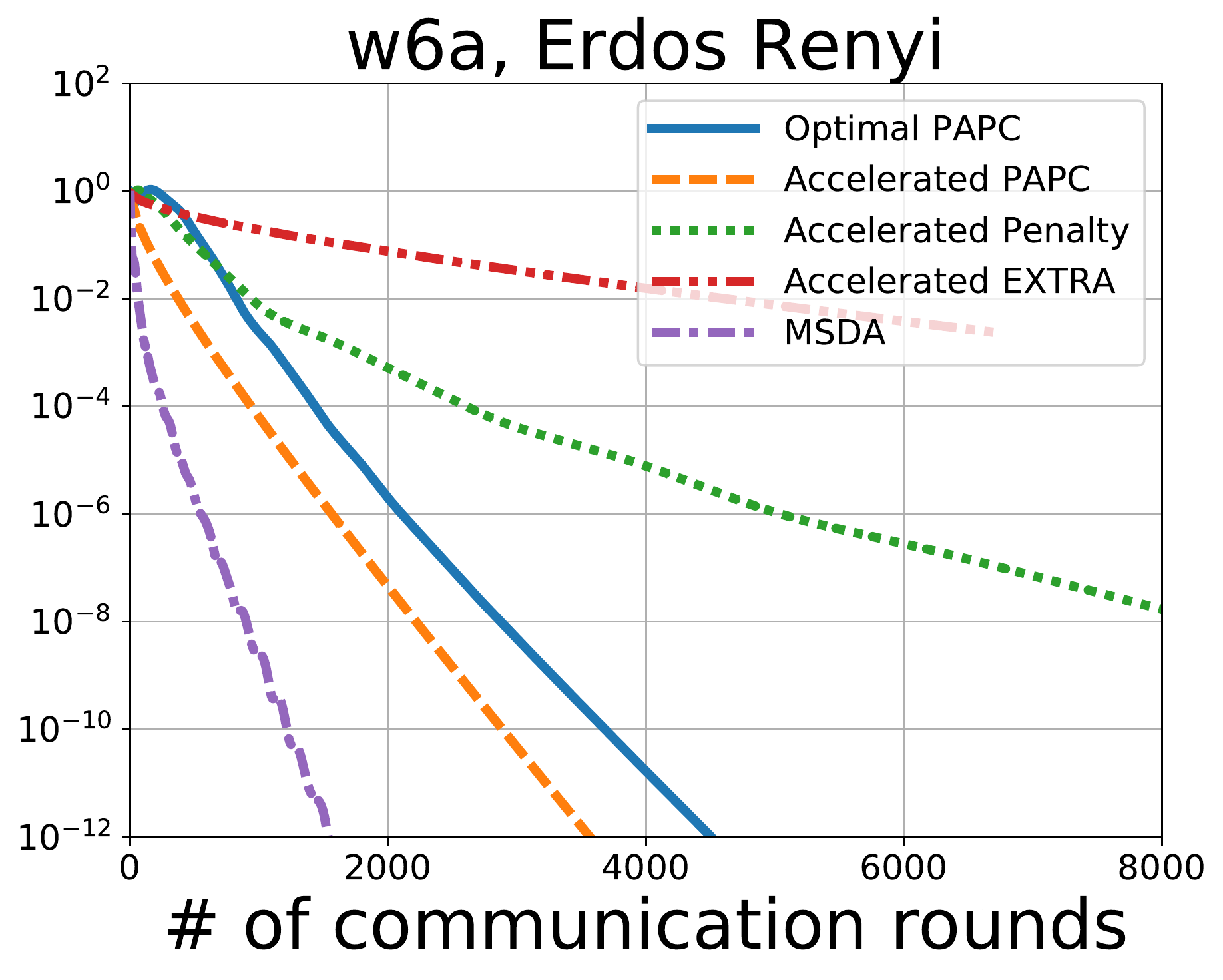}
		\includegraphics[width=.49\textwidth]{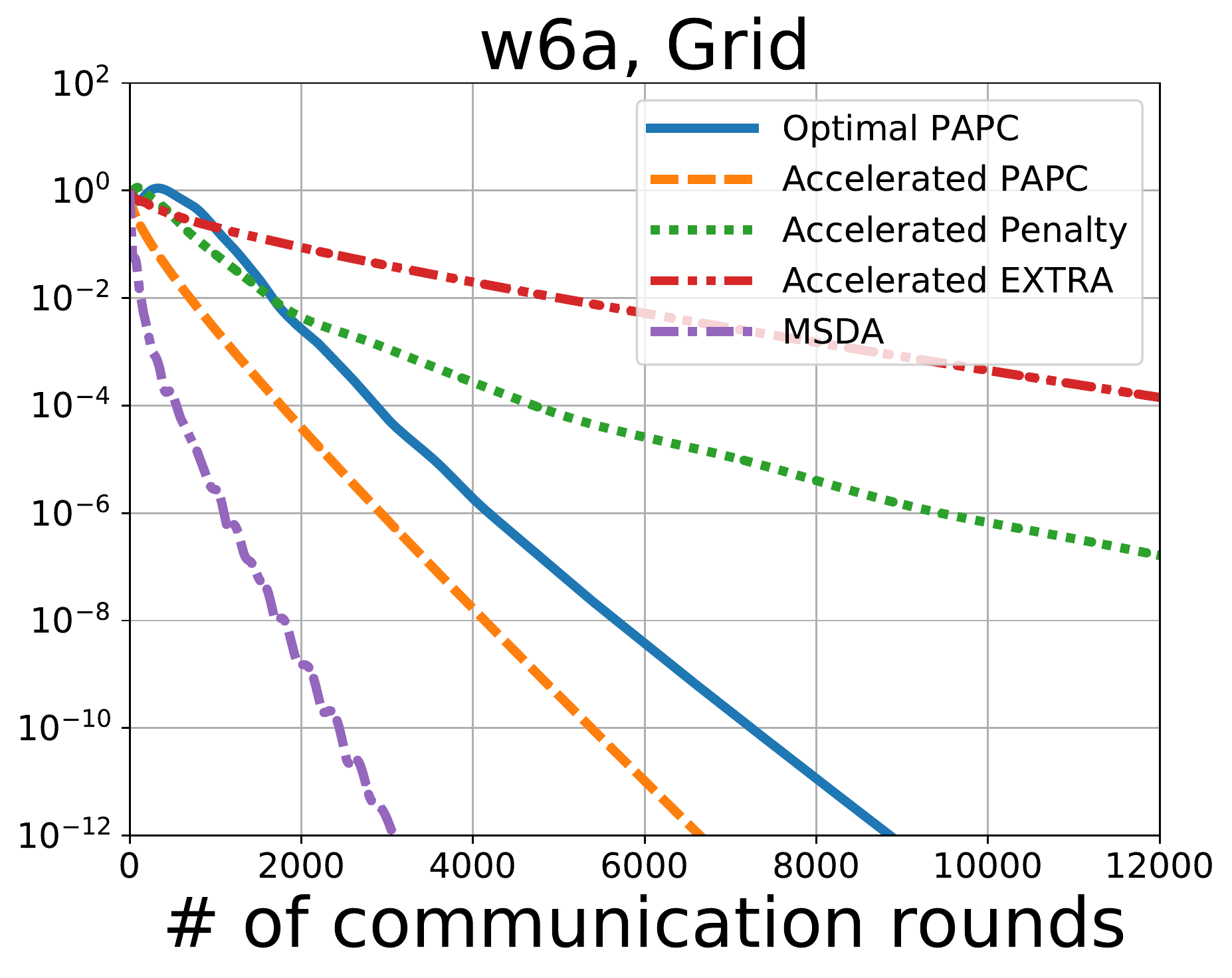}
		\includegraphics[width=.49\textwidth]{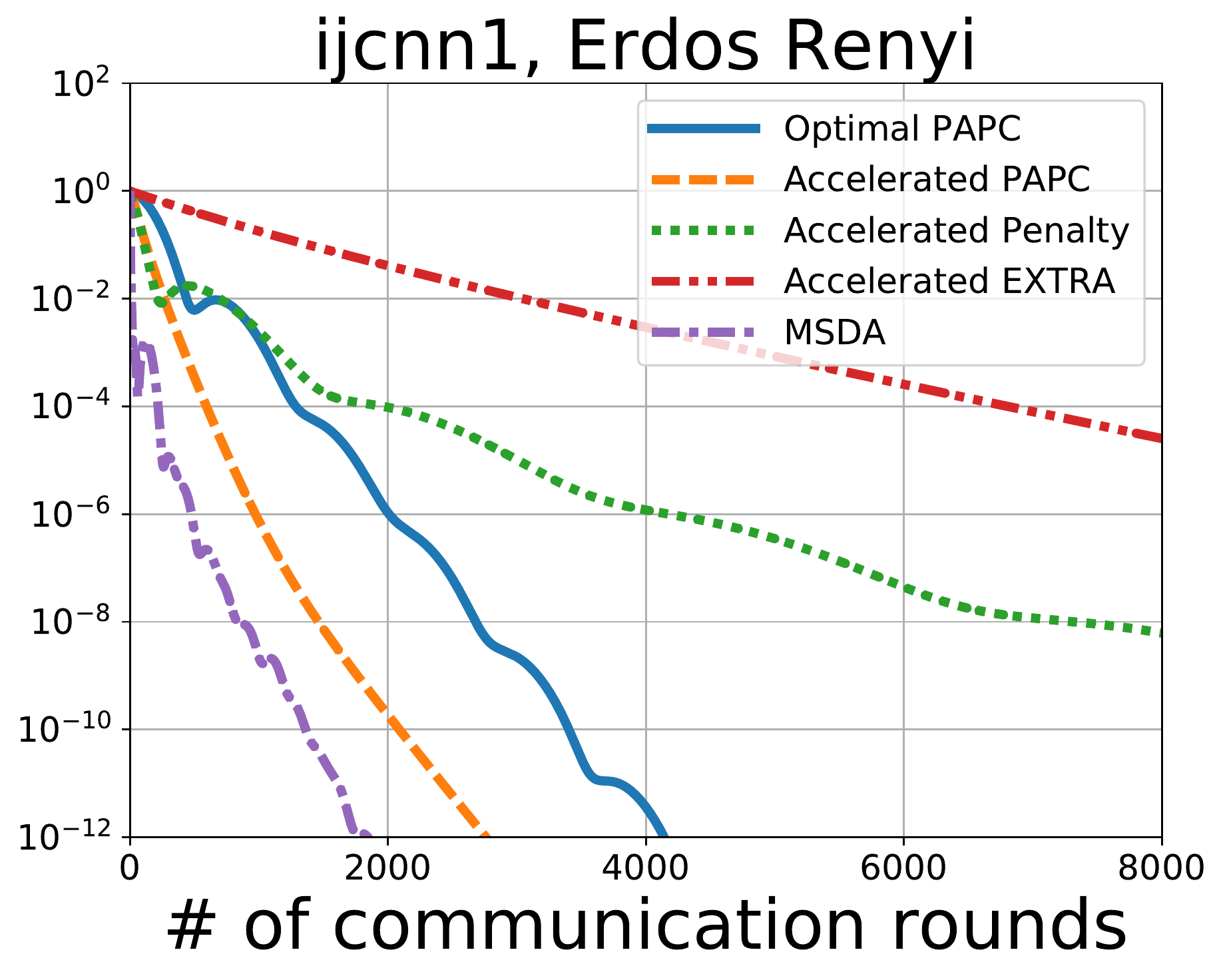}
		\includegraphics[width=.49\textwidth]{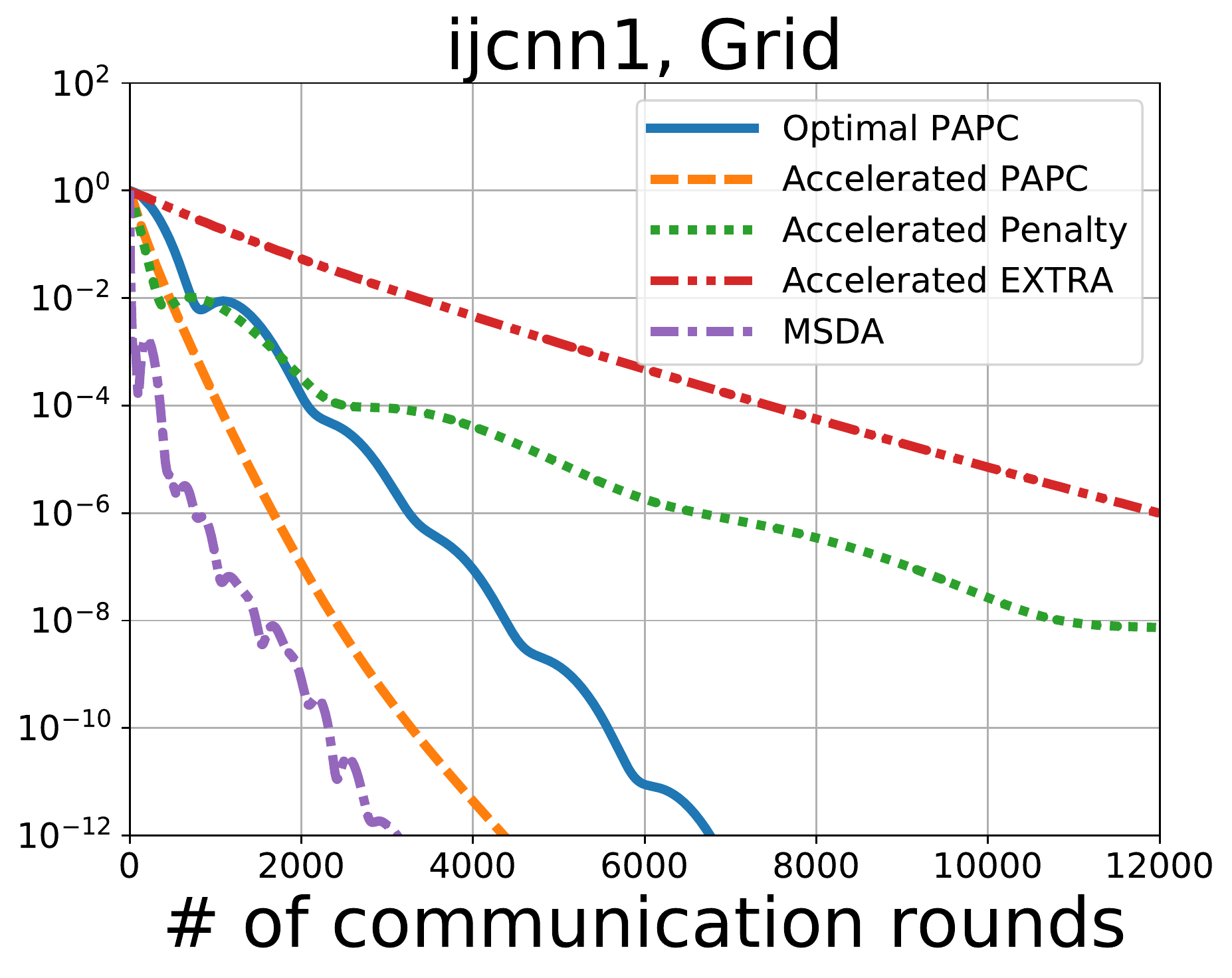}
		\caption{Communication complexity.}
		\label{fig:comm}
	\end{subfigure}
	\begin{subfigure}{.5\textwidth}
		\includegraphics[width=.49\textwidth]{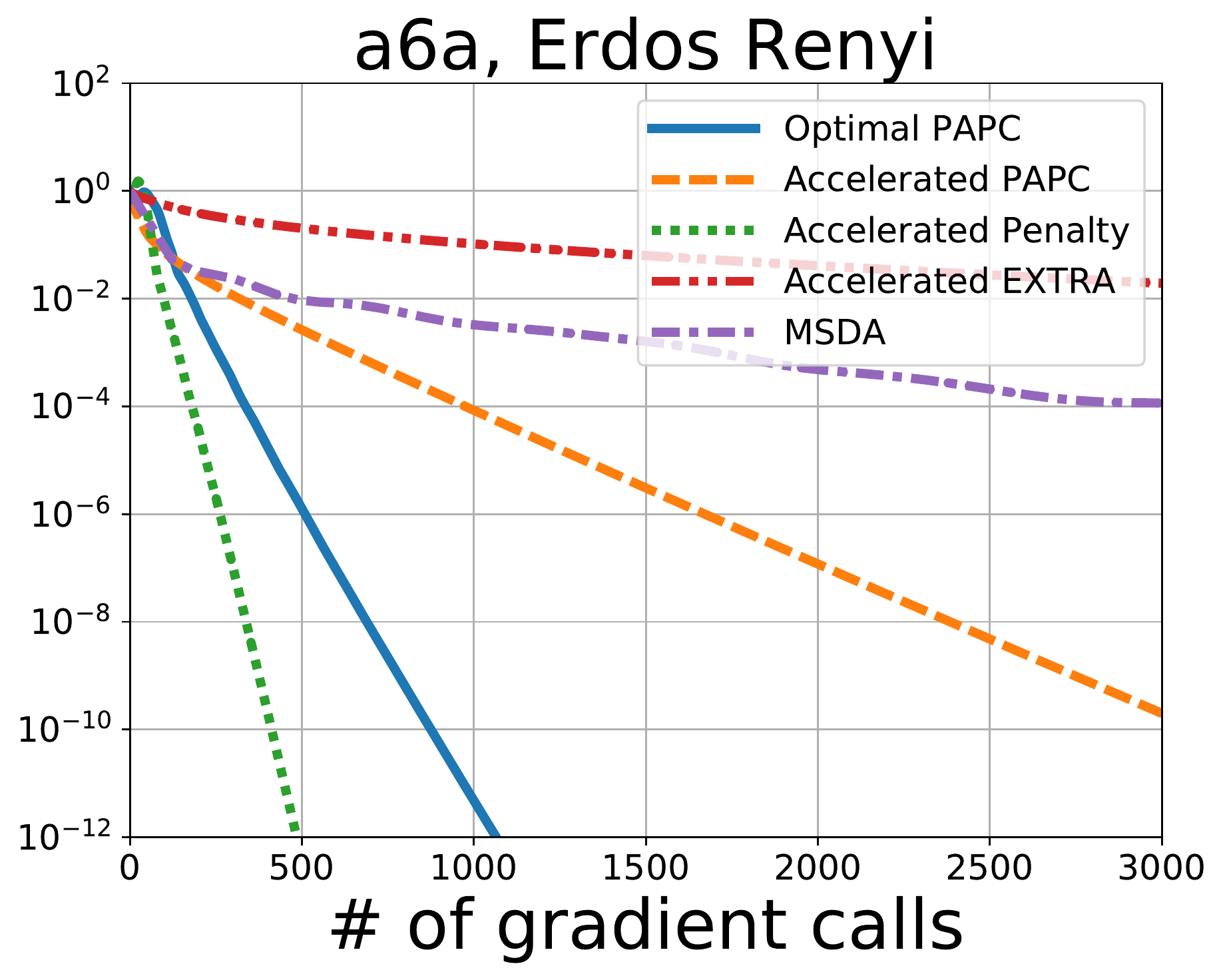}
		\includegraphics[width=.49\textwidth]{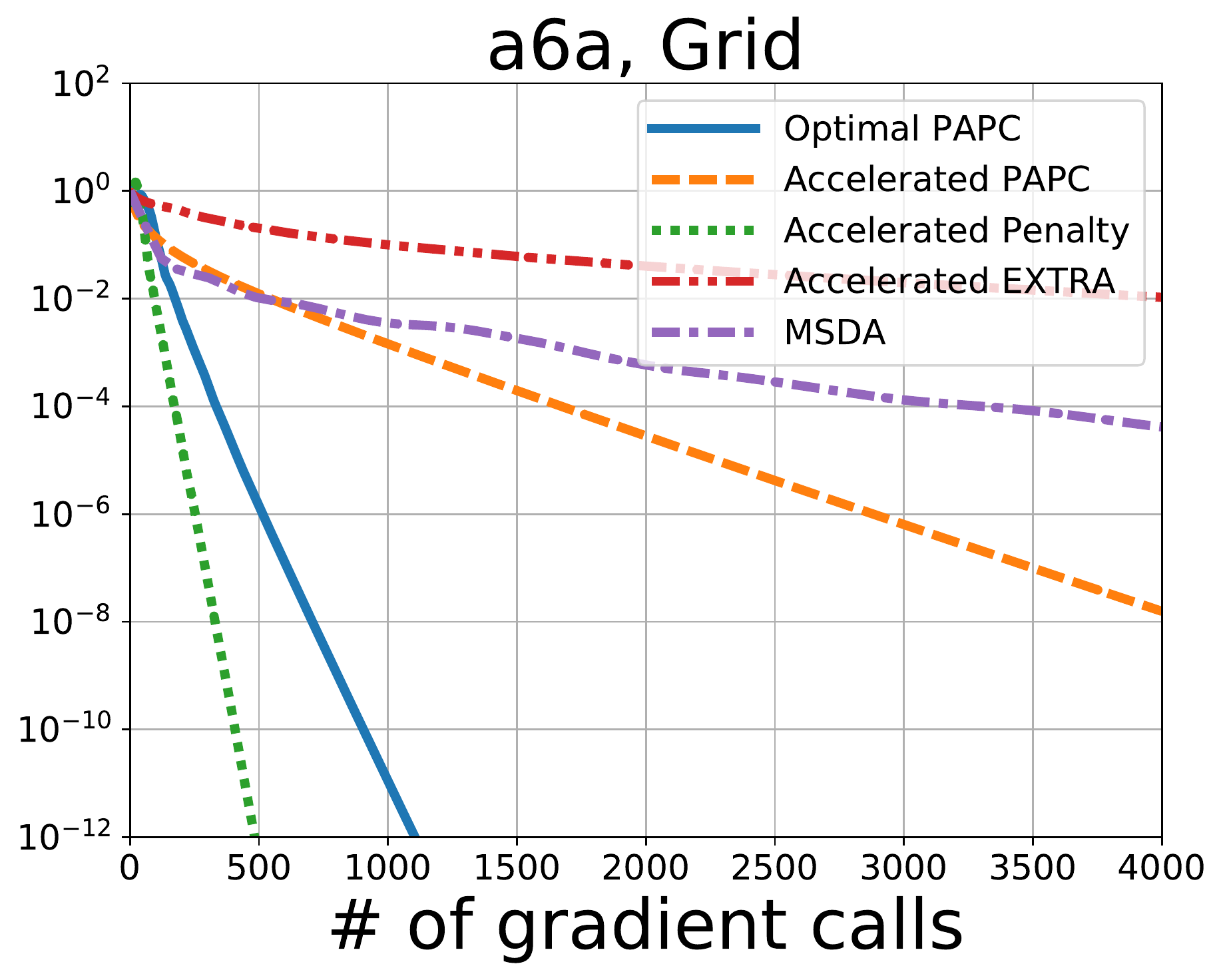}
		\includegraphics[width=.49\textwidth]{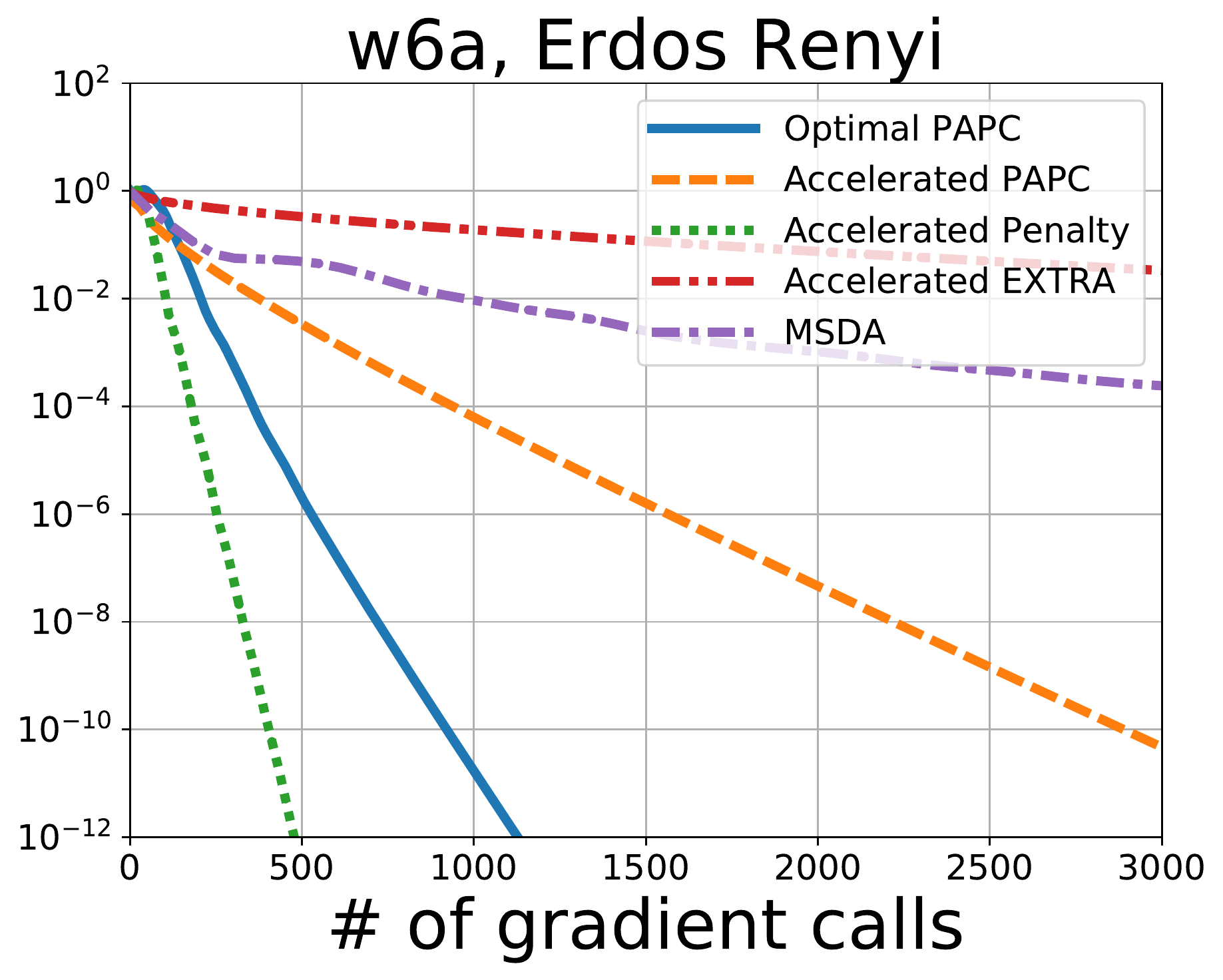}
		\includegraphics[width=.49\textwidth]{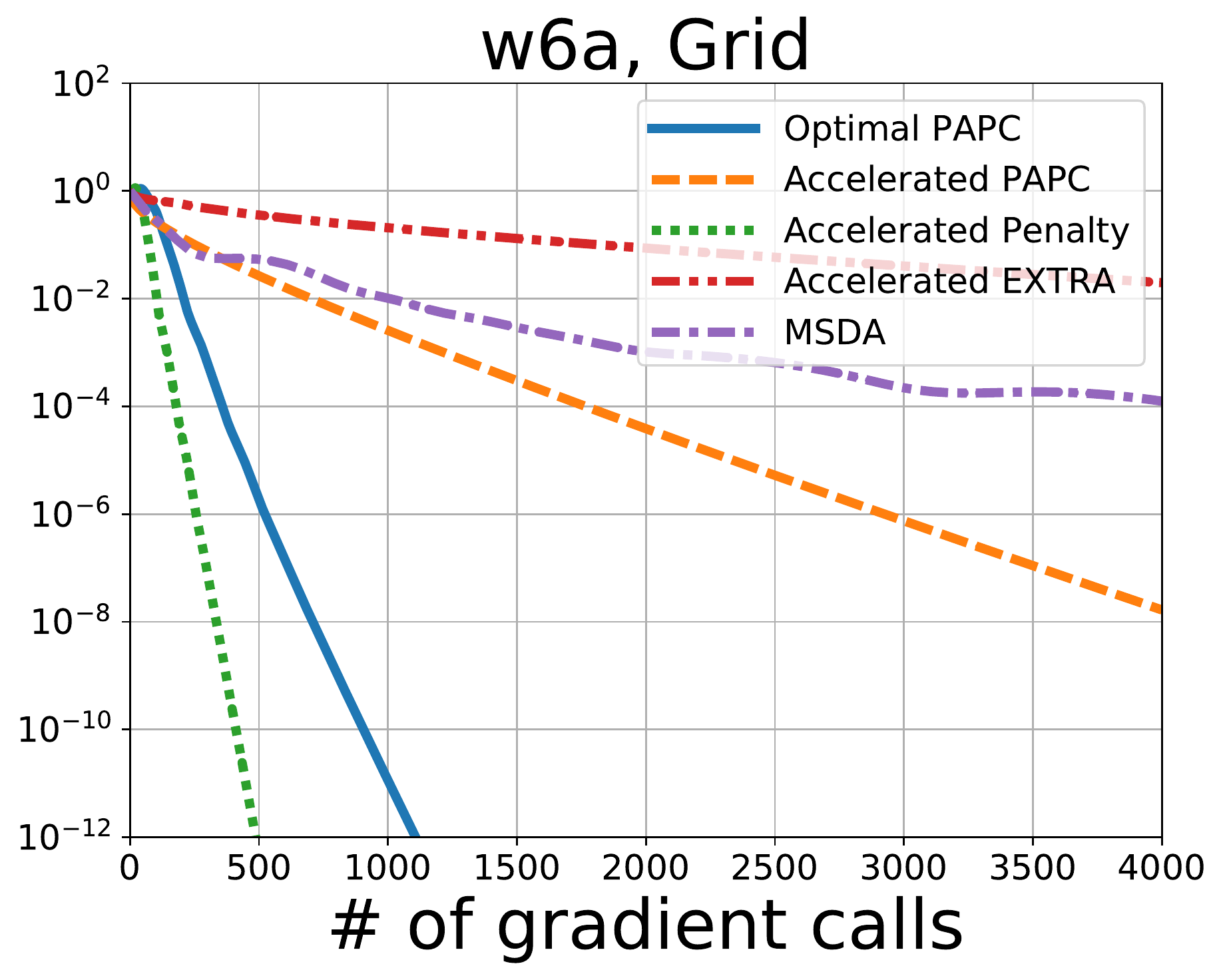}
		\includegraphics[width=.49\textwidth]{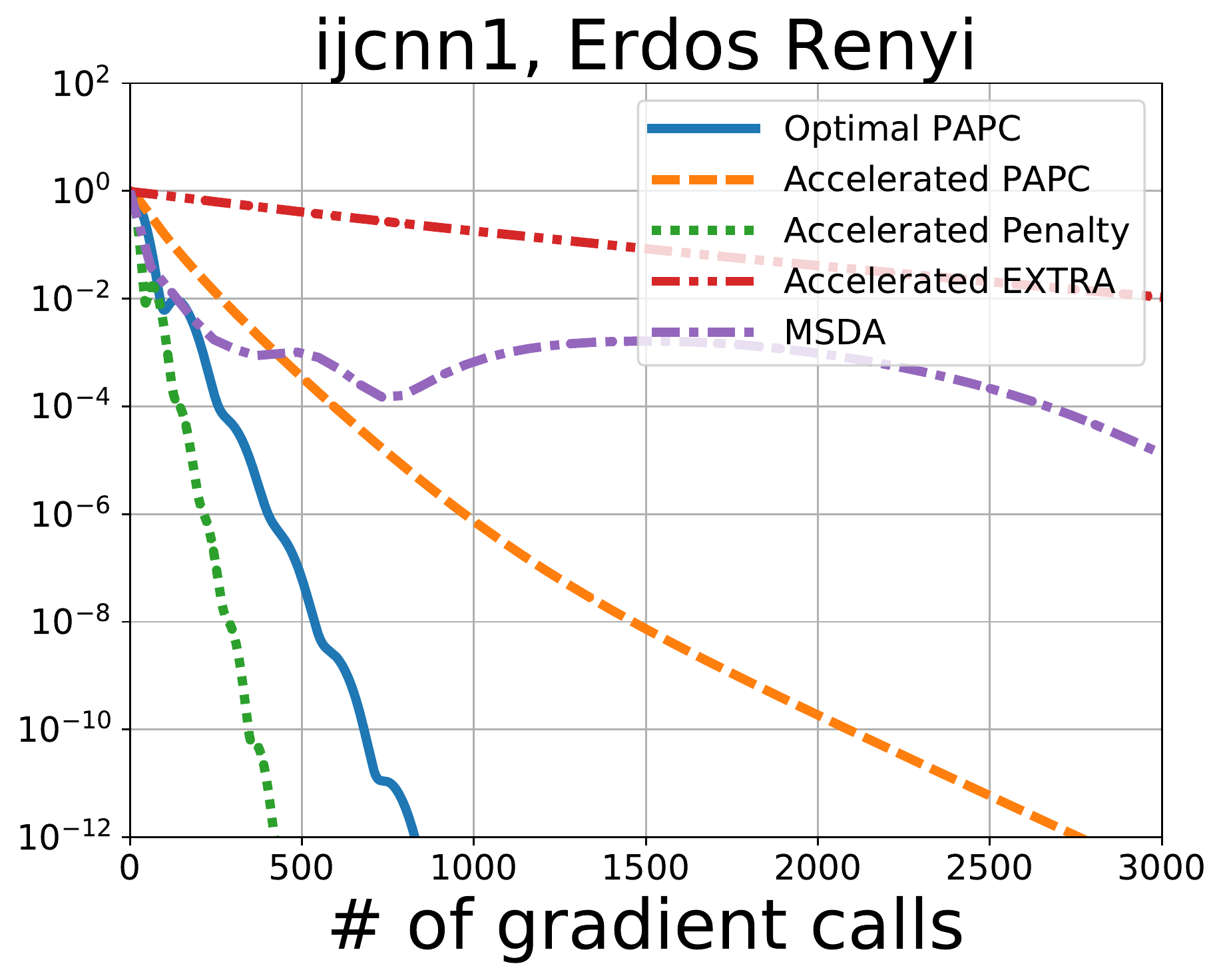}
		\includegraphics[width=.49\textwidth]{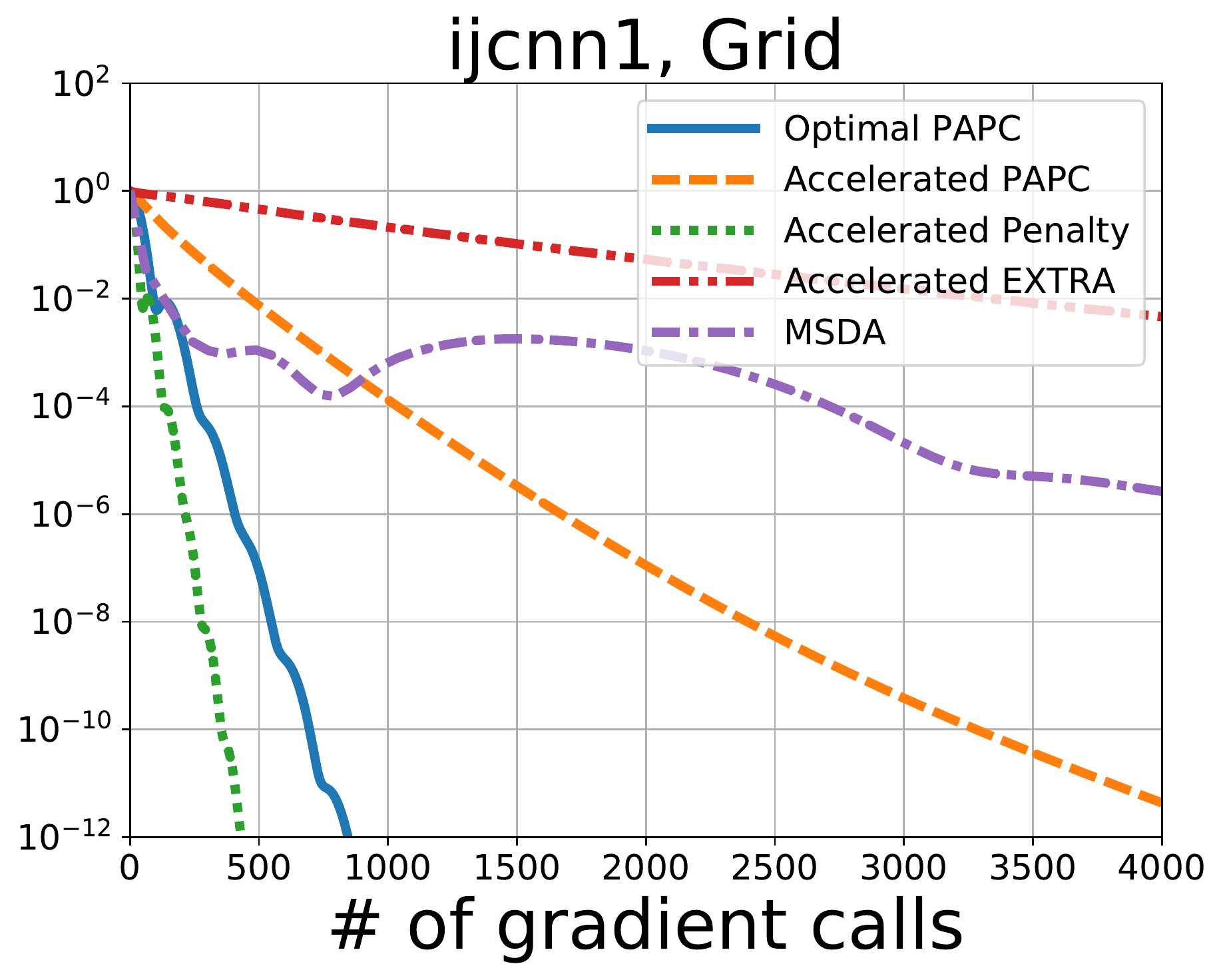}
	\caption{Gradient computation complexity.}
	\label{fig:grad}
	\end{subfigure}
	\caption{Linear convergence of decentralized algorithms in number of communication rounds and gradient computations.}
	\label{fig:sim}
\end{figure}

Figure~\ref{fig:sim} compares Algorithm~\ref{alg:ALV} (Accelerated PAPC) and Algorithm~\ref{alg:ALV-opt} (Optimal PAPC) with three state-of-the-art accelerated benchmarks: Accelerated Penalty \citep{li2018sharp,dvinskikh2019decentralized}, Accelerated Extra \citep{li2020revisiting} and MSDA~\cite{scaman2017optimal}, where we used the subroutine of~\cite{uribe2020dual} to compute the dual gradients. This subroutine uses primal gradients $\nabla f_i$, and the resulting algorithm can be shown to have an optimal communication complexity. We represent the squared distance to the solution as a function of the number of communication rounds and (primal) gradient computations.


The theory developed in this paper concerns the value of the linear rates of the proposed algorithms, i.e., the slope of the curves in Figure~\ref{fig:sim}. In communication complexity, one can see that our Algorithms~\ref{alg:ALV} and~\ref{alg:ALV-opt} have similar rate and perform better than the other benchmarks except MSDA. MSDA performs slightly better in communication complexity. However, MSDA uses dual gradients and has much  higher iteration complexity. In gradient computation complexity, one can see that our main Algorithm~\ref{alg:ALV-opt} is, alongside Accelerated Penalty, the best performing method. Accelerated Penalty performs slightly better in gradient computation complexity. However, the theory of Accelerated Penalty does not predict linear convergence in the number of communication rounds and we see that this algorithm converges sublinearly. Overall, Optimal PAPC is the only universal method which performs well both in communication rounds and gradient computations.

\subsection{Experiments with synthetic data}



In this section, we present additional experiments. The experimental setup is the same as before, with only one difference: we use randomly generated dataset with the following choice of the number of features $d$: 40, 60, 80, 100. The results, which are shown in Figure~\ref{fig:sim2}, are similar to the previous results, and the same conclusions can be made.

\begin{figure}[ht]
	\begin{subfigure}{.5\textwidth}
		\includegraphics[width=.49\textwidth]{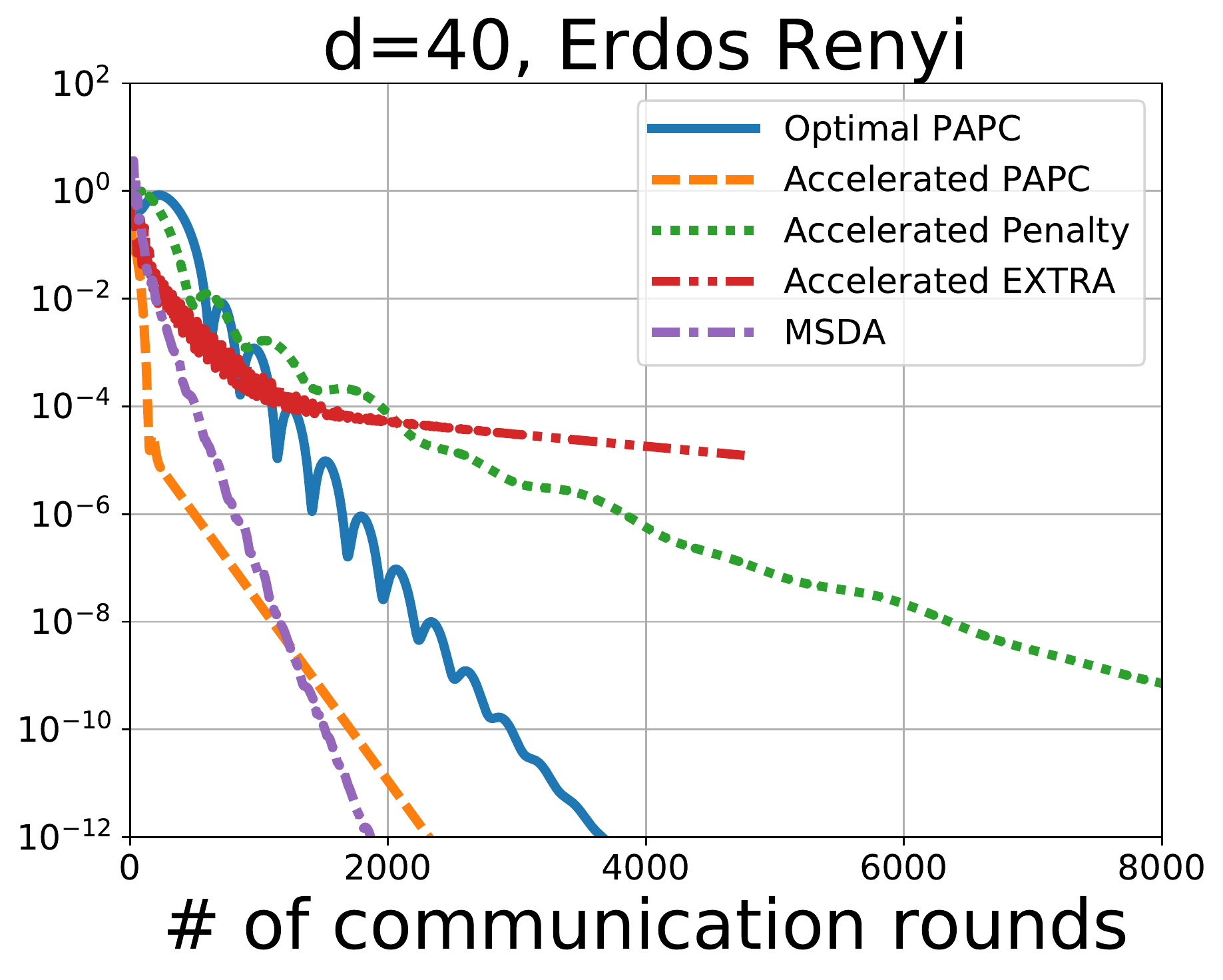}
		\includegraphics[width=.49\textwidth]{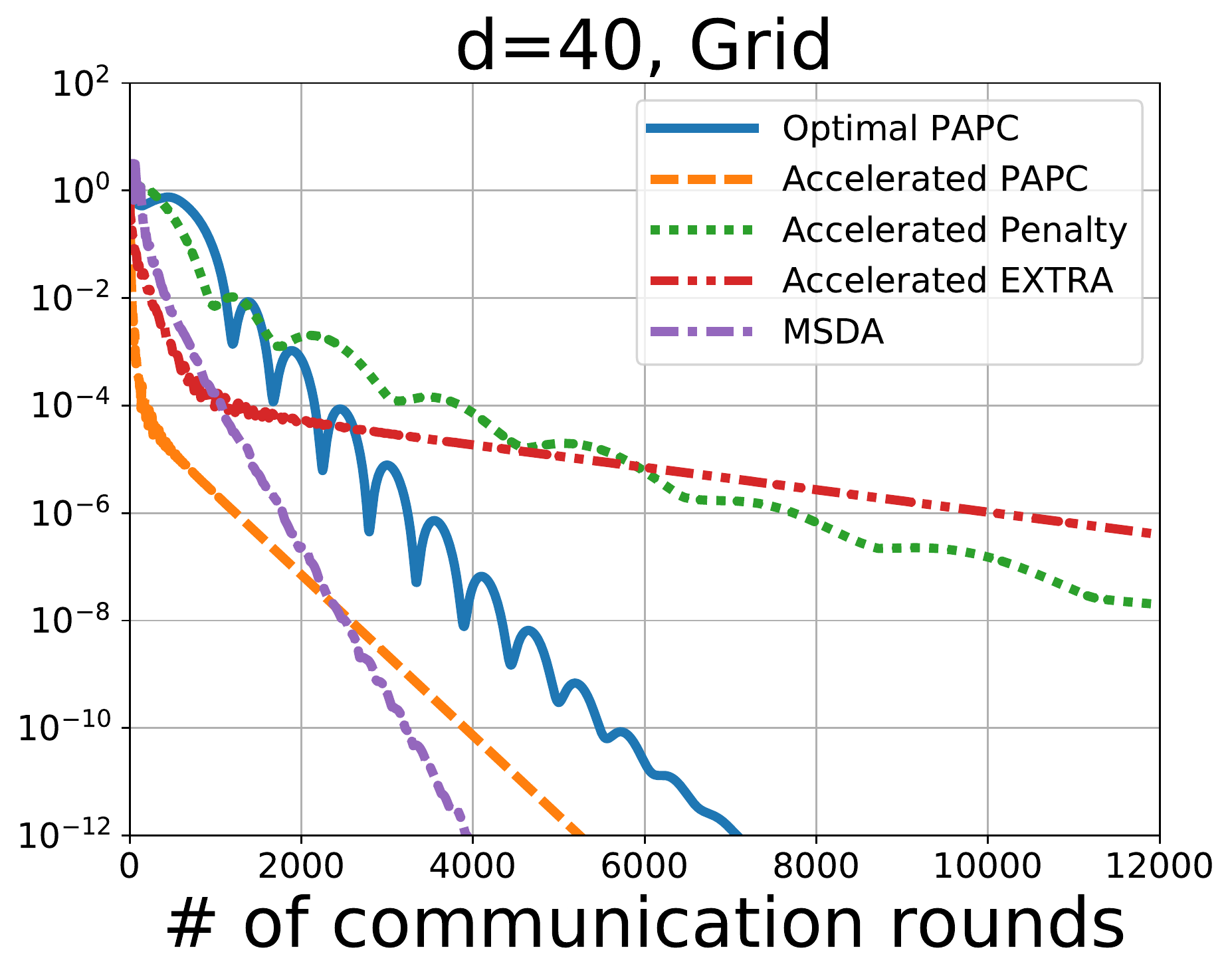}
		\includegraphics[width=.49\textwidth]{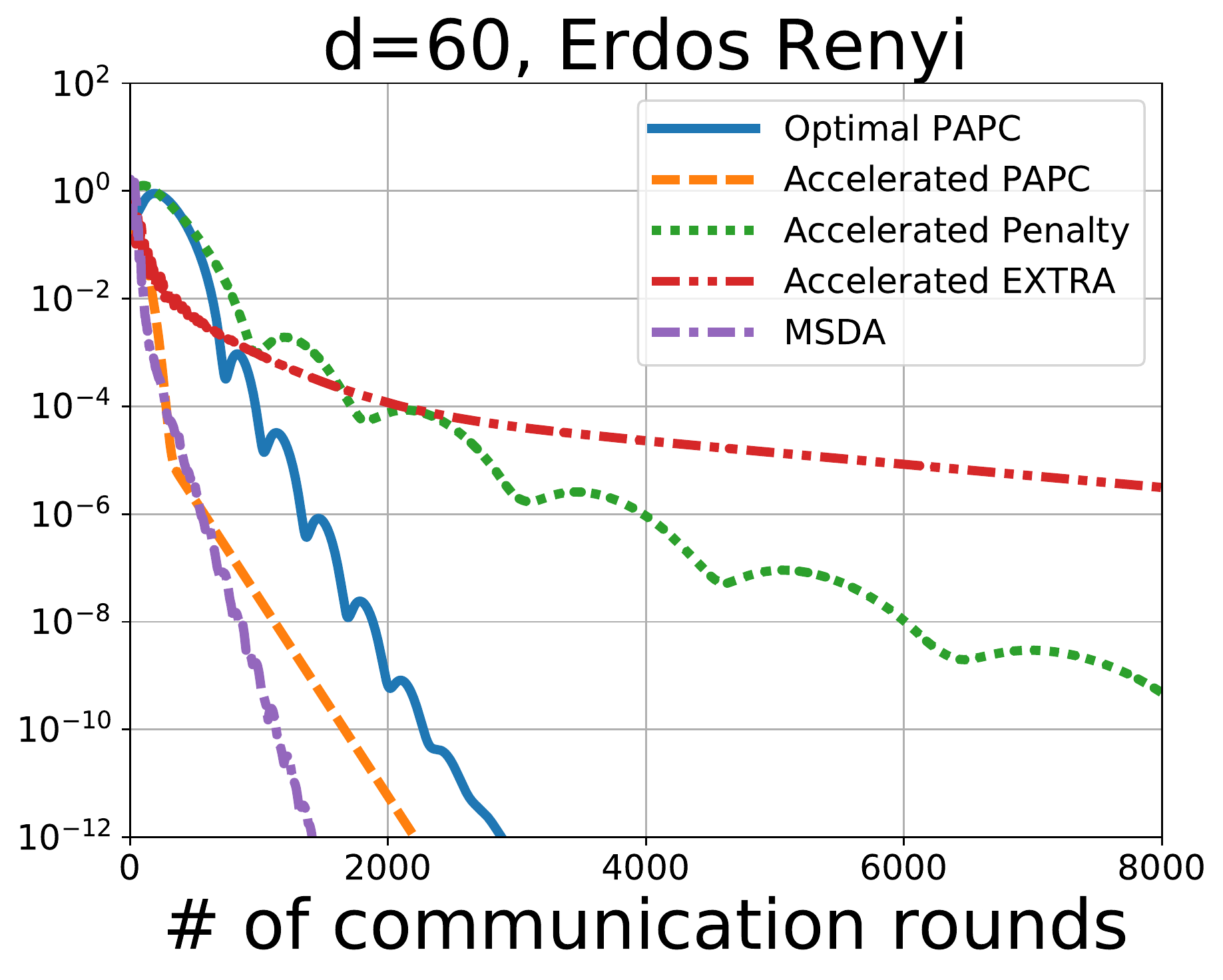}
		\includegraphics[width=.49\textwidth]{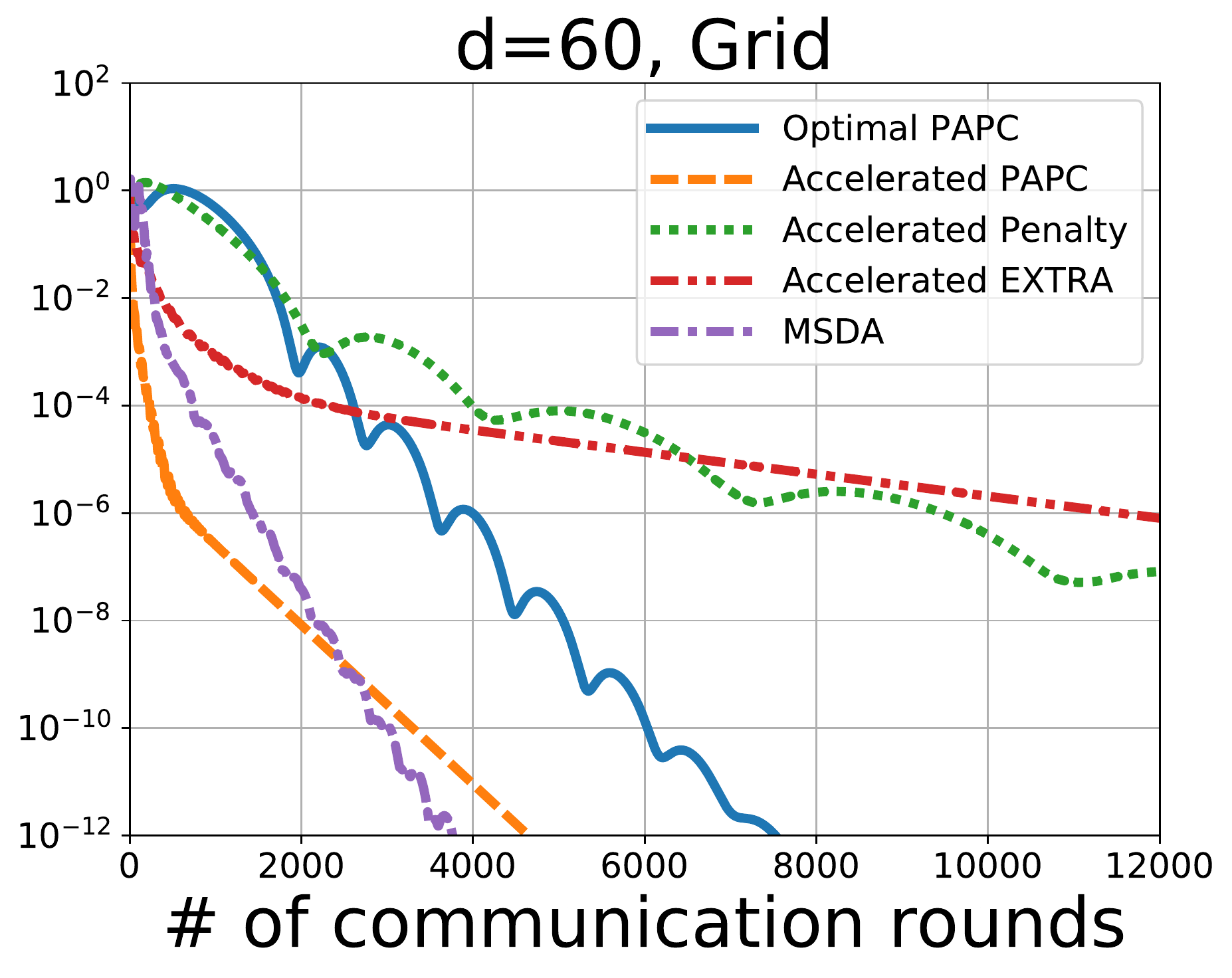}
		\includegraphics[width=.49\textwidth]{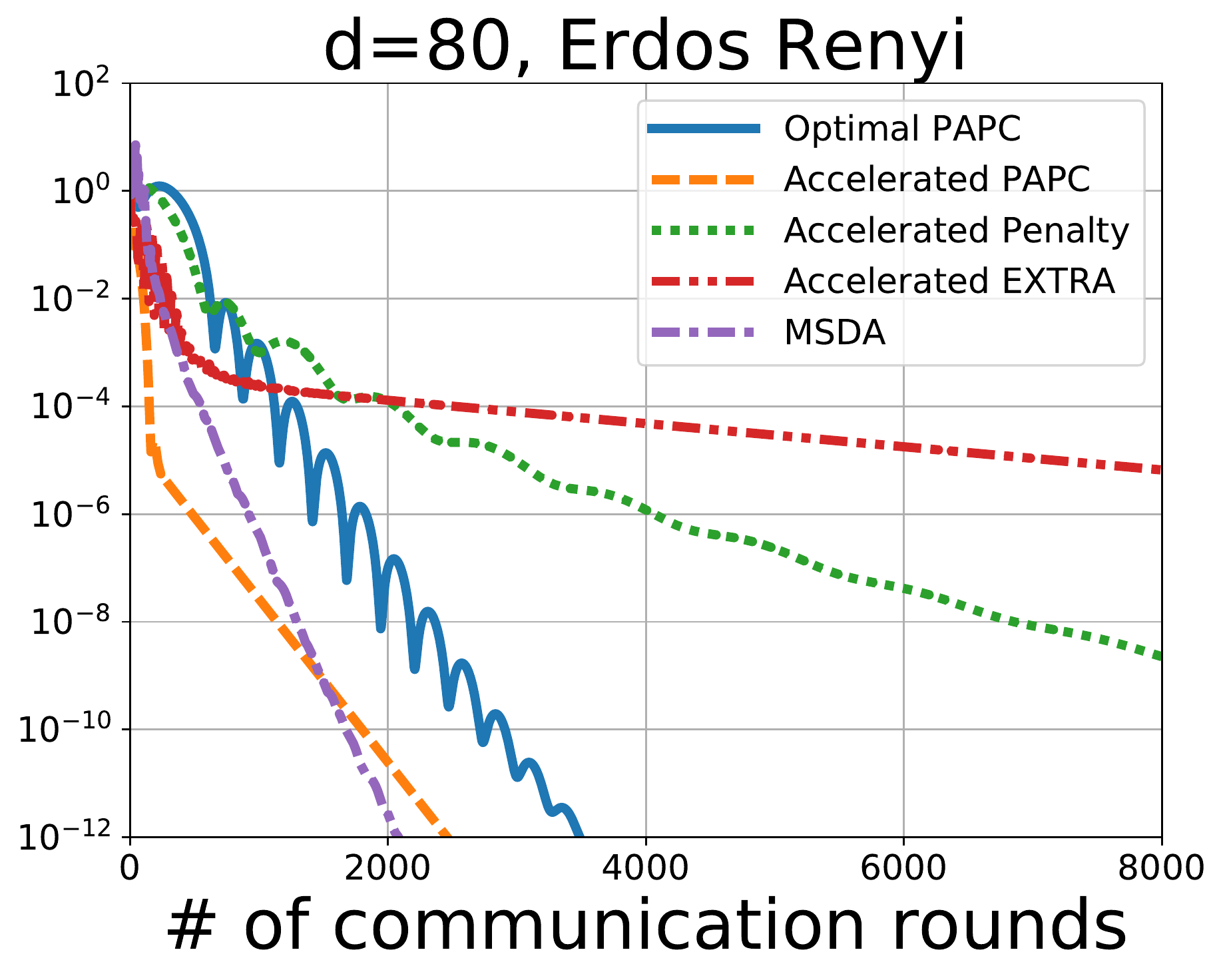}
		\includegraphics[width=.49\textwidth]{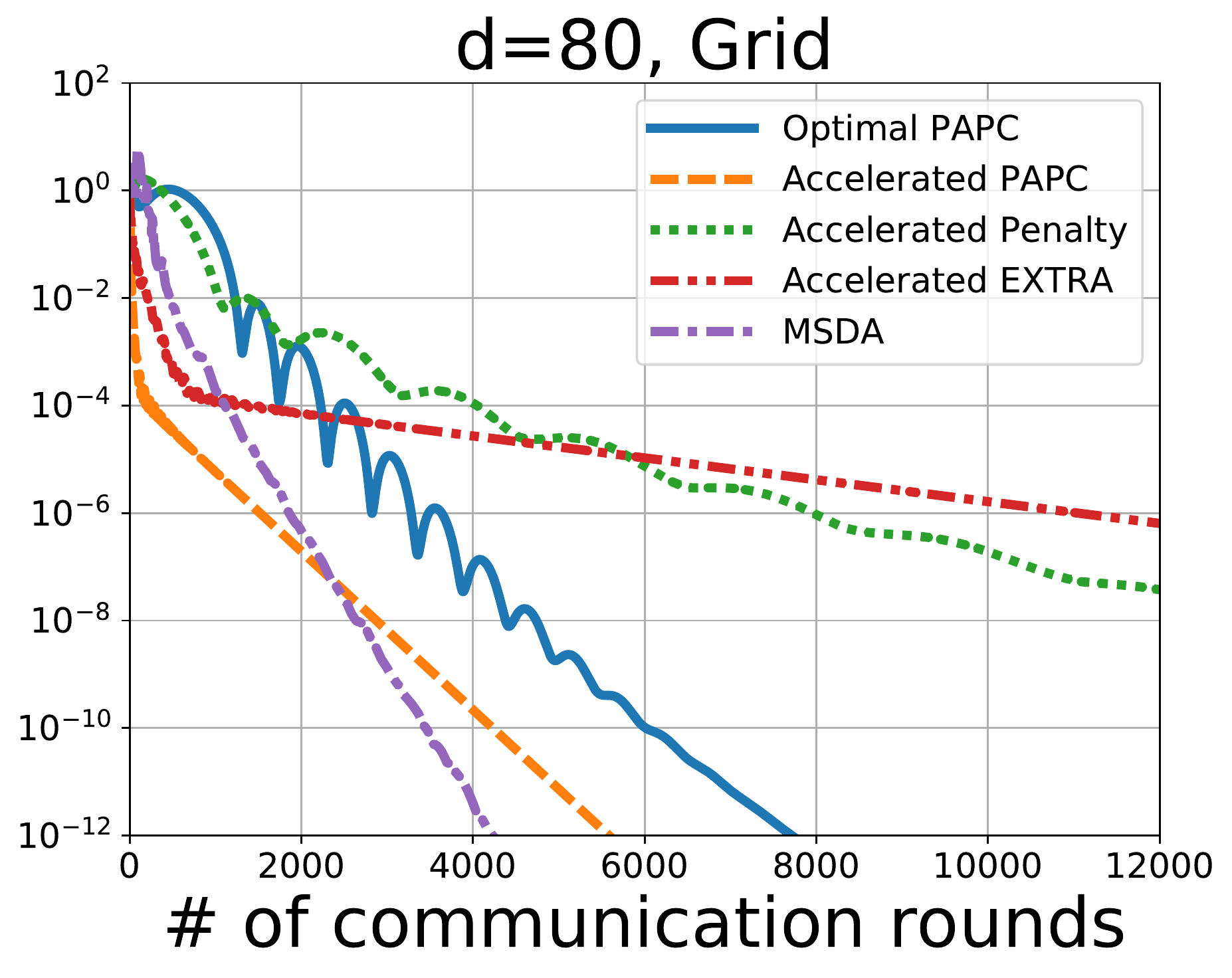}
		\includegraphics[width=.49\textwidth]{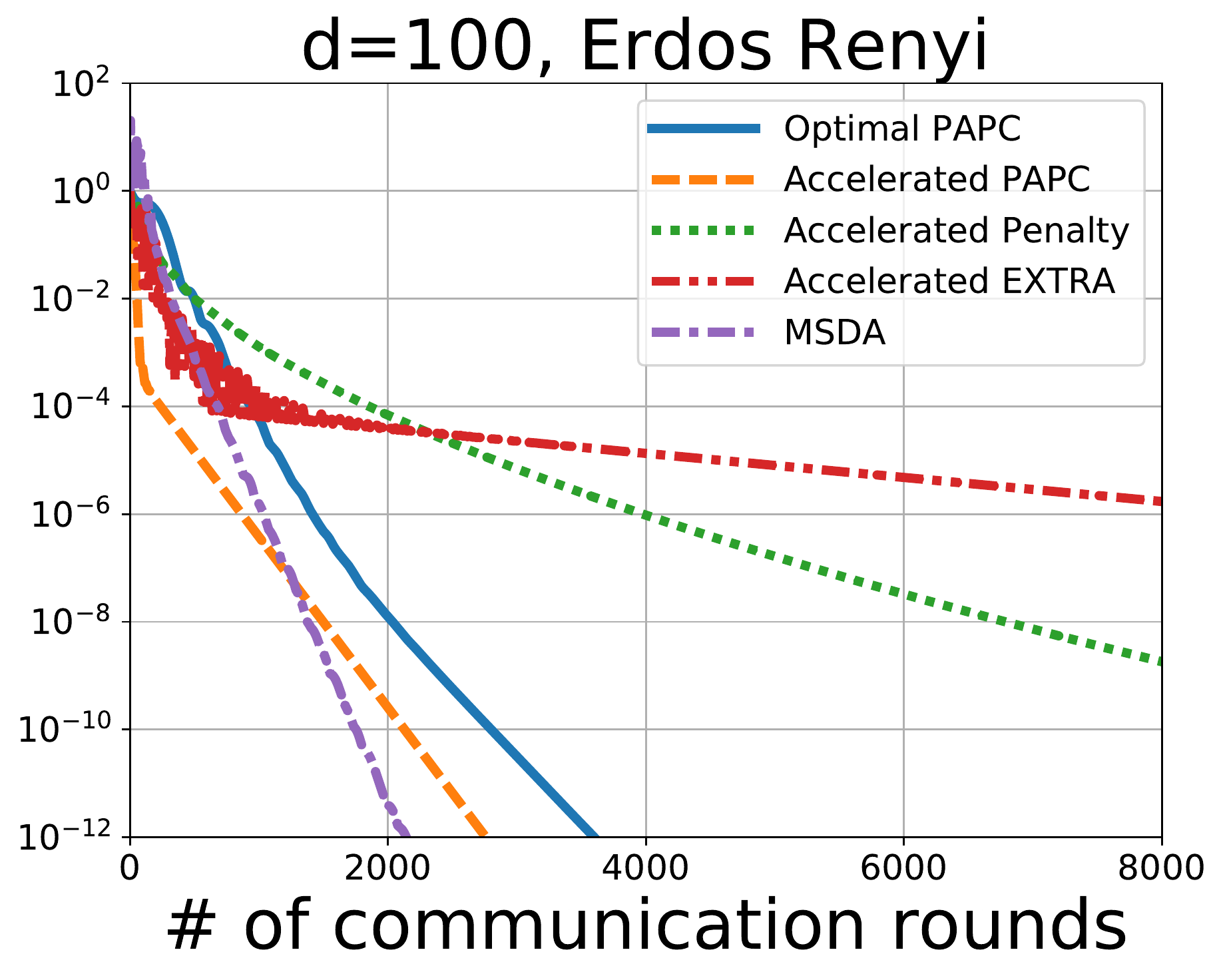}
		\includegraphics[width=.49\textwidth]{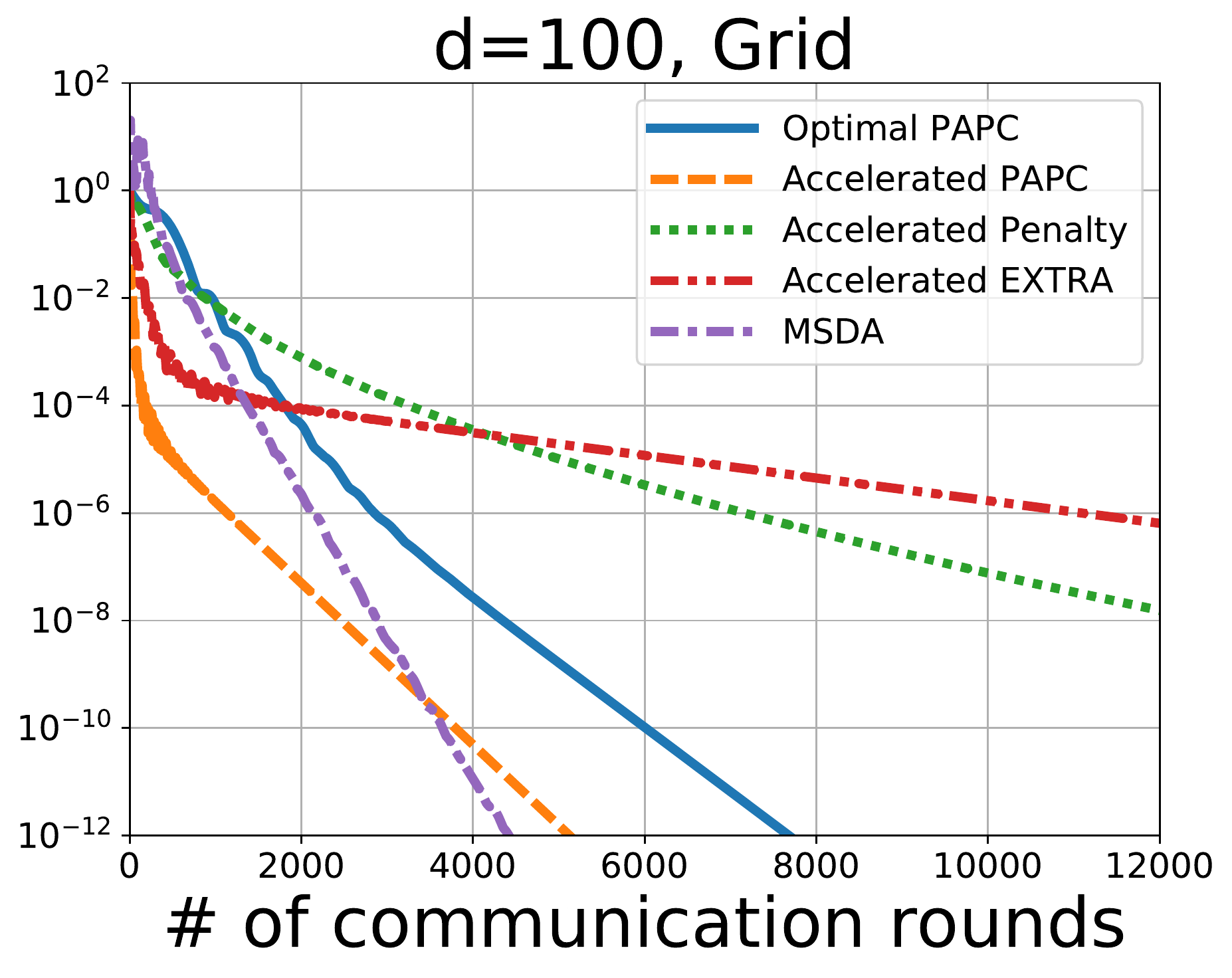}
		\caption{Communication complexity.}
	\end{subfigure}
	\begin{subfigure}{.5\textwidth}
		\includegraphics[width=.49\textwidth]{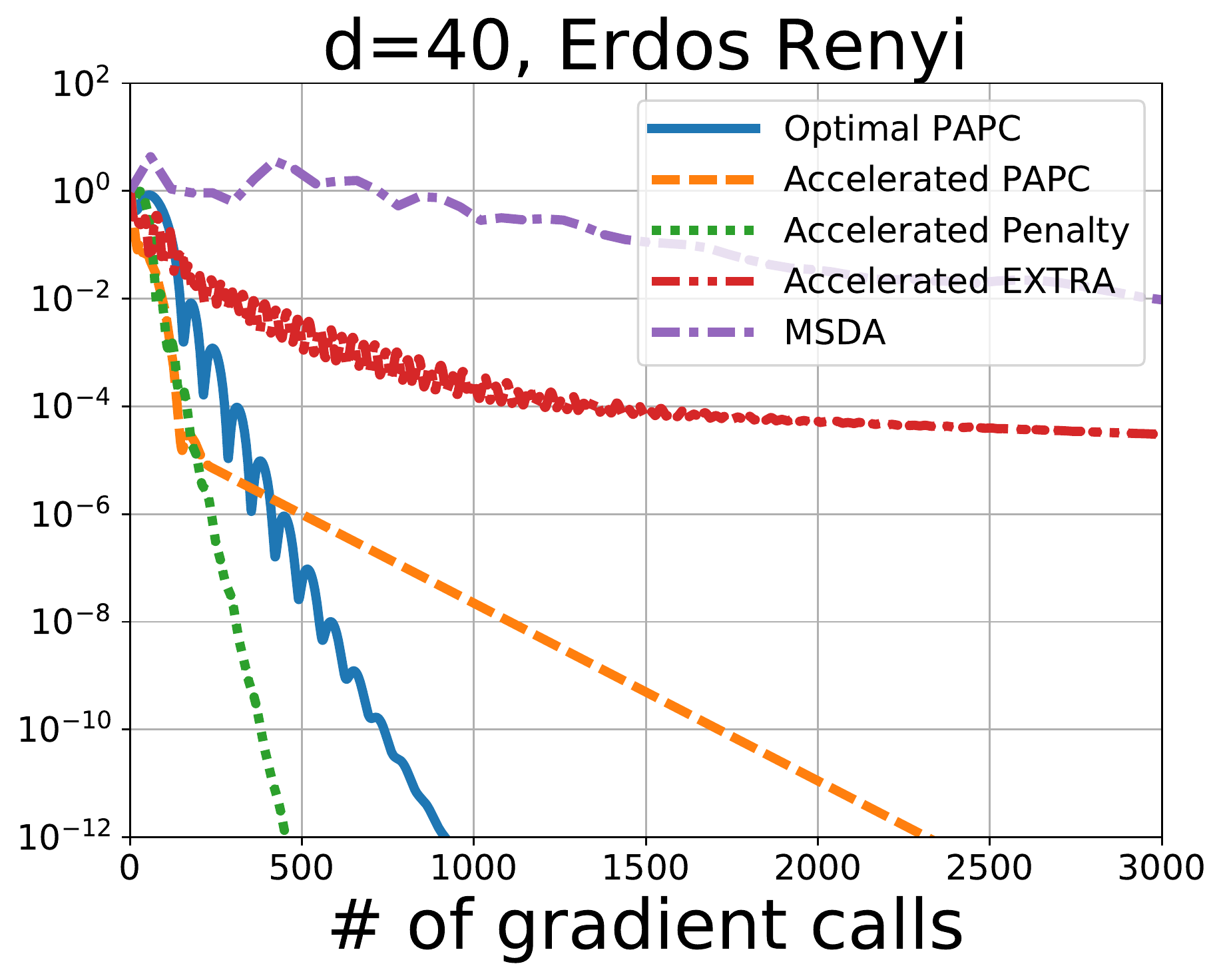}
		\includegraphics[width=.49\textwidth]{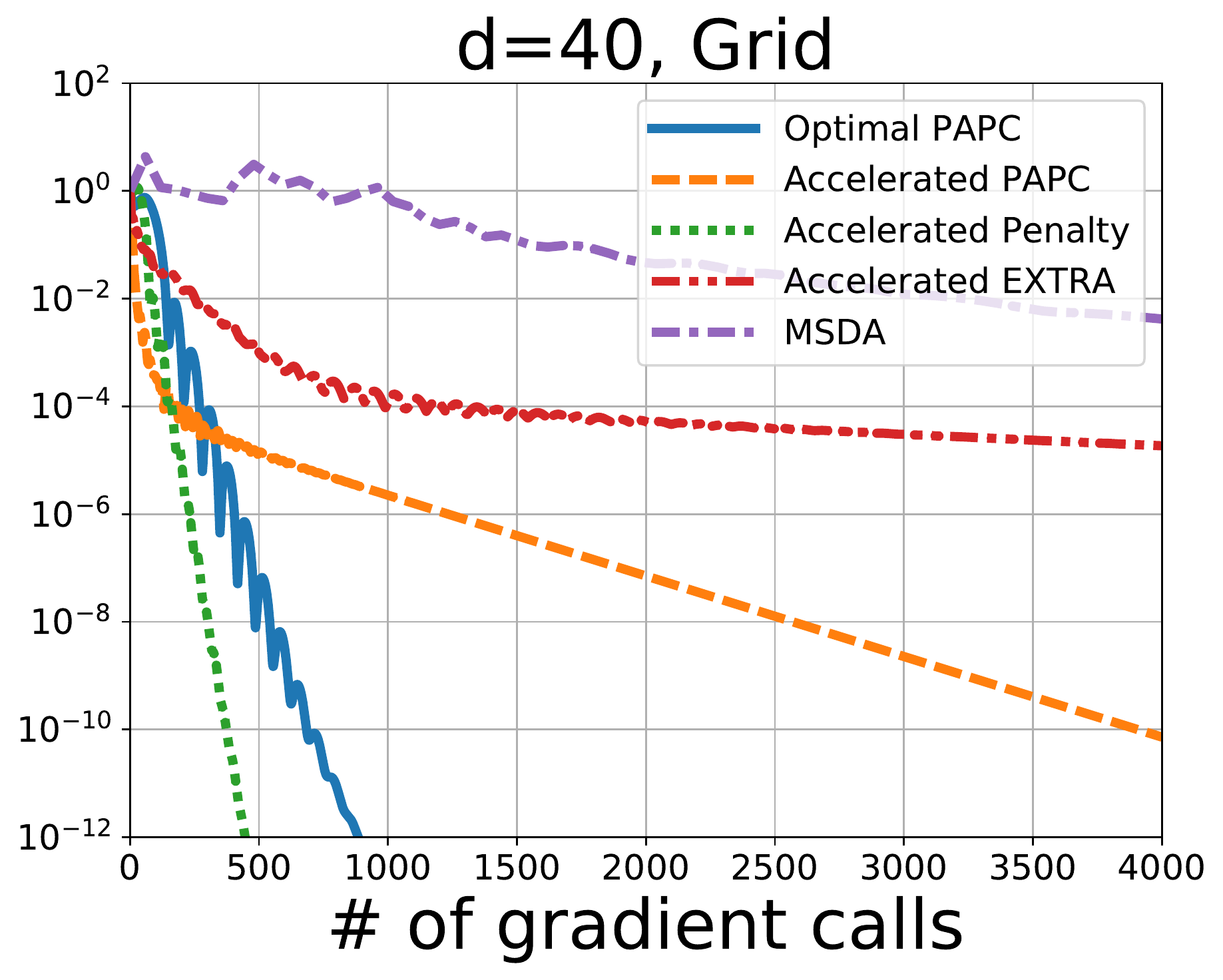}
		\includegraphics[width=.49\textwidth]{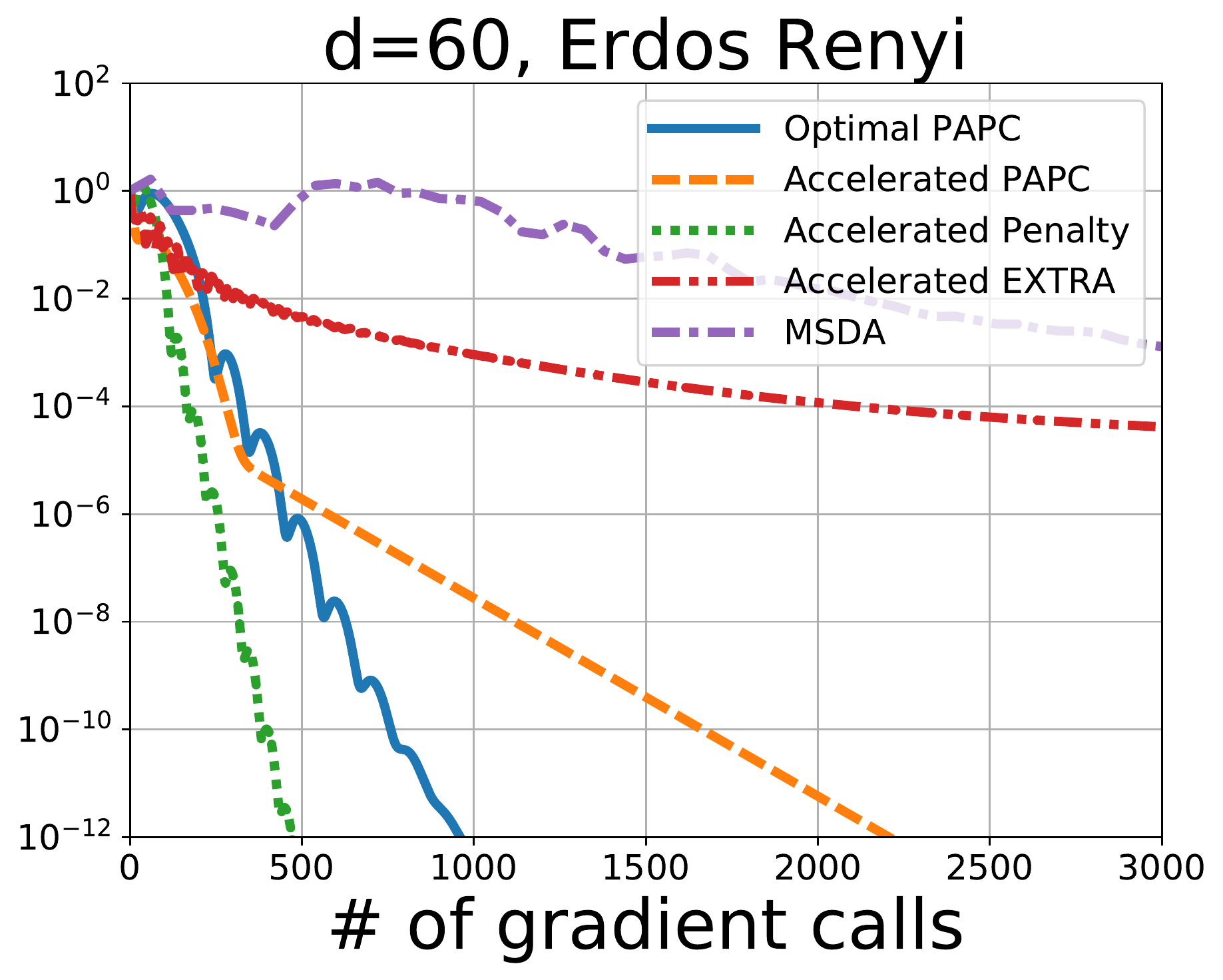}
		\includegraphics[width=.49\textwidth]{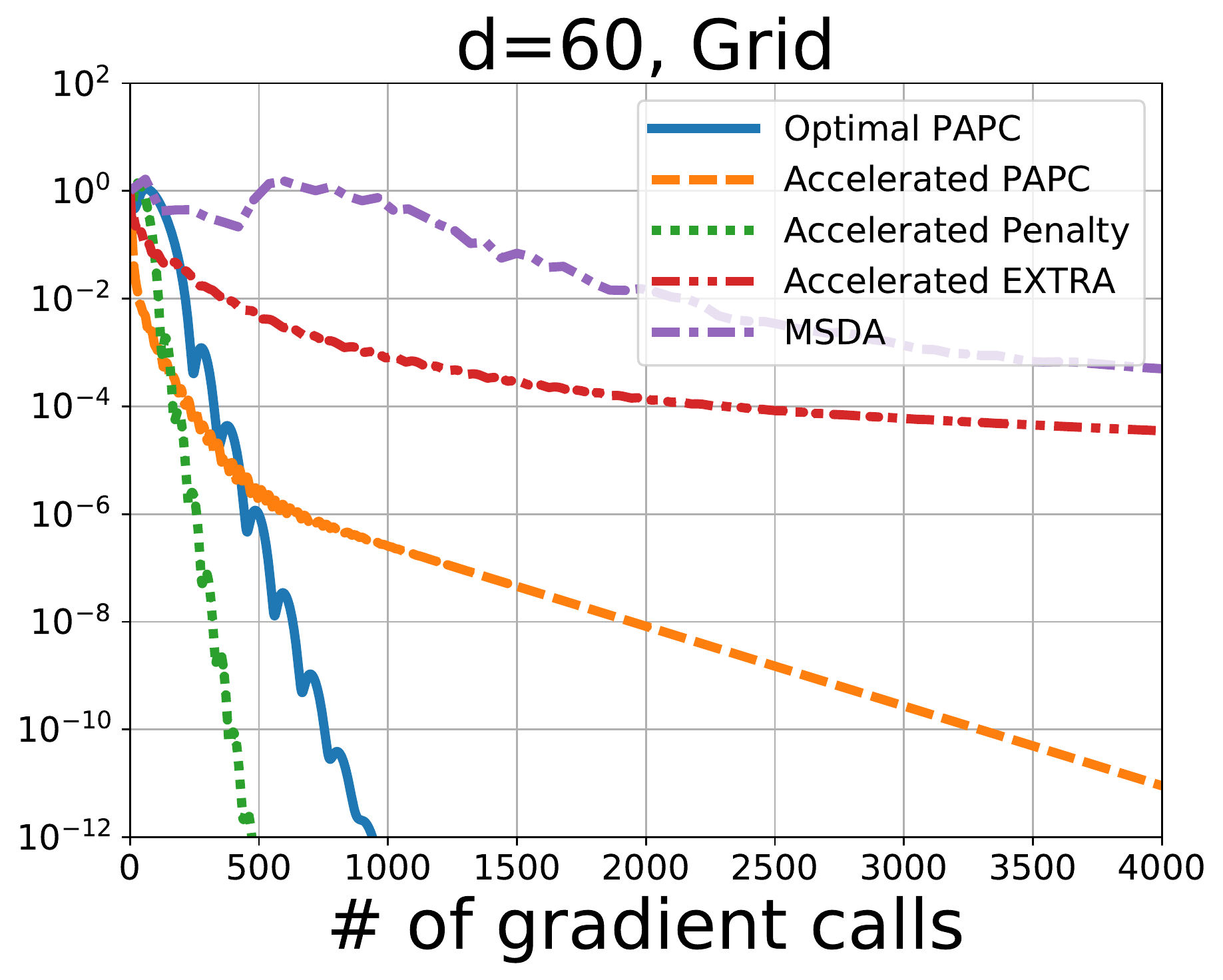}
		\includegraphics[width=.49\textwidth]{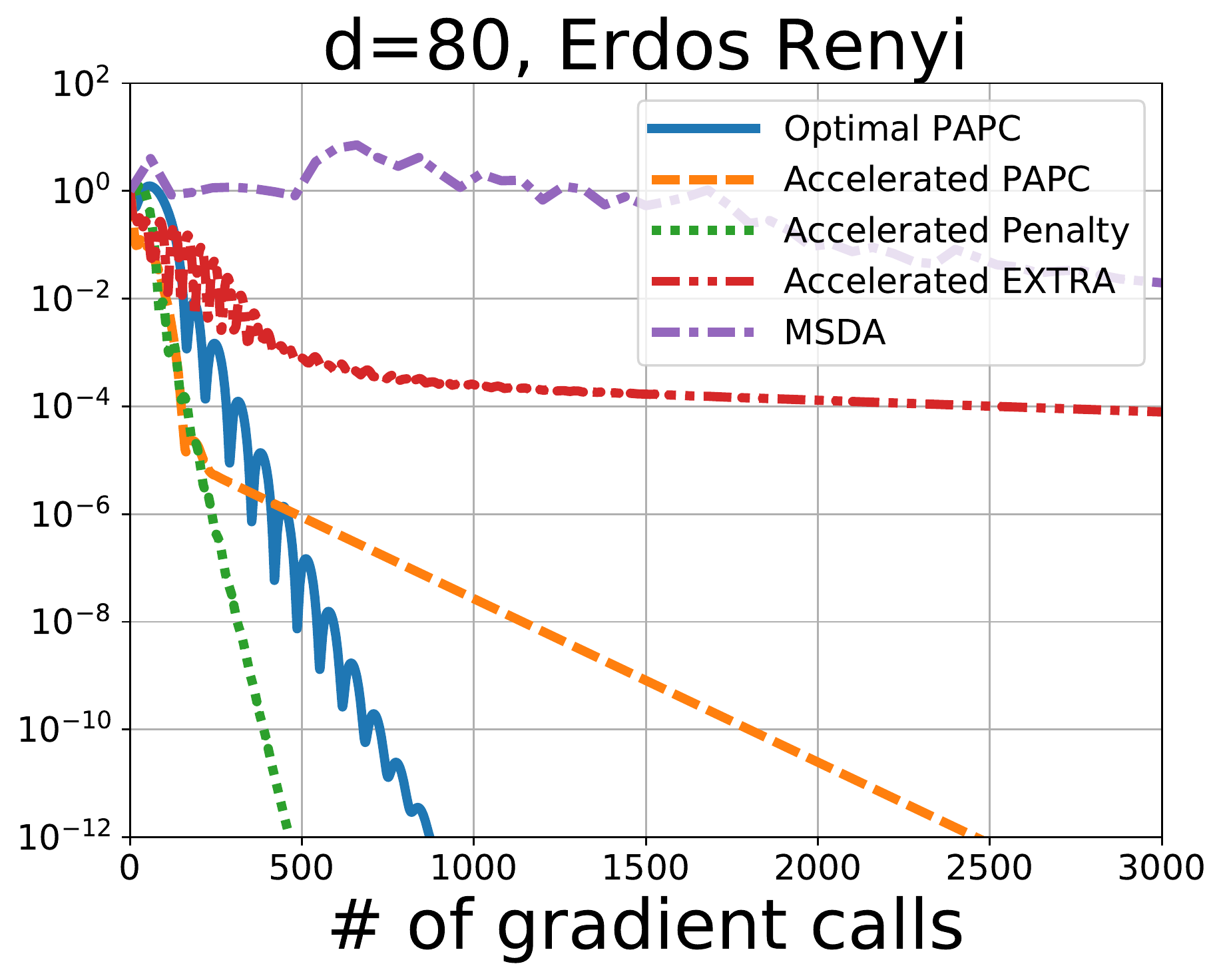}
		\includegraphics[width=.49\textwidth]{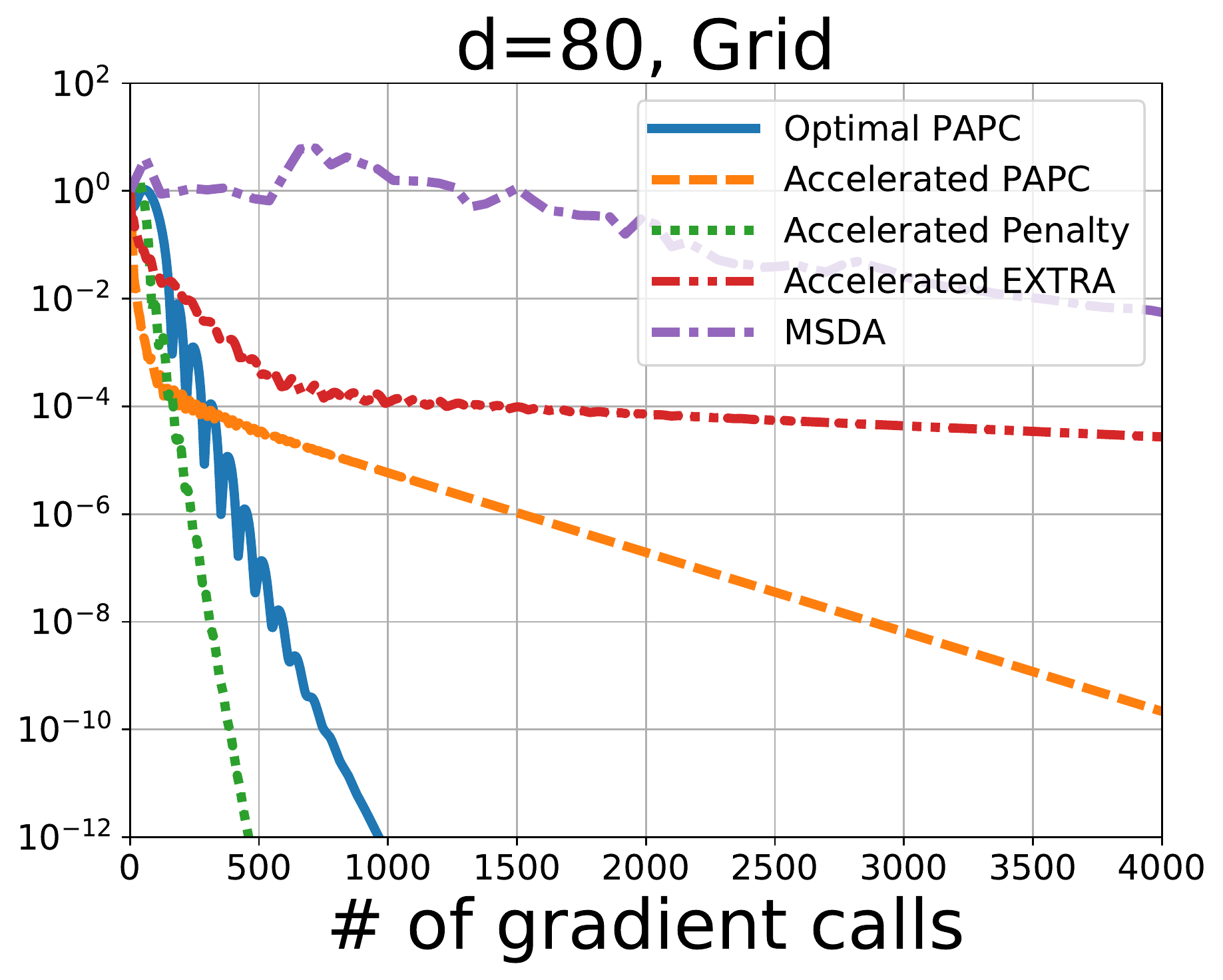}
		\includegraphics[width=.49\textwidth]{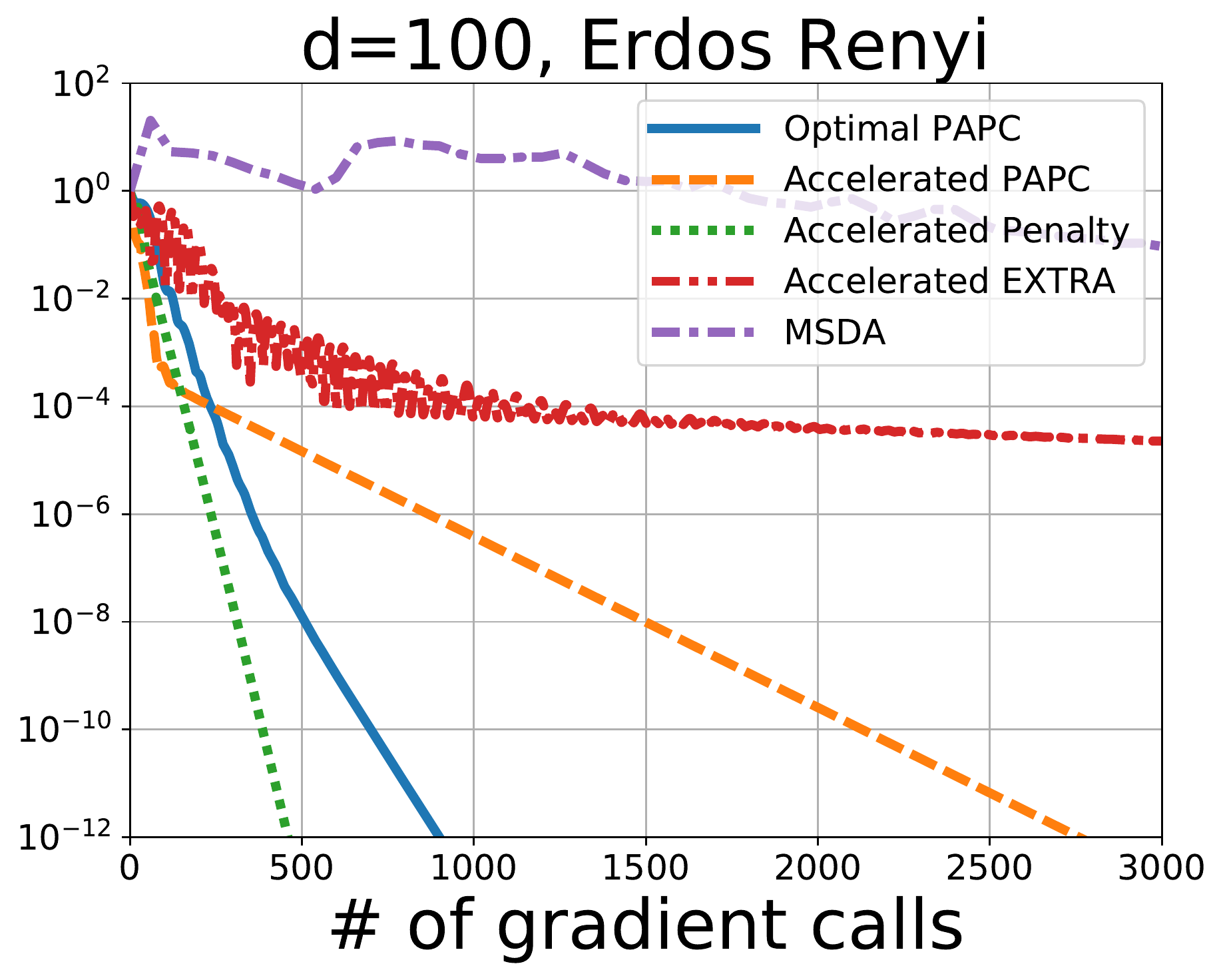}
		\includegraphics[width=.49\textwidth]{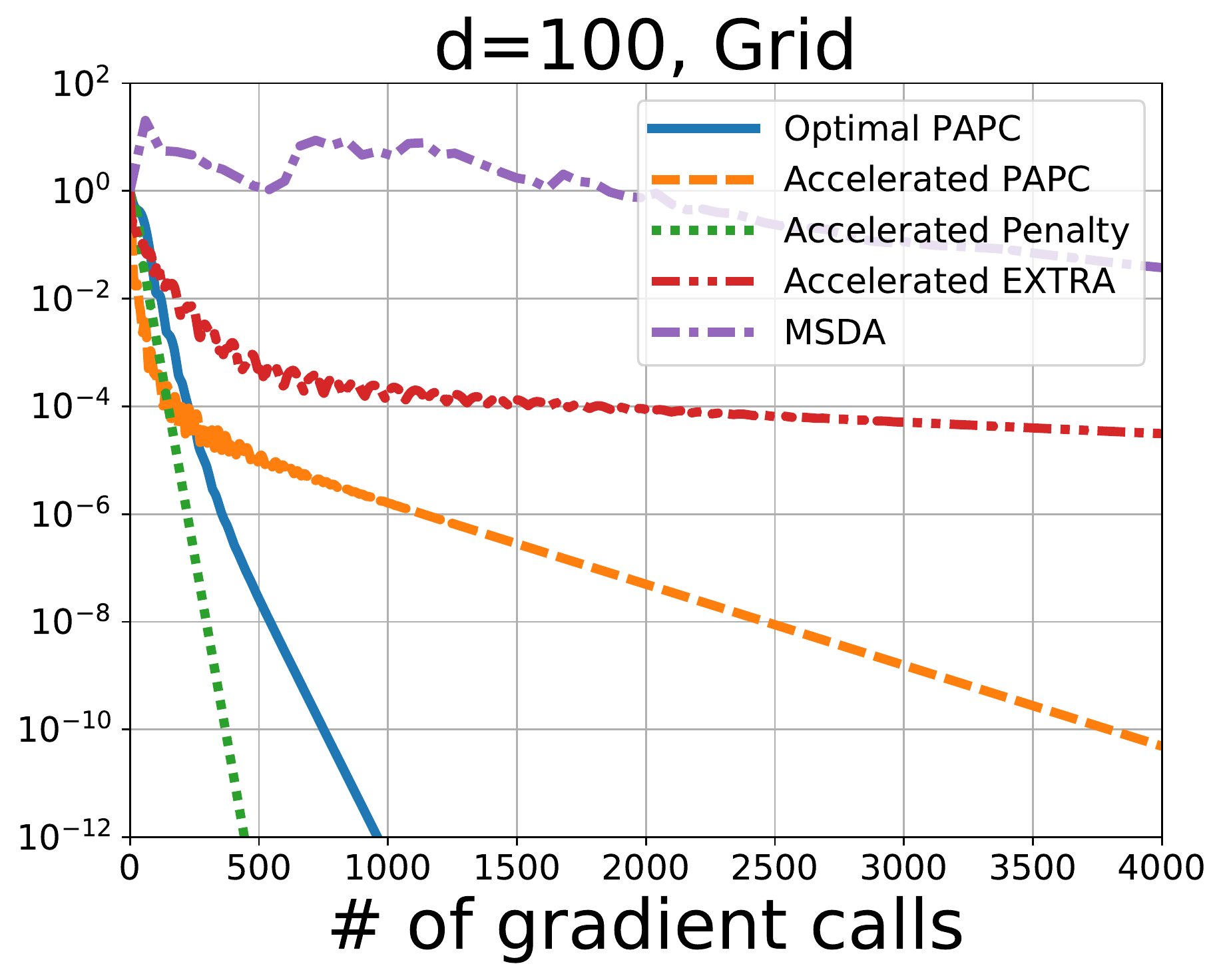}
		\caption{Gradient computation complexity.}
		
	\end{subfigure}
	\caption{Linear convergence of decentralized algorithms in number of communication rounds and gradient computations.}
	\label{fig:sim2}
\end{figure}

	\bibliographystyle{apalike}

\clearpage
\appendix

\part*{Appendix}

\tableofcontents

\clearpage

\newpage

\section{Formal Definition of Decentralized Algorithms}

	In this paper, we considered the resolution of~\eqref{eq:1} distributively across the nodes of the network $G$. Each node $i \in \cV$ is associated with a computing agent that only have access to the local function $f_i$. The goal of the network of computing agent is to minimize the function~\eqref{eq:1} by performing local computations involving $f_i$ at each node $i$ and by communicating vectors along the edges, i.e., with neighbors $j \sim i$.
	
	More precisely, we considered the class of decentralized algorithms, similarly to~\citep[Section 3.1]{scaman2017optimal}. In this paper, a decentralized algorithm is formally defined as an algorithm satisfying the following constraints. At time $k$, each node $i$ possesses a local internal memory $M_i^k \subset \R^d$ and outputs an estimation $x_i^{k} \in M_i^k$ of the solution to Problem~\eqref{eq:1}. This internal memory is updated via gradient computations and communication rounds i.e., $$M_i^{k+1} \subset \Span(\Comm_i^{k+1} \bigcup \Comp_i^{k+1}),$$ where $\Comm_i^{k+1}$ is the communication component and $\Comp_i^{k+1}$ the computation component. The communication component is updated by combining the elements of the local memories of nodes $j \sim i$ at time $k$: $\Comm_i^{k+1} = \Span(\bigcup_{j \sim i} M_j^k)$. The computation component is updated by combining the elements of the local memory of $i$ at time $k$ along with the gradients of the local functions $f_i$ at these elements: $\Comp_i^{k+1} = \Span(\{x, \nabla f_i(x), x \in M_i^k\})$. 
	Compared to the class of black-box optimization procedures of~\citep{scaman2017optimal}, the class of decentralized algorithm is smaller (i.e., included). Indeed, black-box optimization procedures use dual gradients. In other words, they use the following definition of the computation component: $$\widetilde{\Comp}_i^{k+1} = \{x, \nabla f_i(x), \nabla f_i^*(x), x \in M_i^k\}$$ (where $f_i^*$ is the Fenchel transform of $f_i$), which is a set containing $\Comp_i^{k+1}$. Recall that computing the dual gradient $\nabla f_i^*(0)$ is equivalent to minimizing $f_i$.
	
	Finally, as in~\cite{scaman2017optimal}, we say that a decentralized algorithm uses the gossip matrix $\mW$ if the local communication is achieved by multiplication of a vector by $\mW$.

	
%
	


	\newpage

\section{Proof of Theorem~\ref{th:ALV} (APAPC)}

For every $p \geq 0$, we denote by $\|\cdot\|_{\mP}$ the (semi)-norm induced by any positive (semi)-definite matrix $\mP : \R^p \to \R^p$.

\begin{lemma}
	Let $\mP \in \R^{2nd\times 2nd}$ be the following matrix:
	\begin{equation}\label{ALV:eq:P}
		\mP = \begin{bmatrix}
			\frac{1}{\eta}\mI & 0\\
			0&\frac{1}{\theta}\mW^\dagger - (1+\eta\alpha)^{-1}\eta\mI
		\end{bmatrix}.
	\end{equation}
	If parameters $\eta$ and $\theta$ satisfy
	\begin{equation}\label{ALV:eq:etatheta}
		\eta\theta\lambda_{\max}(\mW) \leq 1,
	\end{equation}
	then for all $x \in \R^{nd}$, $y \in \range\mW$ the following inequality holds:
	\begin{equation}\label{ALV:eq:Pbound}
		\frac{1}{\eta}\sqn{x} \leq \sqN{\begin{bmatrix}x\\y\end{bmatrix}}_\mP \leq \frac{1}{\eta}\sqn{x} + \frac{1}{\theta}\sqnw{y}.
	\end{equation}
\end{lemma}

\begin{proof}
	Note that under our assumptions, the matrix $\frac{1}{\theta}\mW^\dagger - (1+\eta\alpha)^{-1}\eta\mI$ is positive semi-definite on $\range \mW$.
\end{proof}

\begin{lemma}\label{ALV:lem:1}
	Let $\alpha$ satisfy $0\leq \alpha \leq \mu$.
	Then the following inequality holds:
	\begin{align}\label{ALV:eq:1}
		-\frac{1}{2\eta}\sqn{x^{k+1} - x^k} \leq -\frac{\eta}{4}\sqn{y^{k+1} - y^*}
		+
		\eta\alpha^2\sqn{x^{k+1} - x^*}
		+
		2\eta L\bg_{f }(x_g^k, x^*).
	\end{align}
\end{lemma}
\begin{proof}
	From line \eqref{alg:ALV:line:x:3} of Algorithm~\ref{alg:ALV} and optimality condition \eqref{eq:primaldualoptimal} it follows that
	\begin{align*}
		\sqn{x^{k+1} - x^k}
		&=
		\sqn{\eta(y^{k+1} - y^*) + \eta(\nabla F(x_g^k) - \nabla F(x^*) - \alpha (x_g^k - x^*)) + \eta\alpha(x^{k+1} - x^*)}
		\\&\geq
		\frac{\eta^2}{2}\sqn{y^{k+1} - y^*}
		-
		2\eta^2\alpha^2\sqn{x^{k+1} - x^*}
		\\&-
		2\eta^2\sqn{\nabla F(x_g^k) - \nabla F(x^*) - \alpha (x_g^k - x^*)}.
	\end{align*}
	Since $f(x)  - \frac{\alpha}{2}\sqn{x}$ is a convex and $(L-\alpha)$-smooth function, we can lower bound the last term and get
	\begin{align*}
	\sqn{x^{k+1} - x^k}
	&=
	\sqn{\eta(y^{k+1} - y^*) + \eta(\nabla F(x_g^k) - \nabla F(x^*) - \alpha (x_g^k - x^*)) + \eta\alpha(x^{k+1} - x^*)}
	\\&\geq
	\frac{\eta^2}{2}\sqn{y^{k+1} - y^*}
	-
	2\eta^2\alpha^2\sqn{x^{k+1} - x^*}
	-
	4\eta^2(L - \alpha)\bg_{f - \frac{\alpha}{2}\sqn{\cdot}}(x_g^k, x^*)
	\\&\geq
	\frac{\eta^2}{2}\sqn{y^{k+1} - y^*}
	-
	2\eta^2\alpha^2\sqn{x^{k+1} - x^*}
	-
	4\eta^2L\bg_{f }(x_g^k, x^*).
	\end{align*}
	Rearranging and dividing by $2\eta$ concludes the proof.
\end{proof}

\begin{lemma}
	Let $\mP$ be the matrix defined by \eqref{ALV:eq:P}:
	\begin{equation*}
		\mP = \begin{bmatrix}
		\frac{1}{\eta}\mI & 0\\
		0&\frac{1}{\theta}\mW^\dagger - (1+\eta\alpha)^{-1}\eta\mI
		\end{bmatrix}.\tag{\ref{ALV:eq:P}}
	\end{equation*}
	Then the following equality holds:
	\begin{equation}\label{ALV:eq:xy}
		\mP \cdot \begin{bmatrix}
		x^{k+1} - x^k\\y^{k+1} - y^k
		\end{bmatrix}
		=
		\begin{bmatrix}
			\alpha (x_g^k - x^{k+1}) - (\nabla F(x_g^k) + y^{k+1})\\
			\mW\mW^\dagger x^{k+1}
		\end{bmatrix}.
	\end{equation}
\end{lemma}

\begin{proof}
	From the definition of $\mP$ it follows that
	\begin{align*}
	\mP \cdot \begin{bmatrix}
	x^{k+1} - x^k\\y^{k+1} - y^k
	\end{bmatrix}
	=
	\begin{bmatrix}
	\frac{1}{\eta}(x^{k+1} - x^k)\\
	\frac{1}{\theta}\mW^\dagger(y^{k+1} - y^k) - (1+\eta\alpha)^{-1}\eta(y^{k+1} - y^k)
	\end{bmatrix}.
	\end{align*}
	From line \eqref{alg:ALV:line:x:3} of Algorithm~\ref{alg:ALV} it follows that
	\begin{align*}
		\frac{1}{\eta}(x^{k+1} - x^k) = \alpha(x_g^k - x^{k+1})-(\nabla F(x_g^k) + y^{k+1}),
	\end{align*}
	and hence,
	\begin{align*}
	\mP \cdot \begin{bmatrix}
	x^{k+1} - x^k\\y^{k+1} - y^k
	\end{bmatrix}
	=
	\begin{bmatrix}
	\alpha(x_g^k - x^{k+1})-(\nabla F(x_g^k) + y^{k+1})\\
	\frac{1}{\theta}\mW^\dagger(y^{k+1} - y^k) - (1+\eta\alpha)^{-1}\eta(y^{k+1} - y^k)
	\end{bmatrix}.
	\end{align*}
	From line \eqref{alg:ALV:line:y} of Algorithm~\ref{alg:ALV} it follows that
	\begin{align*}
		y^{k+1} - y^k = \theta\mW x^{k+1/2},
	\end{align*}
	and hence,
	\begin{align*}
	\mP \cdot \begin{bmatrix}
	x^{k+1} - x^k\\y^{k+1} - y^k
	\end{bmatrix}
	=
	\begin{bmatrix}
	\alpha(x_g^k - x^{k+1})-(\nabla F(x_g^k) + y^{k+1})\\
	\mW\mW^\dagger x^{k+1/2}- (1+\eta\alpha)^{-1}\eta(y^{k+1} - y^k)
	\end{bmatrix}.
	\end{align*}
	Since $y^k \in \range \mW$ for all $k=0,1,2,\ldots$, we have 
	\begin{align*}
	\mW\mW^\dagger (y^{k+1} - y^k) = y^{k+1} - y^k,
	\end{align*}
	and hence we obtain
	\begin{align*}
	\mP \cdot \begin{bmatrix}
	x^{k+1} - x^k\\y^{k+1} - y^k
	\end{bmatrix}
	=
	\begin{bmatrix}
	\alpha(x_g^k - x^{k+1})-(\nabla F(x_g^k) + y^{k+1})\\
	\mW\mW^\dagger \left[x^{k+1/2}- (1+\eta\alpha)^{-1}\eta(y^{k+1} - y^k)\right]
	\end{bmatrix}.
	\end{align*}
	Finally, from lines~\ref{alg:ALV:line:x:2} and~\ref{alg:ALV:line:x:3} of Algorithm~\ref{alg:ALV} it follows that
	\begin{equation}
		x^{k+1} = x^{k+1/2} + (1+\eta\alpha)^{-1}\eta(y^k - y^{k+1}),
	\end{equation}
	and hence,
	\begin{align*}
	\mP \cdot \begin{bmatrix}
	x^{k+1} - x^k\\y^{k+1} - y^k
	\end{bmatrix}
	=
	\begin{bmatrix}
	\alpha(x_g^k - x^{k+1})-(\nabla F(x_g^k) + y^{k+1})\\
	\mW\mW^\dagger x^{k+1}
	\end{bmatrix}.
	\end{align*}
\end{proof}

\begin{lemma}\label{ALV:lem:2}
	Let parameter $\eta$ be defined by
	\begin{equation}\label{ALV:eq:eta}
		\eta = \frac{1} {4\tau L}.
	\end{equation}
	Let parameter $\theta$ be defined by
	\begin{equation}\label{ALV:eq:theta}
		\theta = \frac{1}{\eta\lambda_{\max}(\mW)}.
	\end{equation}
	Let parameter $\alpha$ be defined by
	\begin{equation}\label{ALV:eq:alpha}
		\alpha = \mu.
	\end{equation}
	Let parameter $\tau$ be defined  by
	\begin{equation}\label{ALV:eq:tau}
		\tau = \min \left\{1,  \frac{1}{2}\sqrt{\frac{\mu}{L}\frac{\lambda_{\max}(\mW)}{\lambda_{\min}^+(\mW)}}\right\}.
	\end{equation}
	Let $\Psi^k$ be the following Lyapunov function:
	\begin{equation}\label{ALV:eq:Psi}
		\Psi^k = \sqN{\begin{bmatrix}
			x^{k} - x^*\\y^{k} - y^*
			\end{bmatrix}}_\mP
		+
		\frac{2(1-\tau)}{\tau}\bg_F(x_f^k, x^*),
	\end{equation}
	where $\mP$ is defined by \eqref{ALV:eq:P}:
	\begin{equation*}
	\mP = \begin{bmatrix}
	\frac{1}{\eta}\mI & 0\\
	0&\frac{1}{\theta}\mW^\dagger - (1+\eta\alpha)^{-1}\eta\mI
	\end{bmatrix}.\tag{\ref{ALV:eq:P}}
	\end{equation*}
	Then the following inequality holds:
	\begin{equation}
		\Psi^{k+1}
		\leq
		\left(1 + \frac{1}{4}\min\left\{\sqrt{\frac{\mu}{L}\frac{\lambda_{\min}^+(\mW)}{\lambda_{\max}(\mW)}},\frac{\lambda_{\min}^{+}(\mW)}{\lambda_{\max}(\mW)}\right\}\right)^{-1}
		\Psi^k.
	\end{equation}
\end{lemma}

\begin{proof}
	\begin{align*}
		\sqN{\begin{bmatrix}
				x^{k+1} - x^*\\y^{k+1} - y^*
			\end{bmatrix}}_\mP
		&=
		\sqN{\begin{bmatrix}
			x^{k} - x^*\\y^{k} - y^*
			\end{bmatrix}}_\mP
		-
		\sqN{\begin{bmatrix}
			x^{k+1} - x^k\\y^{k+1} - y^k
			\end{bmatrix}}_\mP
		+
		2\Dotprod{\mP\cdot\begin{bmatrix}
			x^{k+1} - x^k\\y^{k+1} - y^k
			\end{bmatrix}, \begin{bmatrix}
			x^{k+1} - x^*\\y^{k+1} - y^*
			\end{bmatrix}}
	\end{align*}
	Note, that stepsize $\eta$ defined by \eqref{ALV:eq:eta} and stepsize $\theta$ defined by \eqref{ALV:eq:theta} satisfy \eqref{ALV:eq:etatheta}, hence inequality \eqref{ALV:eq:Pbound} holds.
	Using \eqref{ALV:eq:Pbound} and \eqref{ALV:eq:xy} we get
	\begin{align*}
	\sqN{\begin{bmatrix}
		x^{k+1} - x^*\\y^{k+1} - y^*
		\end{bmatrix}}_\mP
	&\leq
	\sqN{\begin{bmatrix}
		x^{k} - x^*\\y^{k} - y^*
		\end{bmatrix}}_\mP
	-
	\frac{1}{\eta}\sqn{x^{k+1} - x^k}
	\\&+
	2\Dotprod{\begin{bmatrix}
		\alpha(x_g^k - x^{k+1})-(\nabla F(x_g^k) + y^{k+1})\\
		\mW\mW^\dagger x^{k+1}
		\end{bmatrix}, \begin{bmatrix}
		x^{k+1} - x^*\\y^{k+1} - y^*
		\end{bmatrix}}
	\\&=
	\sqN{\begin{bmatrix}
		x^{k} - x^*\\y^{k} - y^*
		\end{bmatrix}}_\mP
	-
	\frac{1}{\eta}\sqn{x^{k+1} - x^k}
	+
	2\alpha\<x_g^k - x^{k+1}, x^{k+1} - x^*>
	\\&-
	2\<\nabla F(x_g^k) + y^{k+1}, x^{k+1} - x^*>
	+
	2\<\mW\mW^\dagger x^{k+1}, y^{k+1} - y^*>.
	\end{align*}
	Since $\mW\mW^\dagger x^* = 0$ and $\mW\mW^\dagger(y^{k+1} - y^*) = y^{k+1} - y^*$, we get
	\begin{align*}
	\sqN{\begin{bmatrix}
		x^{k+1} - x^*\\y^{k+1} - y^*
		\end{bmatrix}}_\mP
	&\leq
	\sqN{\begin{bmatrix}
		x^{k} - x^*\\y^{k} - y^*
		\end{bmatrix}}_\mP
	-
	\frac{1}{\eta}\sqn{x^{k+1} - x^k}
	+
	2\alpha\<x_g^k - x^{k+1}, x^{k+1} - x^*>
	\\&-
	2\<\nabla F(x_g^k) + y^{k+1}, x^{k+1} - x^*>
	+
	2\<x^{k+1} - x^*, y^{k+1} - y^*>.
	\end{align*}
	Since $\nabla F(x^*) + y^* = 0$ (optimality condition \eqref{eq:primaldualoptimal}), we get
	\begin{align*}
	\sqN{\begin{bmatrix}
		x^{k+1} - x^*\\y^{k+1} - y^*
		\end{bmatrix}}_\mP
	&\leq
	\sqN{\begin{bmatrix}
		x^{k} - x^*\\y^{k} - y^*
		\end{bmatrix}}_\mP
	-
	\frac{1}{\eta}\sqn{x^{k+1} - x^k}
	+
	2\alpha\<x_g^k - x^{k+1}, x^{k+1} - x^*>
	\\&-
	2\<\nabla F(x_g^k) - \nabla F(x^*)+ y^{k+1} - y^*, x^{k+1} - x^*>
	+
	2\<x^{k+1} - x^*, y^{k+1} - y^*>
	\\&=
	\sqN{\begin{bmatrix}
		x^{k} - x^*\\y^{k} - y^*
		\end{bmatrix}}_\mP
	-
	\frac{1}{\eta}\sqn{x^{k+1} - x^k}
	-
	2\alpha\sqn{x^{k+1} - x^*}
	\\&-
	2\alpha\<x_g^k -  x^*, x^{k+1} - x^*>
	-
	2\<\nabla F(x_g^k) - \nabla F(x^*), x^{k+1} - x^*>.
	\end{align*}
	Using Young's  inequality $2\<a,b> \leq \sqn{a} + \sqn{b}$ we get
	\begin{align*}
	\sqN{\begin{bmatrix}
		x^{k+1} - x^*\\y^{k+1} - y^*
		\end{bmatrix}}_\mP
	&\leq
	\sqN{\begin{bmatrix}
		x^{k} - x^*\\y^{k} - y^*
		\end{bmatrix}}_\mP
	-
	\frac{1}{\eta}\sqn{x^{k+1} - x^k}
	-
	2\alpha\sqn{x^{k+1} - x^*}
	\\&+
	\alpha\sqn{x_g^k - x^*}
	+
	\alpha\sqn{x^{k+1} - x^*}
	-
	2\<\nabla F(x_g^k) - \nabla F(x^*), x^{k+1} - x^*>
	\\&=
	\sqN{\begin{bmatrix}
		x^{k} - x^*\\y^{k} - y^*
		\end{bmatrix}}_\mP
	-
	\frac{1}{\eta}\sqn{x^{k+1} - x^k}
	-
	\alpha\sqn{x^{k+1} - x^*}
	+
	\alpha\sqn{x_g^k - x^*}
	\\&-
	2\<\nabla F(x_g^k) - \nabla F(x^*), x^{k+1} - x^*>.
	\end{align*}
	Now, we use lines~\ref{alg:ALV:line:x:1} and~\ref{alg:ALV:line:x:4} of Algorithm~\ref{alg:ALV} and get
	\begin{align*}
	\sqN{\begin{bmatrix}
		x^{k+1} - x^*\\y^{k+1} - y^*
		\end{bmatrix}}_\mP
	&\leq
	\sqN{\begin{bmatrix}
		x^{k} - x^*\\y^{k} - y^*
		\end{bmatrix}}_\mP
	-
	\alpha\sqn{x^{k+1} - x^*}
	+
	\alpha\sqn{x_g^k - x^*}
	-
	\frac{1}{2\eta}\sqn{x^{k+1} - x^k}
	\\&-
	\frac{2-\tau}{\tau}
	\left(
			\<\nabla F(x_g^k) - \nabla F(x^*), x_f^{k+1} - x_g^k>
			+
			\frac{1}{2\eta}\frac{(2-\tau)}{4\tau}\sqn{x_f^{k+1} - x_g^k}
	\right)
	\\&-
	2\<\nabla F(x_g^k) - \nabla F(x^*), x_g^k - x^*>
	+
	\frac{2(1-\tau)}{\tau}\<\nabla F(x_g^k) - \nabla F(x^*), x_f^k - x_g^k>.
	\end{align*}
	Since parameter $\eta$ defined by \eqref{ALV:eq:eta} satisfy $\eta \leq \frac{2-\tau}{4\tau L}$, we get
	\begin{align*}
	\sqN{\begin{bmatrix}
		x^{k+1} - x^*\\y^{k+1} - y^*
		\end{bmatrix}}_\mP
	&\leq
	\sqN{\begin{bmatrix}
		x^{k} - x^*\\y^{k} - y^*
		\end{bmatrix}}_\mP
	-
	\alpha\sqn{x^{k+1} - x^*}
	+
	\alpha\sqn{x_g^k - x^*}
	-
	\frac{1}{2\eta}\sqn{x^{k+1} - x^k}
	\\&-
	\frac{2-\tau}{\tau}
	\left(
	\<\nabla F(x_g^k) - \nabla F(x^*), x_f^{k+1} - x_g^k>
	+
	\frac{L}{2}\sqn{x_f^{k+1} - x_g^k}
	\right)
	\\&-
	2\<\nabla F(x_g^k) - \nabla F(x^*), x_g^k - x^*>
	+
	\frac{2(1-\tau)}{\tau}\<\nabla F(x_g^k) - \nabla F(x^*), x_f^k - x_g^k>.
	\end{align*}
	Using $\mu$-strong convexity and $L$-smoothness of $f(x)$ we get
	\begin{align*}
	\sqN{\begin{bmatrix}
		x^{k+1} - x^*\\y^{k+1} - y^*
		\end{bmatrix}}_\mP
	&\leq
	\sqN{\begin{bmatrix}
		x^{k} - x^*\\y^{k} - y^*
		\end{bmatrix}}_\mP
	-
	\alpha\sqn{x^{k+1} - x^*}
	+
	\alpha\sqn{x_g^k - x^*}
	-
	\frac{1}{2\eta}\sqn{x^{k+1} - x^k}
	\\&-
	\frac{2-\tau}{\tau}
	\left(
		\bg_F (x_f^{k+1},x^*) - \bg_F (x_g^k,x^*)
	\right)
	+
	\frac{2(1-\tau)}{\tau}\left( \bg_F(x_f^k, x^*) - \bg_F(x_g^k,x^*) \right)
	\\&-
	2\left(\bg_F(x_g^k,x^*) + \frac{\mu}{2}\sqn{x_g^k - x^*}\right)
	\\&=
	\sqN{\begin{bmatrix}
		x^{k} - x^*\\y^{k} - y^*
		\end{bmatrix}}_\mP
	-
	\alpha\sqn{x^{k+1} - x^*}
	+
	\frac{2(1-\tau)}{\tau}\bg_F(x_f^k, x^*)
	-
	\frac{2-\tau}{\tau}\bg_F (x_f^{k+1},x^*) 
	\\&+
	(\alpha - \mu)\sqn{x_g^k - x^*}
	-
	\frac{1}{2\eta}\sqn{x^{k+1} - x^k}
	-
	\bg_F(x_g^k,x^*).
	\end{align*}
	Now, we define $\delta = \min \left\{1, \frac{1}{2\eta L}\right\}$. Since $\alpha$ defined by \eqref{ALV:eq:alpha} satisfies conditions of Lemma~\ref{ALV:lem:1}, we can use \eqref{ALV:eq:1} and get
	\begin{align*}
	\sqN{\begin{bmatrix}
		x^{k+1} - x^*\\y^{k+1} - y^*
		\end{bmatrix}}_\mP
	&\leq
	\sqN{\begin{bmatrix}
		x^{k} - x^*\\y^{k} - y^*
		\end{bmatrix}}_\mP
	-
	\alpha\sqn{x^{k+1} - x^*}
	+
	\frac{2(1-\tau)}{\tau}\bg_F(x_f^k, x^*)
	-
	\frac{2-\tau}{\tau}\bg_F (x_f^{k+1},x^*) 
	\\&+
	(\alpha - \mu)\sqn{x_g^k - x^*}
	-
	\frac{\delta}{2\eta}\sqn{x^{k+1} - x^k}
	-
	\bg_F(x_g^k,x^*)
	\\&\leq
	\sqN{\begin{bmatrix}
		x^{k} - x^*\\y^{k} - y^*
		\end{bmatrix}}_\mP
	-
	\alpha\sqn{x^{k+1} - x^*}
	+
	\frac{2(1-\tau)}{\tau}\bg_F(x_f^k, x^*)
	-
	\frac{2-\tau}{\tau}\bg_F (x_f^{k+1},x^*) 
	\\&-
	\frac{\eta\delta}{4}\sqn{y^{k+1} - y^*}
	+
	\eta\alpha^2\delta\sqn{x^{k+1} - x^*}
	+
	2\eta  L \delta\bg_{f }(x_g^k, x^*)
	\\&+
	(\alpha - \mu)\sqn{x_g^k - x^*}
	-
	\bg_F(x_g^k,x^*)
	\\&\leq
	\sqN{\begin{bmatrix}
		x^{k} - x^*\\y^{k} - y^*
		\end{bmatrix}}_\mP
	-
	\alpha\sqn{x^{k+1} - x^*}
	+
	\frac{2(1-\tau)}{\tau}\bg_F(x_f^k, x^*)
	-
	\frac{2-\tau}{\tau}\bg_F (x_f^{k+1},x^*) 
	\\&-
	\frac{\eta\delta}{4}\sqn{y^{k+1} - y^*}
	+
	\frac{\alpha^2}{2L}\sqn{x^{k+1} - x^*}
	+
	(\alpha - \mu)\sqn{x_g^k - x^*}
	\\&=
	\sqN{\begin{bmatrix}
		x^{k} - x^*\\y^{k} - y^*
		\end{bmatrix}}_\mP
	-
	\left(
		\alpha - \frac{\alpha^2}{2L}
	\right)\sqn{x^{k+1} - x^*}
	-
	\frac{\eta\delta}{4}\sqn{y^{k+1} - y^*}
	\\&+
	\frac{2(1-\tau)}{\tau}\bg_F(x_f^k, x^*)
	-
	\frac{2-\tau}{\tau}\bg_F (x_f^{k+1},x^*) 
	+
	(\alpha - \mu)\sqn{x_g^k - x^*}.
	\end{align*}
	Using parameter $\alpha$ defined by \eqref{ALV:eq:alpha} we get
	\begin{align*}
	\sqN{\begin{bmatrix}
		x^{k+1} - x^*\\y^{k+1} - y^*
		\end{bmatrix}}_\mP
	&\leq
	\sqN{\begin{bmatrix}
		x^{k} - x^*\\y^{k} - y^*
		\end{bmatrix}}_\mP
	-
	\frac{\mu}{2}\sqn{x^{k+1} - x^*}
	-
	\frac{\eta\delta}{4}\sqn{y^{k+1} - y^*}
	\\&+
	\frac{2(1-\tau)}{\tau}\bg_F(x_f^k, x^*)
	-
	\frac{2-\tau}{\tau}\bg_F (x_f^{k+1},x^*) .
	\end{align*}
	Since $y^k,y^*\in \range \mW$, we get
	\begin{align*}
	\sqN{\begin{bmatrix}
		x^{k+1} - x^*\\y^{k+1} - y^*
		\end{bmatrix}}_\mP
	&\leq
	\sqN{\begin{bmatrix}
		x^{k} - x^*\\y^{k} - y^*
		\end{bmatrix}}_\mP
	-
	\frac{\mu}{2}\sqn{x^{k+1} - x^*}
	-
	\frac{\eta\delta\lambda_{\min}^{+}(\mW)}{4}\sqnw{y^{k+1} - y^*}
	\\&+
	\frac{2(1-\tau)}{\tau}\bg_F(x_f^k, x^*)
	-
	\frac{2-\tau}{\tau}\bg_F (x_f^{k+1},x^*).
	\end{align*}
	Using \eqref{ALV:eq:Pbound} we get
	\begin{align*}
	\sqN{\begin{bmatrix}
		x^{k+1} - x^*\\y^{k+1} - y^*
		\end{bmatrix}}_\mP
	&\leq
	\sqN{\begin{bmatrix}
		x^{k} - x^*\\y^{k} - y^*
		\end{bmatrix}}_\mP
	-
	\min\left\{\frac{\eta\mu}{2}, \frac{\eta\theta\delta\lambda_{\min}^{+}(\mW)}{4}\right\}
	\sqN{\begin{bmatrix}
		x^{k+1} - x^*\\y^{k+1} - y^*
		\end{bmatrix}}_\mP
	\\&+
	\frac{2(1-\tau)}{\tau}\bg_F(x_f^k, x^*)
	-
	\frac{2-\tau}{\tau}\bg_F (x_f^{k+1},x^*) .
	\end{align*}
	Using parameter $\theta$ defined by \eqref{ALV:eq:theta} and definition of $\delta$ we get
	\begin{align*}
	\sqN{\begin{bmatrix}
		x^{k+1} - x^*\\y^{k+1} - y^*
		\end{bmatrix}}_\mP
	&\leq
	\sqN{\begin{bmatrix}
	x^{k} - x^*\\y^{k} - y^*
	\end{bmatrix}}_\mP
	-
	 \min\left\{\frac{\eta\mu}{2},\frac{\lambda_{\min}^{+}(\mW)}{4\lambda_{\max}(\mW)}, \frac{\lambda_{\min}^{+}(\mW)}{8\eta L\lambda_{\max}(\mW)} \right\}
	\sqN{\begin{bmatrix}
		x^{k+1} - x^*\\y^{k+1} - y^*
		\end{bmatrix}}_\mP
	\\&+
	\frac{2(1-\tau)}{\tau}\bg_F(x_f^k, x^*)
	-
	\frac{2-\tau}{\tau}\bg_F (x_f^{k+1},x^*).
	\end{align*}
	Plugging parameter $\eta$ defined by \eqref{ALV:eq:eta} we get
	\begin{align*}
	\sqN{\begin{bmatrix}
		x^{k+1} - x^*\\y^{k+1} - y^*
		\end{bmatrix}}_\mP
	&\leq
	\sqN{\begin{bmatrix}
		x^{k} - x^*\\y^{k} - y^*
		\end{bmatrix}}_\mP
	-
	\min\left\{\frac{\mu}{8\tau L},
	\frac{\lambda_{\min}^{+}(\mW)}{4\lambda_{\max}(\mW)}, \frac{\tau\lambda_{\min}^{+}(\mW)}{2\lambda_{\max}(\mW)} \right\}
	\sqN{\begin{bmatrix}
		x^{k+1} - x^*\\y^{k+1} - y^*
		\end{bmatrix}}_\mP
	\\&+
	\frac{2(1-\tau)}{\tau}\bg_F(x_f^k, x^*)
	-
	\frac{2-\tau}{\tau}\bg_F (x_f^{k+1},x^*)
	\\&\leq
	\sqN{\begin{bmatrix}
		x^{k} - x^*\\y^{k} - y^*
		\end{bmatrix}}_\mP
	-
	\min\left\{\frac{\mu}{8\tau L},
	\frac{\lambda_{\min}^{+}(\mW)}{4\lambda_{\max}(\mW)}, \frac{\tau\lambda_{\min}^{+}(\mW)}{2\lambda_{\max}(\mW)} \right\}
	\sqN{\begin{bmatrix}
		x^{k+1} - x^*\\y^{k+1} - y^*
		\end{bmatrix}}_\mP
	\\&+
	\frac{2(1-\tau)}{\tau}\bg_F(x_f^k, x^*)
	-
	\left(1 + \frac{\tau}{2}\right)\frac{2(1-\tau)}{\tau}\bg_F (x_f^{k+1},x^*).
	\end{align*}
	After rearranging and using definition of $\Psi^k$ \eqref{ALV:eq:Psi} we get
	\begin{align*}
		\Psi^k
		&\geq
		\left(
		1 + \min\left\{\frac{\tau}{2},\frac{\mu}{8\tau L},
		\frac{\lambda_{\min}^{+}(\mW)}{4\lambda_{\max}(\mW)}, \frac{\tau\lambda_{\min}^{+}(\mW)}{2\lambda_{\max}(\mW)} \right\}
		\right) \Psi^{k+1}.
	\end{align*}
	Plugging parameter $\tau$ defined by \eqref{ALV:eq:tau} we get
	\begin{align*}
	\Psi^{k}
	&\geq
	\left(1 + \frac{1}{4}\min\left\{\sqrt{\frac{\mu}{L}\frac{\lambda_{\min}^+(\mW)}{\lambda_{\max}(\mW)}},\frac{\lambda_{\min}^{+}(\mW)}{\lambda_{\max}(\mW)}\right\}\right)
	\Psi^{k+1}.
	\end{align*}
\end{proof}

\begin{proof}[Proof of Theorem~\ref{th:ALV} (APAPC)]
	Conditions of Lemma~\ref{ALV:lem:2} are satisfied, hence the following inequality holds for all $k$:
	\begin{equation}
		\Psi^{k+1}
		\leq
		\left(1 + \frac{1}{4}\min\left\{\sqrt{\frac{\mu}{L}\frac{\lambda_{\min}^+(\mW)}{\lambda_{\max}(\mW)}},\frac{\lambda_{\min}^{+}(\mW)}{\lambda_{\max}(\mW)}\right\}\right)^{-1}
		\Psi^k.
	\end{equation}
	After doing telescoping we get
	\begin{equation}
		\Psi^k \leq \left(1 + \frac{1}{4}\min\left\{\sqrt{\frac{\mu}{L}\frac{\lambda_{\min}^+(\mW)}{\lambda_{\max}(\mW)}},\frac{\lambda_{\min}^{+}(\mW)}{\lambda_{\max}(\mW)}\right\}\right)^{-k}
		\Psi^0.
	\end{equation}
	Inequality \eqref{ALV:eq:Pbound} implies $\Psi^0 \leq C$, where $C \eqdef \frac{1}{\eta}\sqN{x^0 - x^*} + \frac{1}{\theta}\sqn{y^0 - y^*}_{\mW^\dagger} +
	\frac{2(1-\tau)}{\tau}\bg_F(x_f^0, x^*).$
	Hence, we obtain.
	\begin{equation}
	\Psi^k \leq \left(1 + \frac{1}{4}\min\left\{\sqrt{\frac{\mu}{L}\frac{\lambda_{\min}^+(\mW)}{\lambda_{\max}(\mW)}},\frac{\lambda_{\min}^{+}(\mW)}{\lambda_{\max}(\mW)}\right\}\right)^{-k}
	C.
	\end{equation}
	It remains to lower bound $\Psi^k$ using \eqref{ALV:eq:Pbound} one more time:
	\begin{equation}
		\frac{1}{\eta}\sqN{x^k - x^*} +
		\frac{2(1-\tau)}{\tau}\bg_F(x_f^k, x^*)
		\leq
		\Psi^k
		\leq
		\left(1 + \frac{1}{4}\min\left\{\sqrt{\frac{\mu}{L}\frac{\lambda_{\min}^+(\mW)}{\lambda_{\max}(\mW)}},\frac{\lambda_{\min}^{+}(\mW)}{\lambda_{\max}(\mW)}\right\}\right)^{-k}
		C.
	\end{equation}
	Finally, choosing number of iterations
	\begin{equation}
	k \geq \left(1+ 4\max\left\{\sqrt{\frac{L \lambda_{\max}(\mW)}{\mu \lambda_{\min}^+(\mW)}}, \frac{\lambda_{\max}(\mW)}{\lambda_{\min}^+(\mW)}\right\}\right)
	\log \left(\frac{\eta C}{\varepsilon}\right).
	\end{equation}
	implies $\sqn{x^k - x^*} \leq \varepsilon$.
\end{proof}

\newpage

\section{Proof of Corollary~\ref{th:ALV-opt-cor} (OPAPC)}
			
		First, Theorem~\ref{th:ALV} still holds true by replacing $\lambda_{\max}(\mW)$ by an upper bound $\lambda_1$, $\lambda_{\min}^+(\mW)$ by a lower bound $\lambda_2 >0$ and $\chi(\mW)$ by the upper bound $\chi = \nicefrac{\lambda_1}{\lambda_2}$.\footnote{The proof is the same by replacing $\lambda_{\max}(\mW)$ by $\lambda_1$ and $\lambda_{\min}^+(\mW)$ by $\lambda_2$.}

		The proof of Corollary~\ref{th:ALV-opt-cor} is similar to the proof of Theorem 4 of~\cite{scaman2017optimal}.

		Denote $\tilde{\mW} = \frac{2\chi(\mW)}{(1+\chi(\mW))\lambda_{\max}(\mW)} \mW$. Let $I$ be the interval $I = [1-\frac{1}{c_2},1+\frac{1}{c_2}] \subset (0,2)$. Then, $\Sp(\tilde{\mW}) \setminus \{0\} \subset I$, where $\Sp$ denotes the spectrum. Moreover, using~\cite{scaman2017optimal}, the polynomial $P$ satisfies $P(0) = 0$ and $\max_{t \in I}|1-P(t)| = \frac{2c_1^T}{1+c_1^{2T}} < 1$. Therefore, 
		\begin{equation}
			\Sp(I - P(\tilde{\mW})) \setminus \{1\} \subset \left[-\frac{2c_1^T}{1+c_1^{2T}},\frac{2c_1^T}{1+c_1^{2T}}\right] \subset (-1,1).
		\end{equation}
		Consequently, 
		\begin{equation}
			\lambda_{\max}(P(\tilde{\mW})) \leq \lambda_1 \eqdef 1+\frac{2c_1^T}{1+c_1^{2T}} < 2, \quad \lambda_{\min}^+(P(\tilde{\mW})) \geq \lambda_2 \eqdef1-\frac{2c_1^T}{1+c_1^{2T}} > 0.
		\end{equation}
		Moreover, by replacing $c_1$ and $T$ by their values, $\chi \eqdef \frac{\lambda_1}{\lambda_2} \leq 4$, see~\cite[Equation 34]{scaman2017optimal}.

		Applying APAPC with the gossip matrix $P(\tilde{\mW})$ leads to OPAPC. Then, we apply Theorem~\ref{th:ALV} to OPAPC. More precisely, we apply Theorem~\ref{th:ALV} by replacing $\mW$ by $P(\tilde{\mW})$ and $\lambda_{\max}(\mW)$ (resp. $\lambda_{\min}^+(\mW)$) by the upper bound (resp. the lower bound) $\lambda_1$ (resp. $\lambda_2$) of $\lambda_{\max}(P(\tilde{\mW}))$ (resp. $\lambda_{\min}^+(P(\tilde{\mW}))$). Denoting $x^k$ the iterates of OPAPC, we obtain 
		\begin{equation}
			\mytextstyle \frac{1}{\eta}\sqN{x^k - x^*} +
			\frac{2-\tau}{\tau}\bg_F(x_f^k, x^*)
			\leq
			\left(1 + \frac{1}{16}\min\left\{\frac{2}{\sqrt{\kappa}},1\right\}\right)^{-k}
			C.
		\end{equation}
		Finally, the gradient computation complexity of OPAPC is $\cO(\sqrt{\kappa}\log(1/\varepsilon))$. One multiplication by $P(\tilde{\mW})$ is equivalent to one application of the procedure \textsc{AcceleratedGossip}$(\mW,\cdot,T)$, which requires exactly $T$ communication rounds. Therefore, the communication complexity of OPAPC is $T\cO(\sqrt{\kappa}\log(1/\varepsilon)) = \cO(\sqrt{\kappa\chi(\mW)}\log(1/\varepsilon))$.


\newpage

\section{A Loopless Algorithm Optimal in Communication Complexity}

	
We propose another accelerated Forward Backward algorithm to solve Problem~\eqref{eq:2}. More precisely, we first provide a reformulation of Problem~\eqref{eq:2}, different from the reformulation~\eqref{eq:primaldualoptimal}. Then, we design an accelerated Forward Backward algorithm associated with this reformulation. Remarkably, the matrix $\mW$ is only involved in the operator $A$ of this new Forward Backward algorithm. This leads to an acceleration compared to APAPC, and to an optimal communication complexity.

In this section, $\sX$ is the Euclidean space $\sX = (\R^{d})^{\cV} \times (\R^{d})^{\cV} \times \range(\mW)$ endowed with the norm $\sqn{(x,y,z)}_{\sX} \eqdef \sqn{x} + \sqn{y} + \sqn{z}_{\mW^{\dagger}}$.

Using the first order optimality conditions, a point $x^*$ is a solution to Problem~\eqref{eq:2} if and only if $\nabla F(x^*) \in \range(\mW)$ and $x^* \in \ker(\mW)$. Solving Problem~\eqref{eq:2} is therefore equivalent to finding $(x^*, y^*,z^*)\in \sX$ such that
\begin{align}
	0 &= \nabla F(x^*) - \frac{\mu}{2}x^* - y^*,\label{opt:4}\\
	0 &= x^* + \frac{2}{\mu}(y^* + z^*),\label{opt:5}\\
	0 &= \frac{2}{\mu}\mW(y^* + z^*).\label{opt:6}
\end{align}
Indeed, if~\eqref{opt:4}--\eqref{opt:6} holds, then using~\eqref{opt:4}, $y^* = \nabla F(x^*) - \frac{\mu}{2}x^*$ and using~\eqref{opt:5} $z^* = -\nabla F(x^*) \in \range(\mW)$. Since $z^* \in \range(\mW)$ and $y^* + z^* = - \frac{\mu}{2}x^*$, we have $\nabla F(x^*) \in \range(\mW)$ and $x^* \in \ker(\mW)$. On the other hand, if $\nabla F(x^*) \in \range(\mW)$ and $x^* \in \ker(\mW)$, then $\mW x^* = 0$ and setting $y^* = \nabla F(x^*) - \frac{\mu}{2}x^*$ and $z^* = -\nabla F(x^*) \in \range(\mW)$ leads to~\eqref{opt:4}--\eqref{opt:6}. 


Consider the map $M : \sX \to \sX$ 
\begin{equation}
M(x,y,z) \eqdef \begin{bmatrix} \nabla F(x) - \frac{\mu}{2}x&  - y &  \\ x & +\frac{2}{\mu}y & +\frac{2}{\mu}z \\  & \frac{2}{\mu} \mW y & +\frac{2}{\mu} \mW z \end{bmatrix}.
\end{equation}
Similary to Section~\ref{sec:ALV}, one can show that $M$ is a monotone operator. Moreover, $M(x^*,y^*,z^*) = 0$, i.e., $(x^*,y^*,z^*)$ is a zero of $M$. 

Consider the maps $A,B : \sX \to \sX$ defined by
\begin{equation}
	A(x,y,z) = \begin{bmatrix*}[c]
		\nabla F(x) - \frac{\mu}{2}x\\
		\frac{2}{\mu}(y + z)  + \nu y\\
		\frac{2}{\mu}\mW(y + z)
	\end{bmatrix*}, \quad
	B(x,y,z) = \begin{bmatrix*}[c]
	-y\\
	x- \nu y\\
	0
	\end{bmatrix*}.
\end{equation}
Then, $M=A+B$. Note that there is a term $\nu y$, where $\nu >0$ in $A(x,y,z)$ and a term $-\nu y$ in $B(x,y,z)$, which cancel out in the sum $A(x,y,z) + B(x,y,z)$. This additional term makes the operator $A(x,y,z)$ strongly monotone. Indeed, $A$ is the gradient of the strongly convex function (in $\sX$) $\sX \ni (x,y,z) \mapsto r(x) + h(y,z)$ defined by
\begin{equation}
	r(x) \eqdef F(x) - \tfrac{\mu}{4}\sqn{x}, \quad
	h(y,z) \eqdef \frac{1}{\mu}\sqn{y+z} + \frac{\nu}{2}\sqn{y}.
\end{equation}
In other words, operator $A(x,y,z)$ can be written as
\begin{equation}
	A(x,y,z) = \begin{bmatrix}
		\nabla r(x)\\\nabla_y h(y,z)\\\mW \nabla_z h(y,z)
	\end{bmatrix},
\end{equation}
and one can check that $A$ is strongly monotone. However, the operator $B(x,y,z)$ is not monotone in general. Indeed, $B$ is only weakly monotone since $B$ satisfies
\begin{equation}
\Dotprod{B(x,y,z) - B(x^*,y^*,z^*), \begin{bmatrix}
	x-x^*\\y-y^*\\z-z^*
	\end{bmatrix}}_{\sX}= -\nu \sqn{y- y^*}.
\end{equation}

One idea to solve~\eqref{opt:4}--\eqref{opt:6} is to apply Algorithm~\eqref{eq:FB} to the sum $A+B$, although $B$ is not monotone. Note that $B$ is linear and, although $B$ is not monotone, the resolvent of $B$ is still well defined while $1 - \gamma\nu + \gamma^2 \neq 0$. Indeed, $(x',y') = J_{\gamma B}(x,y)$ implies $x' = x + \gamma y'$, and $(1 - \gamma\nu + \gamma^2)y' = y - \gamma x$.

In particular, we propose a new algorithm that can be seen as an accelerated version of the Forward Backward Algorithm~\eqref{eq:FB} to find a zero of $A+B$. The proposed algorithm is defined in Algorithm~\ref{alg:scary} and its complexity is given in Theorem~\ref{th:scary}. We show that the complexity of Algorithm~\ref{alg:scary} is $\cO(\sqrt{\kappa \chi(\mW)}\log(1/\varepsilon))$, both in communication rounds and gradient computations. The proposed algorithm is therefore optimal in communication complexity, see Section~\ref{sec:lower-bound}. Moreover, Algorithm~\ref{alg:scary} uses only one gradient computation by communication round.


\begin{algorithm}
	\caption{}
	\label{alg:scary}
	\begin{algorithmic}[1]
		\State {\bf Parameters:}  $x^0 ,y^0\in \R^{nd},  z^0 \in  \range \mW$, $\eta,\theta ,\lambda,\alpha,\beta,\gamma, \nu>0, \tau, \sigma\in (0,1)$
		\State Set $x_f^0 = x^0$
		\State Set $y_f^0 = y^0$
		\State Set $z_f^0 = z^0$
		\For{$k=0,1,2,\ldots$}{}
		\State $x_g^k = \tau x^k + (1-\tau)x_f^k$\label{alg:scary:line:x:1}
		\State $y_g^k= \sigma y^k  + (1-\sigma)y_f^k$\label{alg:scary:line:y:1}
		\State $z_g^k= \sigma z^k  + (1-\sigma)z_f^k$\label{alg:scary:line:z:1}
		\State $x^{k+1} = x^k + \eta\alpha  (x_g^k - x^{k+1}) - \eta \nabla r(x_g^k) + \eta y^{k+1}$\label{alg:scary:line:x:2}
		\State $y^{k+1} = y^k + \theta\beta ( y_g^k - y^{k+1}) - \theta \nabla_y h(y_g^k,z_g^k) + \theta\nu y^{k+1} - \theta x^{k+1}$\label{alg:scary:line:y:2}
		\State $z^{k+1} = z^k + \lambda\gamma (z_g^k - z^{k+1}) - \lambda\mW\nabla_z h(y_g^k,z_g^k)$\label{alg:scary:line:z:2}
		\State $x_f^{k+1} = x_g^k + \tfrac{2\tau}{2-\tau}(x^{k+1} - x^k)$\label{alg:scary:line:x:3}
		\State $y_f^{k+1} = y_g^k+ \sigma (y^{k+1}-y^k) $\label{alg:scary:line:y:3}
		\State $z_f^{k+1} = z_g^k+ \sigma (z^{k+1}-z^k) $\label{alg:scary:line:z:3}
		\EndFor
	\end{algorithmic}
\end{algorithm}

\begin{theorem}[Algorithm~\ref{alg:scary}]
	\label{th:scary}
	Set the parameters $\eta,\theta ,\lambda,\alpha,\beta,\gamma, \nu>0, \tau, \sigma\in (0,1)$ to
	\begin{align}
		\eta &= \left[2\sqrt{L\mu}+ \mu\right]^{-1},
		&
		\alpha &= \frac{\mu}{3},
		&
		\tau &=\frac{1}{2}\sqrt{\frac{\mu}{L}},
		\\
		\theta &= \left[\frac{1}{4}\sqrt{\frac{\lambda_{\min}^+(\mW)}{\lambda_{\max}(\mW)\mu L}}+\frac{5}{96L}\right]^{-1},
		&
		\beta &= \frac{1}{96L},
		&
		\sigma &=\frac{1}{20}\sqrt{\frac{\lambda_{\min}^+(\mW)}{\lambda_{\max}(\mW)}\frac{\mu}{L}},
		\\
		\lambda &= \left[\frac{1}{4}\sqrt{\frac{\lambda_{\min}^+(\mW)\lambda_{\max}(\mW)}{\mu L}}+\frac{\lambda_{\min}^+(\mW)}{96L} \right]^{-1},
		&
		\gamma &= \frac{\lambda_{\min}^+(\mW)}{96L},
		&
		\nu &= \frac{1}{24L}.
		\end{align}
	Then, the sequence $(x^k)$ converges linearly to $x^*$. Moreover, for every $\varepsilon >0$, Algorithm~\ref{alg:scary} finds $x^k$ for which $\sqn{x^k - x^*} \leq \varepsilon$ in at most $\cO\left(\sqrt{\kappa \chi(\mW)}\log(1/\varepsilon)\right)$ gradient computations (resp.\ communication rounds).
\end{theorem}
The Algorithm~\ref{alg:scary} achieves the communication lower bound of Theorem~\ref{th:lower}.
The proof of Theorem~\ref{th:scary} intuitively relies on viewing Algorithm~\ref{alg:scary} as an accelerated version of~\eqref{eq:FB}, although Nesterov's acceleration does not apply to general monotone operators and even less to non monotone operators.

\newpage

\section{Proof of Theorem~\ref{th:scary} (Algorithm~\ref{alg:scary})}

\begin{lemma}\label{scary:lem:1}
	Let $\alpha$ satisfy
	\begin{equation}\label{scary:alpha}
	\alpha \leq \frac{\mu}{2}.
	\end{equation}
	Let $\delta$ be defined by
	\begin{equation}\label{scary:delta}
	\delta = \min\left\{1,\frac{1}{2\eta L}\right\}.
	\end{equation}
	Then the following inequality holds:
	\begin{equation}
	-\frac{1}{2\eta}\sqn{x^{k+1} - x^*}
	\leq
	-\frac{\eta\delta}{4}\sqn{y^{k+1} - y^*} + \frac{\alpha}{4}\sqn{x^{k+1} - x^*}
	+
	\bg_{r}(x_g^k,x^*).
	\end{equation}
\end{lemma}
\begin{proof}
	From line~\ref{alg:scary:line:x:2} of Algortihm~\ref{alg:scary} it follows that
	\begin{align*}
	x^{k+1}-x^k = \eta\alpha(x_g^k - x^{k+1}) - \eta \nabla r(x_g^k) + \eta y^{k+1}.
	\end{align*}
	From optimality condition \eqref{opt:4} it follows that $\nabla r(x^*) = y^*$ and hence
	\begin{align*}
	\sqn{x^{k+1} - x^k}
	&=
	\eta^2\sqn{y^{k+1} - y^* - \alpha(x^{k+1} - x^*) - (\nabla r(x_g^k) - \nabla r(x^*)- \alpha (x_g^k - x^*))}
	\\&\geq
	\frac{\eta^2}{2}\sqn{y^{k+1} - y^*} - \eta^2\sqn{\alpha(x^{k+1} - x^*) + (\nabla r(x_g^k) - \nabla r(x^*)- \alpha (x_g^k - x^*))}
	\\&\geq
	\frac{\eta^2}{2}\sqn{y^{k+1} - y^*} - 2\eta^2\alpha^2\sqn{x^{k+1} - x^*}
	\\&-
	2\eta^2\sqn{\nabla r(x_g^k) - \nabla r(x^*)- \alpha (x_g^k - x^*)}.
	\end{align*}
	From \eqref{scary:alpha} it follows that function $r(x) - \frac{\alpha}{2}\sqn{x} = F(x) - \frac{\mu+2\alpha}{4}\sqn{x}$ is convex and $L$-smooth, hence we can bound the last term:
	\begin{align*}
	\sqn{x^{k+1} - x^k}
	&\geq
	\frac{\eta^2}{2}\sqn{y^{k+1} - y^*} - 2\eta^2\alpha^2\sqn{x^{k+1} - x^*}
	-
	4\eta^2L \bg_{r(\cdot)- \frac{\alpha}{2}\sqn{\cdot}}(x_g^k,x^*)
	\\&=
	\frac{\eta^2}{2}\sqn{y^{k+1} - y^*} - 2\eta^2\alpha^2\sqn{x^{k+1} - x^*}
	-
	4\eta^2L \bg_{r}(x_g^k,x^*)
	+
	2\eta^2 L\alpha\sqn{x_g^k - x^*}
	\\&\geq
	\frac{\eta^2}{2}\sqn{y^{k+1} - y^*} - 2\eta^2\alpha^2\sqn{x^{k+1} - x^*}
	-
	4\eta^2L \bg_{r}(x_g^k,x^*).
	\end{align*}
	Multiplying by $\frac{1}{2\eta}$ and rearranging gives
	\begin{align*}
	-\frac{1}{2\eta}\sqn{x^{k+1} - x^*}
	\leq
	-\frac{\eta}{4}\sqn{y^{k+1} - y^*} + \eta\alpha^2\sqn{x^{k+1} - x^*}
	+
	2\eta L \bg_{r}(x_g^k,x^*).
	\end{align*}
	Using  $\delta$ defined by \eqref{scary:delta} we obtain
	\begin{align*}
	-\frac{1}{2\eta}\sqn{x^{k+1} - x^*}
	&\leq
	-\frac{\delta}{2\eta}\sqn{x^{k+1} - x^*}
	\\&\leq
	-\frac{\eta\delta}{4}\sqn{y^{k+1} - y^*} + \delta\eta\alpha^2\sqn{x^{k+1} - x^*}
	+
	2\delta\eta L \bg_{r}(x_g^k,x^*)
	\\&\leq
	-\frac{\eta\delta}{4}\sqn{y^{k+1} - y^*} + \frac{\eta\alpha^2}{2\eta L}\sqn{x^{k+1} - x^*}
	+
	\bg_{r}(x_g^k,x^*)
	\\&\leq
	-\frac{\eta\delta}{4}\sqn{y^{k+1} - y^*} + \frac{\alpha \mu}{4 L}\sqn{x^{k+1} - x^*}
	+
	\bg_{r}(x_g^k,x^*)
	\\&\leq
	-\frac{\eta\delta}{4}\sqn{y^{k+1} - y^*} + \frac{\alpha}{4}\sqn{x^{k+1} - x^*}
	+
	\bg_{r}(x_g^k,x^*).
	\end{align*}
\end{proof}

\begin{lemma}\label{scary:lem:x}
	Let $\alpha$ satisfy
	\begin{equation}
	\alpha \leq \frac{\mu}{2}. \tag{\ref{scary:alpha}}
	\end{equation}
	Let $\eta$ satisfy
	\begin{equation}\label{scary:eta}
	\eta \leq \frac{1}{4\tau L}.
	\end{equation}
	Then the following inequality holds:
	\begin{align}\label{scary:x}
	\frac{1}{\eta}\sqn{x^{k+1} - x^*}
	&\leq
	\frac{1}{\eta}\sqn{x^k - x^*}
	-
	\frac{3\alpha}{4}\sqn{x^{k+1} - x^*}
	+
	\frac{2(1-\tau)}{\tau}\bg_r(x_f^k,x^*)
	-
	\frac{2-\tau}{\tau}\bg_r(x_f^{k+1},x^*)
	\\&-
	\frac{\eta\delta}{4}\sqn{y^{k+1} - y^*}
	+
	2\< y^{k+1} - y^*, x^{k+1} - x^*>.\nonumber
	\end{align}	
	
\end{lemma}
\begin{proof}
	Using line~\ref{alg:scary:line:x:2} of Algorithm~\ref{alg:scary} we get
	\begin{align*}
	\frac{1}{\eta}\sqn{x^{k+1} - x^*}
	&=
	\frac{1}{\eta}\sqn{x^k - x^*}
	+
	\frac{2}{\eta}\<x^{k+1} - x^k, x^{k+1} - x^*>
	-
	\frac{1}{\eta}\sqn{x^{k+1}-x^k}
	\\&=
	\frac{1}{\eta}\sqn{x^k - x^*}
	-
	\frac{1}{\eta}\sqn{x^{k+1}-x^k}
	+
	2\alpha\<  x_g^k - x^{k+1}, x^{k+1} - x^*>
	\\&-
	2\< \nabla r(x_g^k) - y^{k+1}, x^{k+1} - x^*>
	\\&=
	\frac{1}{\eta}\sqn{x^k - x^*}
	-
	\frac{1}{\eta}\sqn{x^{k+1}-x^k}
	+
	2\alpha\<  x_g^k - x^*, x^{k+1} - x^*>
	-
	2\alpha\sqn{x^{k+1} - x^*}
	\\&-
	2\< \nabla r(x_g^k) - y^{k+1}, x^{k+1} - x^*>
	\\&\leq
	\frac{1}{\eta}\sqn{x^k - x^*}
	-
	\frac{1}{\eta}\sqn{x^{k+1}-x^k}
	+
	\alpha \sqn{x_g^k - x^*}
	-
	\alpha\sqn{x^{k+1} - x^*}
	\\&-
	2\< \nabla r(x_g^k) - y^{k+1}, x^{k+1} - x^*>.
	\end{align*}
	From optimality condition \eqref{opt:4} it follows that $\nabla r(x^*) = y^*$ and hence
	\begin{align*}
	\frac{1}{\eta}\sqn{x^{k+1} - x^*}
	&\leq
	\frac{1}{\eta}\sqn{x^k - x^*}
	+
	\alpha \sqn{x_g^k - x^*}
	-
	\alpha\sqn{x^{k+1} - x^*}
	-
	\frac{1}{\eta}\sqn{x^{k+1}-x^k}
	\\&-
	2\< \nabla r(x_g^k) - \nabla r(x^*), x^{k+1} - x^*>
	+
	2\< y^{k+1} - y^*, x^{k+1} - x^*>.
	\end{align*}
	Using lemma~\ref{scary:lem:1} we get
	\begin{align*}
	\frac{1}{\eta}\sqn{x^{k+1} - x^*}
	&\leq
	\frac{1}{\eta}\sqn{x^k - x^*}
	+
	\alpha \sqn{x_g^k - x^*}
	-
	\alpha\sqn{x^{k+1} - x^*}
	-
	\frac{1}{2\eta}\sqn{x^{k+1}-x^k}
	\\&-
	\frac{\eta\delta}{4}\sqn{y^{k+1} - y^*}
	+
	\frac{\alpha}{4}\sqn{x^{k+1} - x^*}
	+
	\bg_{r}(x_g^k,x^*)
	\\&-
	2\< \nabla r(x_g^k) - \nabla r(x^*), x^{k+1} - x^*>
	+
	2\< y^{k+1} - y^*, x^{k+1} - x^*>
	\\&\leq
	\frac{1}{\eta}\sqn{x^k - x^*}
	+
	\alpha \sqn{x_g^k - x^*}
	-
	\frac{3\alpha}{4}\sqn{x^{k+1} - x^*}
	-
	\frac{1}{2\eta}\sqn{x^{k+1}-x^k}
	\\&-
	2\< \nabla r(x_g^k) - \nabla r(x^*), x^{k+1} - x^*>
	+
	2\< y^{k+1} - y^*, x^{k+1} - x^*>
	\\&+
	\bg_{r}(x_g^k,x^*)
	-
	\frac{\eta\delta}{4}\sqn{y^{k+1} - y^*}.
	\end{align*}
	Using lines~\ref{alg:scary:line:x:1} and~\ref{alg:scary:line:x:3} of Algorithm~\ref{alg:scary} we get
	\begin{align*}
	\frac{1}{\eta}\sqn{x^{k+1} - x^*}
	&\leq
	\frac{1}{\eta}\sqn{x^k - x^*}
	+
	\alpha \sqn{x_g^k - x^*}
	-
	\frac{3\alpha}{4}\sqn{x^{k+1} - x^*}
	\\&+
	\frac{2(1-\tau)}{\tau}\< \nabla r(x_g^k) - \nabla r(x^*), x_f^k - x_g^k>
	-
	2\< \nabla r(x_g^k) - \nabla r(x^*), x_g^k - x^*>
	\\&-
	\frac{2-\tau}{\tau}\< \nabla r(x_g^k) - \nabla r(x^*),x_f^{k+1} - x_g^k>
	-
	\frac{(2-\tau)^2}{8\eta\tau^2}\sqn{x_f^{k+1}-x_g^k}
	\\&+
	\bg_{r}(x_g^k,x^*)
	-
	\frac{\eta\delta}{4}\sqn{y^{k+1} - y^*}
	+
	2\< y^{k+1} - y^*, x^{k+1} - x^*>
	\\&\leq
	\frac{1}{\eta}\sqn{x^k - x^*}
	+
	\alpha \sqn{x_g^k - x^*}
	-
	\frac{3\alpha}{4}\sqn{x^{k+1} - x^*}
	\\&+
	\frac{2(1-\tau)}{\tau}\< \nabla r(x_g^k) - \nabla r(x^*), x_f^k - x_g^k>
	-
	2\< \nabla r(x_g^k) - \nabla r(x^*), x_g^k - x^*>
	\\&-
	\frac{2-\tau}{\tau}
	\left(
	\< \nabla r(x_g^k) - \nabla r(x^*),x_f^{k+1} - x_g^k>
	+
	\frac{1}{8\eta\tau}\sqn{x_f^{k+1}-x_g^k}
	\right)
	\\&+
	\bg_{r}(x_g^k,x^*)
	-
	\frac{\eta\delta}{4}\sqn{y^{k+1} - y^*}
	+
	2\< y^{k+1} - y^*, x^{k+1} - x^*>.
	\end{align*}
	Using $\frac{\mu}{2}$-strong convexity and $L$-smoothness of $r(x)$ and $\eta$ defined by \eqref{scary:eta} we get
	\begin{align*}
	\frac{1}{\eta}\sqn{x^{k+1} - x^*}
	&\leq
	\frac{1}{\eta}\sqn{x^k - x^*}
	+
	\alpha \sqn{x_g^k - x^*}
	-
	\frac{3\alpha}{4}\sqn{x^{k+1} - x^*}
	\\&+
	\frac{2(1-\tau)}{\tau}(\bg_r(x_f^k,x^*) - \bg_r(x_g^k,x^*))
	-
	2\bg_r(x_g^k,x^*) - \frac{\mu}{2}\sqn{x_g^k - x^*}
	\\&-
	\frac{2-\tau}{\tau}
	\left(
	\< \nabla r(x_g^k) - \nabla r(x^*),x_f^{k+1} - x_g^k>
	+
	\frac{L}{2}\sqn{x_f^{k+1}-x_g^k}
	\right)
	\\&+
	\bg_{r}(x_g^k,x^*)
	-
	\frac{\eta\delta}{4}\sqn{y^{k+1} - y^*}
	+
	2\< y^{k+1} - y^*, x^{k+1} - x^*>
	\\&\leq
	\frac{1}{\eta}\sqn{x^k - x^*}
	+
	\left(\alpha - \frac{\mu}{2}\right) \sqn{x_g^k - x^*}
	-
	\frac{3\alpha}{4}\sqn{x^{k+1} - x^*}
	\\&+
	\frac{2(1-\tau)}{\tau}(\bg_r(x_f^k,x^*) - \bg_r(x_g^k,x^*))
	-
	\frac{2-\tau}{\tau}(\bg_r(x_f^{k+1},x^*) - \bg_r(x_g^k,x^*))
	\\&-
	\bg_r(x_g^k,x^*)
	-
	\frac{\eta\delta}{4}\sqn{y^{k+1} - y^*}
	+
	2\< y^{k+1} - y^*, x^{k+1} - x^*>
	\\&=
	\frac{1}{\eta}\sqn{x^k - x^*}
	+
	\left(\alpha - \frac{\mu}{2}\right) \sqn{x_g^k - x^*}
	-
	\frac{3\alpha}{4}\sqn{x^{k+1} - x^*}
	\\&+
	\frac{2(1-\tau)}{\tau}\bg_r(x_f^k,x^*)
	-
	\frac{2-\tau}{\tau}\bg_r(x_f^{k+1},x^*)
	\\&-
	\frac{\eta\delta}{4}\sqn{y^{k+1} - y^*}
	+
	2\< y^{k+1} - y^*, x^{k+1} - x^*>.
	\end{align*}
	Using $\alpha$ defined by \eqref{scary:alpha} we get
	\begin{align*}
	\frac{1}{\eta}\sqn{x^{k+1} - x^*}
	&\leq
	\frac{1}{\eta}\sqn{x^k - x^*}
	-
	\frac{3\alpha}{4}\sqn{x^{k+1} - x^*}
	+
	\frac{2(1-\tau)}{\tau}\bg_r(x_f^k,x^*)
	-
	\frac{2-\tau}{\tau}\bg_r(x_f^{k+1},x^*)
	\\&-
	\frac{\eta\delta}{4}\sqn{y^{k+1} - y^*}
	+
	2\< y^{k+1} - y^*, x^{k+1} - x^*>.
	\end{align*}	
\end{proof}

\begin{lemma}\label{scary:lem:2}
	For all $y_1,y_2\in\R^{nd}$ and $z_1,z_2\in \range \mW$ the following inequality holds:
	\begin{equation}
	\bg_h((y_1,z_1),(y_2,z_2)) \leq \left(\frac{2}{\mu} + \frac{\nu}{2}\right)\sqn{y_1-y_2}+\frac{2}{\mu}\sqn{z_1-z_2}.
	\end{equation}
\end{lemma}
\begin{proof}
	It follows from from the definition of $\bg_h$:
	\begin{align*}
	\bg_h((y_1,z_1),(y_2,z_2)) &= \frac{1}{\mu}\sqn{y_1+z_1 - y_2-z_2} + \frac{\nu}{2}\sqn{y_1 - y_2}
	\\&\leq
	\left(\frac{2}{\mu} + \frac{\nu}{2}\right)\sqn{y_1-y_2}+\frac{2}{\mu}\sqn{z_1-z_2}.
	\end{align*}
\end{proof}

\begin{lemma}\label{scary:lem:yz}
	Let $\theta$ satisfy
	\begin{equation}\label{scary:theta}
	\theta \leq \left[\sigma \left(\frac{4}{\mu} + \nu\right)\right]^{-1}.
	\end{equation}
	Let $\lambda$ satisfy
	\begin{equation}\label{scary:lambda}
	\lambda \leq \left[\frac{4\sigma \lambda_{\max}(\mW) }{\mu}\right]^{-1}.
	\end{equation}
	Let $\beta$ satisfy
	\begin{equation}\label{scary:beta}
	\beta \leq \min\left\{\frac{1}{\mu}, \frac{\nu}{3}\right\}.
	\end{equation}
	Let $\gamma$ satisfy
	\begin{equation}\label{scary:gamma}
	\gamma \leq \lambda_{\min}^+(\mW)\beta.
	\end{equation}
	Then the following inequality holds:
	\begin{align}\label{scary:yz}
	\sqN{\vect{y^{k+1} - y^*\\z^{k+1}-z^*}}_{\mM}
	&\leq
	\sqN{\vect{y^{k} - y^*\\z^{k}-z^*}}_{\mM}
	-
	(\beta - 2\nu)\sqn{y^{k+1} - y^*}
	-
	\gamma\sqnw{z^{k+1} - z^*}
	\\&+
	\frac{2(1-\sigma)}{\sigma}
	\bg_h((y_f^k,z_f^k),(y^*,z^*))
	-
	\frac{2}{\sigma}
	\bg_h((y_f^{k+1},z_f^{k+1}),(y^*,z^*))
	\\&-
	2\<x^{k+1} - x^*,y^{k+1} - y^*>,
	\end{align}
	where $\mM\in \R^{2nd\times 2nd}$ is a matrix defined by
	\begin{equation}\label{scary:M}
	\mM = \begin{bmatrix}
	\frac{1}{\theta}\mI&0\\0&\frac{1}{\lambda}\mWp
	\end{bmatrix}.
	\end{equation}
\end{lemma}
\begin{proof}
	Using line~\ref{alg:scary:line:y:2} of Algorithm~\ref{alg:scary} we get
	\begin{align*}
	\frac{1}{\theta}\sqn{y^{k+1} - y^*}
	&=
	\frac{1}{\theta}\sqn{y^k - y^*}
	+
	\frac{2}{\theta}\<y^{k+1} - y^k,y^{k+1} - y^*>
	-
	\frac{1}{\theta}\sqn{y^{k+1} - y^k}
	\\&=
	\frac{1}{\theta}\sqn{y^k - y^*}
	-
	\frac{1}{\theta}\sqn{y^{k+1} - y^k}
	+
	2\beta\<  y_g^k - y^{k+1} ,y^{k+1} - y^*>
	\\&-
	2\<\nabla_y h(y_g^k,z_g^k) - \nu y^{k+1} +  x^{k+1},y^{k+1} - y^*>
	\\&=
	\frac{1}{\theta}\sqn{y^k - y^*}
	-
	\frac{1}{\theta}\sqn{y^{k+1} - y^k}
	+
	2\beta\<  y_g^k - y^* ,y^{k+1} - y^*>
	-
	2\beta\sqn{y^{k+1} - y^*}
	\\&-
	2\<\nabla_y h(y_g^k,z_g^k) - \nu y^{k+1} +  x^{k+1},y^{k+1} - y^*>
	\\&\leq
	\frac{1}{\theta}\sqn{y^k - y^*}
	+
	\beta\sqn{y_g^k - y^*}
	-
	\beta\sqn{y^{k+1} - y^*}
	-
	\frac{1}{\theta}\sqn{y^{k+1} - y^k}
	\\&-
	2\<\nabla_y h(y_g^k,z_g^k) - \nu y^{k+1} +  x^{k+1},y^{k+1} - y^*>.
	\end{align*}
	From optimality condition \eqref{opt:5} it follows that $x^* = -\frac{2}{\mu}(y^*+z^*) = -\nabla_y h(y^*,z^*) + \nu y^*$ and hence
	\begin{align*}
	\frac{1}{\theta}\sqn{y^{k+1} - y^*}
	&\leq
	\frac{1}{\theta}\sqn{y^k - y^*}
	+
	\beta\sqn{y_g^k - y^*}
	-
	\beta\sqn{y^{k+1} - y^*}
	-
	\frac{1}{\theta}\sqn{y^{k+1} - y^k}
	\\&-
	2\<\nabla_y h(y_g^k,z_g^k) - \nabla_y h(y^*,z^*) ,y^{k+1} - y^*>
	\\&+
	2\nu\sqn{y^{k+1} - y^*}
	-
	2\<x^{k+1} - x^*,y^{k+1} - y^*>
	\\&=
	\frac{1}{\theta}\sqn{y^k - y^*}
	-
	(\beta - 2\nu)\sqn{y^{k+1} - y^*}
	+
	\beta\sqn{y_g^k - y^*}
	-
	\frac{1}{\theta}\sqn{y^{k+1} - y^k}
	\\&-
	2\<\nabla_y h(y_g^k,z_g^k) - \nabla_y h(y^*,z^*) ,y^{k+1} - y^*>
	-
	2\<x^{k+1} - x^*,y^{k+1} - y^*>.
	\end{align*}
	Using lines~\ref{alg:scary:line:y:1} and~\ref{alg:scary:line:y:3} of Algorithm~\ref{alg:scary} we get
	\begin{align*}
	\frac{1}{\theta}\sqn{y^{k+1} - y^*}
	&\leq
	\frac{1}{\theta}\sqn{y^k - y^*}
	-
	(\beta - 2\nu)\sqn{y^{k+1} - y^*}
	+
	\beta\sqn{y_g^k - y^*}
	\\&+
	\frac{2(1-\sigma)}{\sigma}\<\nabla_y h(y_g^k,z_g^k) - \nabla_y h(y^*,z^*) ,y_f^k - y_g^k>
	\\&-
	\frac{2}{\sigma}\<\nabla_y h(y_g^k,z_g^k) - \nabla_y h(y^*,z^*) ,y_f^{k+1} - y_g^k>
	-
	\frac{1}{\theta\sigma^2}\sqn{y_f^{k+1} - y_g^k}
	\\&-
	2\<\nabla_y h(y_g^k,z_g^k) - \nabla_y h(y^*,z^*) ,y_g^k - y^*>
	-
	2\<x^{k+1} - x^*,y^{k+1} - y^*>.
	\end{align*}
	Using $\theta$ defined by \eqref{scary:theta} we get
	\begin{align}\label{scary:y}
	\frac{1}{\theta}\sqn{y^{k+1} - y^*}
	&\leq
	\frac{1}{\theta}\sqn{y^k - y^*}
	-
	(\beta - 2\nu)\sqn{y^{k+1} - y^*}
	+
	\beta\sqn{y_g^k - y^*}
	\\&+
	\frac{2(1-\sigma)}{\sigma}\<\nabla_y h(y_g^k,z_g^k) - \nabla_y h(y^*,z^*) ,y_f^k - y_g^k>
	\\&-
	\frac{2}{\sigma}\left(\<\nabla_y h(y_g^k,z_g^k) - \nabla_y h(y^*,z^*) ,y_f^{k+1} - y_g^k>
	+
	\left(\frac{2}{\mu}+\frac{\nu}{2}\right)\sqn{y_f^{k+1} - y_g^k}\right)
	\\&-
	2\<\nabla_y h(y_g^k,z_g^k) - \nabla_y h(y^*,z^*) ,y_g^k - y^*>
	-
	2\<x^{k+1} - x^*,y^{k+1} - y^*>.
	\end{align}
	Using line~\ref{alg:scary:line:z:2} of Algorithm~\ref{alg:scary} we get
	\begin{align*}
	\frac{1}{\lambda}\sqnw{z^{k+1} - z^*}
	&=
	\frac{1}{\lambda}\sqnw{z^k-z^*}
	+
	\frac{2}{\lambda}\< z^{k+1} - z^k,\mW^\dagger(z^{k+1} - z^*)>
	-
	\frac{1}{\lambda}\sqnw{z^{k+1} - z^k}
	\\&=
	\frac{1}{\lambda}\sqnw{z^k-z^*}
	-
	\frac{1}{\lambda}\sqnw{z^{k+1} - z^k}
	+
	2\gamma\< z_g^k - z^{k+1},\mW^\dagger(z^{k+1} - z^*)>
	\\&-
	2\<\mW \nabla_z h(y_g^k,z_g^k),\mW^\dagger(z^{k+1} - z^*)>
	\\&=
	\frac{1}{\lambda}\sqnw{z^k-z^*}
	+
	2\gamma\< z_g^k - z^*,\mW^\dagger(z^{k+1} - z^*)>
	-
	2\gamma\sqnw{z^{k+1} - z^*}
	\\&-
	\frac{1}{\lambda}\sqnw{z^{k+1} - z^k}
	-
	2\<\mW \nabla_z h(y_g^k,z_g^k), \mW^\dagger(z^{k+1} - z^*)>
	\\&\leq
	\frac{1}{\lambda}\sqnw{z^k-z^*}
	+
	\gamma \sqnw{z_g^k - z^*}
	-
	\gamma\sqnw{z^{k+1} - z^*}
	-
	\frac{1}{\lambda}\sqnw{z^{k+1} - z^k}
	\\&-
	2\<\mW\nabla_z h(y_g^k,z_g^k), \mW^\dagger(z^{k+1} - z^*)>.
	\end{align*}
	From optimality condition \eqref{opt:6} it follows that $\mW \nabla_z h(y^*,z^*) = 0$ and hence
	\begin{align*}
	\frac{1}{\lambda}\sqnw{z^{k+1} - z^*}
	&\leq
	\frac{1}{\lambda}\sqnw{z^k-z^*}
	+
	\gamma \sqnw{z_g^k - z^*}
	-
	\gamma\sqnw{z^{k+1} - z^*}
	-
	\frac{1}{\lambda}\sqnw{z^{k+1} - z^k}
	\\&-
	2\<\mW(\nabla_z h(y_g^k,z_g^k)-\nabla_z h(y^*,z^*) ), \mW^\dagger(z^{k+1} - z^*)>
	\\&=
	\frac{1}{\lambda}\sqnw{z^k-z^*}
	+
	\gamma \sqnw{z_g^k - z^*}
	-
	\gamma\sqnw{z^{k+1} - z^*}
	-
	\frac{1}{\lambda}\sqnw{z^{k+1} - z^k}
	\\&-
	2\<\nabla_z h(y_g^k,z_g^k)-\nabla_z h(y^*,z^*) , \mW\mW^\dagger(z^{k+1} - z^*)>.
	\end{align*}
	It's easy to observe that $z^k, z^* \in \range \mW$  for all $k=0,1,2,\ldots$, which implies
	\begin{align*}
	\mW\mW^\dagger (z^{k+1} - z^*) = z^{k+1} - z^* \;\text{and}\;
	\sqnw{z^{k+1} - z^k} \geq \frac{1}{\lambda_{\max}(\mW)}\sqn{z^{k+1} - z^k}.
	\end{align*}
	Hence,
	\begin{align*}
	\frac{1}{\lambda}\sqnw{z^{k+1} - z^*}
	&\leq
	\frac{1}{\lambda}\sqnw{z^k-z^*}
	-
	\gamma\sqnw{z^{k+1} - z^*}
	+
	\frac{\gamma}{\lambda_{\min}^+(\mW)} \sqn{z_g^k - z^*}
	\\&-
	\frac{1}{\lambda\cdot\lambda_{\max}(\mW)}\sqn{z^{k+1} - z^k}
	-
	2\<\nabla_z h(y_g^k,z_g^k)-\nabla_z h(y^*,z^*) ,z^{k+1} - z^*>.
	\end{align*}
	Using lines~\ref{alg:scary:line:z:1} and~\ref{alg:scary:line:z:3} of Algorithm~\ref{alg:scary} we get
	\begin{align*}
	\frac{1}{\lambda}\sqnw{z^{k+1} - z^*}
	&\leq
	\frac{1}{\lambda}\sqnw{z^k-z^*}
	-
	\gamma\sqnw{z^{k+1} - z^*}
	+
	\frac{\gamma}{\lambda_{\min}^+(\mW)} \sqnw{z_g^k - z^*}
	\\&+
	\frac{2(1-\sigma)}{\sigma}
	\<\nabla_z h(y_g^k,z_g^k)-\nabla_z h(y^*,z^*) ,z_f^k - z_g^k>
	\\&-
	\frac{2}{\sigma}\<\nabla_z h(y_g^k,z_g^k)-\nabla_z h(y^*,z^*) ,z_f^{k+1} - z_g^k>
	-
	\frac{1}{\lambda\sigma^2\lambda_{\max}(\mW)}\sqn{z_f^{k+1} - z_g^k}\\
	&-
	2\<\nabla_z h(y_g^k,z_g^k)-\nabla_z h(y^*,z^*) , z_g^k - z^*>.
	\end{align*}
	Using $\lambda$ defined by \eqref{scary:lambda} we get
	\begin{align}\label{scary:z}
	\frac{1}{\lambda}\sqnw{z^{k+1} - z^*}
	&\leq
	\frac{1}{\lambda}\sqnw{z^k-z^*}
	-
	\gamma\sqnw{z^{k+1} - z^*}
	+
	\frac{\gamma}{\lambda_{\min}^+(\mW)} \sqnw{z_g^k - z^*}
	\\&+
	\frac{2(1-\sigma)}{\sigma}
	\<\nabla_z h(y_g^k,z_g^k)-\nabla_z h(y^*,z^*) ,z_f^k - z_g^k>
	\\&-
	\frac{2}{\sigma}
	\left(\<\nabla_z h(y_g^k,z_g^k)-\nabla_z h(y^*,z^*) ,z_f^{k+1} - z_g^k>
	-
	\frac{2}{\mu}\sqn{z_f^{k+1} - z_g^k}
	\right)\\
	&-
	2\<\nabla_z h(y_g^k,z_g^k)-\nabla_z h(y^*,z^*) , z_g^k - z^*>.
	\end{align}
	After combining \eqref{scary:y} and \eqref{scary:z} we get
	\begin{align*}
	\sqN{\vect{y^{k+1} - y^*\\z^{k+1}-z^*}}_{\mM}
	&\leq
	\sqN{\vect{y^{k} - y^*\\z^{k}-z^*}}_{\mM}
	-
	(\beta - 2\nu)\sqn{y^{k+1} - y^*}
	-
	\gamma\sqnw{z^{k+1} - z^*}
	\\&+
	\beta\sqn{y_g^k - y^*}
	+
	\frac{\gamma}{\lambda_{\min}^+(\mW)}\sqn{z_g^k - z^*}\\
	&+
	\frac{2(1-\sigma)}{\sigma}\Dotprod{\nabla h(y_g^k,z_g^k) - \nabla h(y^*,z^*),\vect{y_f^k\\z_f^k}-\vect{y_g^k\\z_g^k}}
	\\&
	-\frac{2}{\sigma}
	\Dotprod{\nabla h(y_g^k,z_g^k)-\nabla h(y^*,z^*),\vect{y_f^{k+1}\\z_f^{k+1}}-\vect{y_g^k\\z_g^k}}
	\\&-
	\frac{2}{\sigma}
	\left(\left(\frac{2}{\mu}+\frac{\nu}{2}\right)\sqn{y_f^{k+1} - z_g^k}
	+
	\frac{2}{\mu}\sqn{z_f^{k+1} - z_g^k}\right)
	\\&-
	2\Dotprod{\nabla h(y_g^k,z_g^k) - \nabla h(y^*,z^*), \vect{y_g^k\\z_g^k} - \vect{y^*\\z^*}}
	-
	2\<x^{k+1} - x^*,y^{k+1} - y^*>,
	\end{align*}
	where $\mM\in \R^{2nd\times 2nd}$ is a matrix defined by \eqref{scary:M}.
	Using convexity of $h(y,z)$ and the fact that $\nabla h(y,z) = \vect{\frac{2}{\mu}(y+z) + \nu y\\\frac{2}{\mu}(y+z)}$ we get
	\begin{align*}
	\sqN{\vect{y^{k+1} - y^*\\z^{k+1}-z^*}}_{\mM}
	&\leq
	\sqN{\vect{y^{k} - y^*\\z^{k}-z^*}}_{\mM}
	-
	(\beta - 2\nu)\sqn{y^{k+1} - y^*}
	-
	\gamma\sqnw{z^{k+1} - z^*}
	\\&+
	\beta\sqn{y_g^k - y^*}
	+
	\frac{\gamma}{\lambda_{\min}^+(\mW)}\sqn{z_g^k - z^*}\\
	&+
	\frac{2(1-\sigma)}{\sigma}
	\left[
	\bg_h((y_f^k,z_f^k),(y^*,z^*))
	-
	\bg_h((y_g^k,z_g^k),(y^*,z^*))
	\right]
	\\&
	-\frac{2}{\sigma}
	\Dotprod{\nabla h(y_g^k,z_g^k)-\nabla h(y^*,z^*),\vect{y_f^{k+1}\\z_f^{k+1}}-\vect{y_g^k\\z_g^k}}
	\\&-
	\frac{2}{\sigma}
	\left(\left(\frac{2}{\mu}+\frac{\nu}{2}\right)\sqn{y_f^{k+1} - z_g^k}
	+
	\frac{2}{\mu}\sqn{z_f^{k+1} - z_g^k}\right)
	\\&-
	\frac{4}{\mu}\sqn{y_g^k + z_g^k - y^* - z^*}
	-
	2\nu\sqn{y_g^k - y^*}
	-
	2\<x^{k+1} - x^*,y^{k+1} - y^*>.
	\end{align*}
	Using lemma~\ref{scary:lem:2} we can obtain
	\begin{align*}
	\sqN{\vect{y^{k+1} - y^*\\z^{k+1}-z^*}}_{\mM}
	&\leq
	\sqN{\vect{y^{k} - y^*\\z^{k}-z^*}}_{\mM}
	-
	(\beta - 2\nu)\sqn{y^{k+1} - y^*}
	-
	\gamma\sqnw{z^{k+1} - z^*}
	\\&+
	(\beta - 2\nu)\sqn{y_g^k - y^*}
	+
	\frac{\gamma}{\lambda_{\min}^+(\mW)}\sqn{z_g^k - z^*}
	-
	\frac{4}{\mu}\sqn{y_g^k + z_g^k - y^* - z^*}
	\\&+
	\frac{2(1-\sigma)}{\sigma}
	\left[
	\bg_h((y_f^k,z_f^k),(y^*,z^*))
	-
	\bg_h((y_g^k,z_g^k),(y^*,z^*))
	\right]
	\\&
	-\frac{2}{\sigma}\left[
	\bg_h((y_f^{k+1},z_f^{k+1}),(y^*,z^*))
	-
	\bg_h((y_g^k,z_g^k),(y^*,z^*))
	\right]
	\\&-
	2\<x^{k+1} - x^*,y^{k+1} - y^*>
	\\&=
	\sqN{\vect{y^{k} - y^*\\z^{k}-z^*}}_{\mM}
	-
	(\beta - 2\nu)\sqn{y^{k+1} - y^*}
	-
	\gamma\sqnw{z^{k+1} - z^*}
	\\&+
	(\beta - 2\nu)\sqn{y_g^k - y^*}
	+
	\frac{\gamma}{\lambda_{\min}^+(\mW)}\sqn{z_g^k - z^*}
	-
	\frac{4}{\mu}\sqn{y_g^k + z_g^k - y^* - z^*}
	\\&+
	\frac{2(1-\sigma)}{\sigma}
	\bg_h((y_f^k,z_f^k),(y^*,z^*))
	-
	\frac{2}{\sigma}
	\bg_h((y_f^{k+1},z_f^{k+1}),(y^*,z^*))
	\\&+
	2\bg_h((y_g^k,z_g^k), (y^*,z^*))
	-
	2\<x^{k+1} - x^*,y^{k+1} - y^*>
	\\&=
	\sqN{\vect{y^{k} - y^*\\z^{k}-z^*}}_{\mM}
	-
	(\beta - 2\nu)\sqn{y^{k+1} - y^*}
	-
	\gamma\sqnw{z^{k+1} - z^*}
	\\&+
	(\beta - 2\nu)\sqn{y_g^k - y^*}
	+
	\frac{\gamma}{\lambda_{\min}^+(\mW)}\sqn{z_g^k - z^*}
	-
	\frac{4}{\mu}\sqn{y_g^k + z_g^k - y^* - z^*}
	\\&+
	\frac{2(1-\sigma)}{\sigma}
	\bg_h((y_f^k,z_f^k),(y^*,z^*))
	-
	\frac{2}{\sigma}
	\bg_h((y_f^{k+1},z_f^{k+1}),(y^*,z^*))
	\\&+
	\frac{2}{\mu}\sqn{y_g^k + z_g^k - y^* - z^*}
	+
	\nu\sqn{y_g^k - y^*}
	-
	2\<x^{k+1} - x^*,y^{k+1} - y^*>
	\\&=
	\sqN{\vect{y^{k} - y^*\\z^{k}-z^*}}_{\mM}
	-
	(\beta - 2\nu)\sqn{y^{k+1} - y^*}
	-
	\gamma\sqnw{z^{k+1} - z^*}
	\\&+
	\frac{2(1-\sigma)}{\sigma}
	\bg_h((y_f^k,z_f^k),(y^*,z^*))
	-
	\frac{2}{\sigma}
	\bg_h((y_f^{k+1},z_f^{k+1}),(y^*,z^*))
	\\&+
	(\beta - \nu)\sqn{y_g^k - y^*}
	+
	\frac{\gamma}{\lambda_{\min}^+(\mW)}\sqn{z_g^k - z^*}
	-
	\frac{2}{\mu}\sqn{y_g^k + z_g^k - y^* - z^*}
	\\&-
	2\<x^{k+1} - x^*,y^{k+1} - y^*>.
	\end{align*}
	Using $\gamma$ defined by \eqref{scary:gamma} and the fact that $\beta \leq \frac{1}{\mu}$ which follows from \eqref{scary:beta} we get
	\begin{align*}
	\sqN{\vect{y^{k+1} - y^*\\z^{k+1}-z^*}}_{\mM}
	&\leq
	\sqN{\vect{y^{k} - y^*\\z^{k}-z^*}}_{\mM}
	-
	(\beta - 2\nu)\sqn{y^{k+1} - y^*}
	-
	\gamma\sqnw{z^{k+1} - z^*}
	\\&+
	\frac{2(1-\sigma)}{\sigma}
	\bg_h((y_f^k,z_f^k),(y^*,z^*))
	-
	\frac{2}{\sigma}
	\bg_h((y_f^{k+1},z_f^{k+1}),(y^*,z^*))
	\\&+
	(\beta - \nu)\sqn{y_g^k - y^*}
	+
	\beta\sqn{z_g^k - z^*}
	-
	2\beta\sqn{y_g^k + z_g^k - y^* - z^*}
	\\&-
	2\<x^{k+1} - x^*,y^{k+1} - y^*>
	\\&\leq
	\sqN{\vect{y^{k} - y^*\\z^{k}-z^*}}_{\mM}
	-
	(\beta - 2\nu)\sqn{y^{k+1} - y^*}
	-
	\gamma\sqnw{z^{k+1} - z^*}
	\\&+
	\frac{2(1-\sigma)}{\sigma}
	\bg_h((y_f^k,z_f^k),(y^*,z^*))
	-
	\frac{2}{\sigma}
	\bg_h((y_f^{k+1},z_f^{k+1}),(y^*,z^*))
	\\&+
	(\beta - \nu)\sqn{y_g^k - y^*}
	+
	\beta\sqn{z_g^k - z^*}
	-
	\beta\sqn{z_g^k - z^*}
	+2\beta\sqn{y_g^k - y^*}
	\\&-
	2\<x^{k+1} - x^*,y^{k+1} - y^*>
	\\&=
	\sqN{\vect{y^{k} - y^*\\z^{k}-z^*}}_{\mM}
	-
	(\beta - 2\nu)\sqn{y^{k+1} - y^*}
	-
	\gamma\sqnw{z^{k+1} - z^*}
	\\&+
	\frac{2(1-\sigma)}{\sigma}
	\bg_h((y_f^k,z_f^k),(y^*,z^*))
	-
	\frac{2}{\sigma}
	\bg_h((y_f^{k+1},z_f^{k+1}),(y^*,z^*))
	\\&+
	(3\beta - \nu)\sqn{y_g^k - y^*}
	-
	2\<x^{k+1} - x^*,y^{k+1} - y^*>.
	\end{align*}
	Using the fact that $\beta \leq \frac{\nu}{3}$ which follows from \eqref{scary:beta} we get
	\begin{align*}
	\sqN{\vect{y^{k+1} - y^*\\z^{k+1}-z^*}}_{\mM}
	&\leq
	\sqN{\vect{y^{k} - y^*\\z^{k}-z^*}}_{\mM}
	-
	(\beta - 2\nu)\sqn{y^{k+1} - y^*}
	-
	\gamma\sqnw{z^{k+1} - z^*}
	\\&+
	\frac{2(1-\sigma)}{\sigma}
	\bg_h((y_f^k,z_f^k),(y^*,z^*))
	-
	\frac{2}{\sigma}
	\bg_h((y_f^{k+1},z_f^{k+1}),(y^*,z^*))
	\\&-
	2\<x^{k+1} - x^*,y^{k+1} - y^*>.
	\end{align*}
\end{proof}

\begin{theorem}
	Let $\tau$ be defined by
	\begin{equation}
	\tau = \frac{1}{2}\sqrt{\frac{\mu}{L}}.
	\end{equation}
	Let $\alpha$ be defined by
	\begin{equation}
	\alpha = \frac{\mu}{2}.
	\end{equation}
	Let $\eta$ be defined by
	\begin{equation}
	\eta = \frac{1}{2\sqrt{\mu L}}.
	\end{equation}
	Let $\sigma$ be defined by
	\begin{equation}
	\sigma = \frac{1}{18}\sqrt{\frac{\mu \lambda_{\min}^+(\mW)}{L \lambda_{\max}(\mW)}}.
	\end{equation}
	Let $\nu$ be defined by
	\begin{equation}
	\nu = \frac{3}{80L}.
	\end{equation}
	Let $\beta$ be defined by
	\begin{equation}
	\beta = \frac{1}{80L}.
	\end{equation}
	Let $\theta$ be defined by
	\begin{equation}
	\theta = \frac{18\sqrt{\mu L \lambda_{\max}(\mW)}}{5\sqrt{ \lambda_{\min}^+(\mW)}}
	\end{equation}
	Let $\gamma$ be defined by
	\begin{equation}
	\gamma = \frac{\lambda_{\min}^+(\mW)}{80 L}.
	\end{equation}
	Let $\lambda$ be defined by
	\begin{equation}
	\lambda = \frac{9 \sqrt{\mu L}}{2 \sqrt{ \lambda_{\min}^+(\mW)\lambda_{\max}(\mW)}}.
	\end{equation}

	where $\mP\in \R^{3nd\times 3nd}$ is a matrix defined by 
	\begin{equation}\label{scary:P}
	\mP = \begin{bmatrix}
	\frac{1}{\eta}\mI & 0 &0\\
	0&\frac{1}{\theta}\mI&0\\
	0&0&\frac{1}{\lambda}\mWp
	\end{bmatrix}.
	\end{equation}
	Let $\rho$ be defined by
	\begin{equation}\label{scary:rho}
		 \rho = \frac{1}{18}\sqrt{\frac{\mu \lambda_{\min}^+(\mW)}{L \lambda_{\max}(\mW)}}.
	\end{equation}
	Let $\Psi^k$ be the following Lyapunov function:
	\begin{equation}\label{scary:psi}
		\Psi^k = (1+\rho)\sqN{\vect{x^k - x^*\\y^k - y^*\\z^k - z^*}}_\mP
		+
		\frac{(2-\tau)}{\tau}\bg_r(x_f^k,x^*)
		+
		\frac{2}{\sigma}\bg_h((y_f^k,z_f^k),(y^*,z^*)).
	\end{equation}
	Then the following inequality holds:
	\begin{equation}
		\Psi^{k+1} \leq \left(1-\frac{1}{1+\rho^{-1}}\right)\Psi^k.
	\end{equation}
\end{theorem}

\begin{proof}
	One can observe that conditions of lemma~\ref{scary:lem:x} and lemma~\ref{scary:lem:yz} are satisfied. Hence we can combine \eqref{scary:x} and \eqref{scary:yz} and get
	\begin{align*}
	\sqN{\vect{x^{k+1} - x^*\\y^{k+1} - y^*\\z^{k+1} - z^*}}_\mP
	&=
	\frac{1}{\eta}\sqn{x^{k+1} - y^*}
	+
	\sqN{\vect{y^{k+1} - y^*\\z^{k+1}-z^*}}_\mM
	\\&\leq
	\frac{1}{\eta}\sqn{x^k - x^*}
	-
	\frac{3\alpha}{4}\sqn{x^{k+1} - x^*}
	+
	\frac{2(1-\tau)}{\tau}\bg_r(x_f^k,x^*)
	\\&-
	\frac{2-\tau}{\tau}\bg_r(x_f^{k+1},x^*)
	-
	\frac{\eta\delta}{4}\sqn{y^{k+1} - y^*}
	+
	2\< y^{k+1} - y^*, x^{k+1} - x^*>
	\\&+
	\sqN{\vect{y^{k} - y^*\\z^{k}-z^*}}_{\mM}
	-
	(\beta - 2\nu)\sqn{y^{k+1} - y^*}
	-
	\gamma\sqnw{z^{k+1} - z^*}
	\\&+
	\frac{2(1-\sigma)}{\sigma}
	\bg_h((y_f^k,z_f^k),(y^*,z^*))
	-
	\frac{2}{\sigma}
	\bg_h((y_f^{k+1},z_f^{k+1}),(y^*,z^*))
	\\&-
	2\<x^{k+1} - x^*,y^{k+1} - y^*>
	\\&=
	\sqN{\vect{x^k - x^*\\y^k - y^*\\z^k - z^*}}_\mP
	-
	\frac{3\alpha}{4}\sqn{x^{k+1} - x^*}
	-
	\left(\frac{\eta\delta}{4} + \beta - 2\nu\right)\sqn{y^{k+1} - y^*}
	\\&-
	\gamma\sqnw{z^{k+1} - z^*}
	+
	\frac{2(1-\tau)}{\tau}\bg_r(x_f^k,x^*)
	-
	\frac{2-\tau}{\tau}\bg_r(x_f^{k+1},x^*)
	\\&+
	\frac{2(1-\sigma)}{\sigma}
	\bg_h((y_f^k,z_f^k),(y^*,z^*))
	-
	\frac{2}{\sigma}
	\bg_h((y_f^{k+1},z_f^{k+1}),(y^*,z^*)),
	\end{align*}	
	where $\mP\in \R^{3nd\times 3nd}$ is a matrix defined by \eqref{scary:P}.
	From \eqref{scary:delta} it follows that
	\begin{align*}
	\frac{\eta\delta}{4} = \min \left\{\frac{1}{8\sqrt{\mu L}}, \frac{1}{8 L}\right\} = \frac{1}{8L},
	\end{align*}
	and hence, using choice of  $\alpha$, $\beta$ and $\nu$, we get
	\begin{align*}
	\sqN{\vect{x^{k+1} - x^*\\y^{k+1} - y^*\\z^{k+1} - z^*}}_\mP
	&\leq
	\sqN{\vect{x^k - x^*\\y^k - y^*\\z^k - z^*}}_\mP
	-
	\frac{3\mu}{8}\sqn{x^{k+1} - x^*}
	-
	\left(\frac{1}{8L}+ \frac{1}{80L} - \frac{6}{80L}\right)\sqn{y^{k+1} - y^*}
	\\&-
	\frac{\lambda_{\min}^+(\mW)}{80 L}\sqnw{z^{k+1} - z^*}
	+
	\frac{2(1-\tau)}{\tau}\bg_r(x_f^k,x^*)
	-
	\frac{2-\tau}{\tau}\bg_r(x_f^{k+1},x^*)
	\\&+
	\frac{2(1-\sigma)}{\sigma}
	\bg_h((y_f^k,z_f^k),(y^*,z^*))
	-
	\frac{2}{\sigma}
	\bg_h((y_f^{k+1},z_f^{k+1}),(y^*,z^*))
	\\&\leq
	\sqN{\vect{x^k - x^*\\y^k - y^*\\z^k - z^*}}_\mP
	-
	\min\left\{\frac{3\eta\mu}{8}, \frac{\theta}{16 L}, \frac{\lambda\cdot\lambda_{\min}^+(\mW)}{80 L}\right\}
	\sqN{\vect{x^{k+1} - x^*\\y^{k+1} - y^*\\z^{k+1} - z^*}}_\mP
	\\&+
	\left(1 - \frac{\tau}{2}\right)\frac{(2-\tau)}{\tau}\bg_r(x_f^k,x^*)
	-
	\frac{2-\tau}{\tau}\bg_r(x_f^{k+1},x^*)
	\\&+
	(1-\sigma)\frac{2}{\sigma}
	\bg_h((y_f^k,z_f^k),(y^*,z^*))
	-
	\frac{2}{\sigma}
	\bg_h((y_f^{k+1},z_f^{k+1}),(y^*,z^*))
	\\&=
	\sqN{\vect{x^k - x^*\\y^k - y^*\\z^k - z^*}}_\mP
	-
	\min\left\{
	\frac{3}{16}\sqrt{\frac{\mu}{L}},
	\frac{9\sqrt{\mu  \lambda_{\max}(\mW)}}{40\sqrt{L \lambda_{\min}^+(\mW)}},
	\frac{9 \sqrt{\mu \lambda_{\min}^+(\mW)}}{160 \sqrt{L \lambda_{\max}(\mW)}}\right\}
	\sqN{\vect{x^{k+1} - x^*\\y^{k+1} - y^*\\z^{k+1} - z^*}}_\mP
	\\&+
	\left(1 - \frac{1}{4}\sqrt{\frac{\mu}{L}}\right)\frac{(2-\tau)}{\tau}\bg_r(x_f^k,x^*)
	-
	\frac{2-\tau}{\tau}\bg_r(x_f^{k+1},x^*)
	\\&+
	\left(1-\frac{1}{18}\sqrt{\frac{\mu \lambda_{\min}^+(\mW)}{L \lambda_{\max}(\mW)}}\right)\frac{2}{\sigma}
	\bg_h((y_f^k,z_f^k),(y^*,z^*))
	-
	\frac{2}{\sigma}
	\bg_h((y_f^{k+1},z_f^{k+1}),(y^*,z^*))
	\\&\leq
	\sqN{\vect{x^k - x^*\\y^k - y^*\\z^k - z^*}}_\mP
	-
	\rho
	\sqN{\vect{x^{k+1} - x^*\\y^{k+1} - y^*\\z^{k+1} - z^*}}_\mP
	\\&+
	(1-\rho)\frac{(2-\tau)}{\tau}\bg_r(x_f^k,x^*)
	-
	\frac{2-\tau}{\tau}\bg_r(x_f^{k+1},x^*)
	\\&+
	(1-\rho)\frac{2}{\sigma}
	\bg_h((y_f^k,z_f^k),(y^*,z^*))
	-
	\frac{2}{\sigma}
	\bg_h((y_f^{k+1},z_f^{k+1}),(y^*,z^*)).
	\end{align*}
	After rearranging and using definition of $\Psi^k$ \eqref{scary:psi} we get
	\begin{align*}
		\Psi^{k+1}
		&\leq
		\sqN{\vect{x^k - x^*\\y^k - y^*\\z^k - z^*}}_\mP
		+
		(1-\rho)\frac{(2-\tau)}{\tau}\bg_r(x_f^k,x^*)
		\\&+
		(1-\rho)\frac{2}{\sigma}
		\bg_h((y_f^k,z_f^k),(y^*,z^*))
		\\&\leq
		\left(1-\frac{1}{1+\rho^{-1}}\right)\Psi^k.
	\end{align*}
\end{proof}

\end{document}